\documentclass{amsart}
\usepackage{amscd,amssymb,amsmath,amsthm}
\usepackage[dvips]{graphicx}
\usepackage[all]{xy}
\theoremstyle{plain}
\newtheorem{definition}{Definition}
\newtheorem{proposition}{Proposition}
\newtheorem{theorem}[proposition]{Theorem}

\newtheorem{lemma}[proposition]{Lemma}
\newtheorem{remark}[proposition]{Remark}
\newtheorem*{proposition*}{Proposition}
\newtheorem*{theorem*}{Theorem}
\newtheorem*{corollary*}{Corollary}
\newtheorem*{lemma*}{Lemma}
\newtheorem*{remark*}{Remark}
\newtheorem*{example*}{Example}

\newcommand{\Z}{\mathbb{Z}}
\newcommand{\Q}{\mathbb{Q}}
\newcommand{\R}{\mathbb{R}}
\newcommand{\C}{\mathbb{C}}

\begin{document}

\title{Point-Splitting Perturbative Chern-Simons I}
 
\author{Vito Iacovino}


\email{vito.iacovino@gmail.com}

\date{version: \today}


\begin{abstract}
We define mathematically the Point-Splitting Perturbative Chern-Simons Theory. 
The framing anomaly arise directly from the dependence of the configuration space of points on the frame. 
\end{abstract}

\maketitle

\section{Introduction}

Let $M$ be an oriented compact three manifold. 
The Perturbative Chern-Simons ($PCS$ for short) theory on $M$ was defined mathematically  by Axelrod and Singer more than thirty years ago  (\cite{AS1}, \cite{AS2})    . 
The $PCS$ partition function was defined in terms of an expansion in Feymann diagrams, where
 the contribution of a graph is obtained integrating a differential form (which depends on a metric of $M$) on the compactification of the configuration space of points associated to the vertices of the graph.
The framing anomaly arise because in order to cancel the dependence on the metric of $M$ it is necessary to introduce correction terms which depend on a frame of $M$.

In this paper we introduce the mathematical theory of  Point-Splitting-Perturbative-Chern-Simons ($PSPCS$ for short). The frame of $M$ is used directly to define (the analogous of) the configuration space of points. Making this definition mathematically precise is quite subtle. We introduce the Point-Splitting-Homology ($PSH$) which is the object playing the main role in our construction.

$PSPCS$ provides the natural ambient space to understand the Skein relations, a longstanding open problem in the context of standard $PCS$.  Indeed, in \cite{PSPCS2} we show that (a version of) $PSH$ is isomorphic to (a version of) the Skein.


The structures studied in this paper plays a fundamental role in the general Open Gromov-Witten theory (see \cite{MCH}, \cite{OGW-QME1}, \cite{OGW-QME2} ).
In degree zero the decorated graphs used in \cite{OGW-QME2}  reduces to the decorated graphs used in this paper and, accordingly, the Multi-Curve-Homology introduced in  \cite{OGW-QME2}  reduces  to the Pont-Splitting-Homology introduced in this paper.
The conditions defining Multi-Curve-Cycles are directly related to the constrains of the perturbation of the moduli space of (multi-)pseudo-holomorphic curves defined in \cite{MCH}. 

In this paper we consider the case of pure Chern-Simons, i.e. , without Wilson loops. In \cite{PSPCS2}  we extend the formalism to include Wilson loops. In \cite{PSPCS3} we consider the not compact case and the gluing procedure.

The starting point of the construction is as follows. 
In the standard approach to $PCS$ to any graph $G$ it is associated a (compactified) configuration space of points on $M$ which is a differential blow-up along the diagonals of the space of arrays $M^{V(G)}$ of points of $M$ labelled by the vertices of the graph (see \cite{AS2}). Denoting with $H(G)$ the set of half-edges of $G$, the space $M^{V(G)}$ may be identified with a subspace of $M^{H(G)}$ (this is the subspace where the points corresponding to the half-edges attached to each vertex coincide). 
In $PSPCS$ we consider a perturbation of this subspace which is transversal to the diagonals associated to the edges of the graph, that is, we \emph{split} the points associated to the half-edges attached to the same vertex. We are lead to consider a collection of singular chains labelled by the graphs. 
Of course this perturbation suffer of a very high indetermination. 
We constrains the admissible perturbations introducing the cycle condition and the forgetful compatibility condition which relate chains associated to different graphs. These conditions leads to the notion of Point-Splitting-Cycle.  A consequence of the cycle condition is that the integration of the Chern-Simons differential forms does not produces boundary terms (in contrast with the situation of the standard $PCS$ where these terms are responsible of the framing anomaly). Let us stress that in $PSPCS$ we do not use the compactified configuration spaces of points but we simply work with chains on $M^{H(G)}$.
In order to formulate the cycle condition it is necessary to consider graphs equipped with a suitable decoration: any graph is equipped with a set (called components), to each component it is associated a genus and a set of vertices. (The vertices may have attached an arbitrary number of half-edges, each component is required to satisfy a stability condition.)
There is a natural notion of isotopies of $PS$-cycles.
Taking isotopies classes of $PS$-cycles we obtain $PS$-homology.
The next step consists in showing that to each frame of $M$ we can associate canonically a $PSH$-class. 
Roughly this class may be characterized by requiring the invariance under the translation of the vertices. This condition is defined in a certain limit defined using the frame.  We make this condition precise introducing the notion of flat-approaching-$PS$-cycle. We show that, for any fixed frame, there is a unique flat-approaching-$PS$-cycle up to isotopies.

The paper is organized as follow.

In Section 2 we introduce $PS$-cycles and isotopies of $PS$-cycles. 
A $PS$-cycle associates to any decorated graph $G$ a singular chain on $M^{H(G)}$.
This collection of chains is required to satisfy the transversality condition, the cycle condition and the forget compatibility condition. To formulate the last condition we need to fix an Annulus Datum. 

In Section 3 we introduce $PS$-homology. We first define nice $PS$-cycles which may be thought as a sequence of isotopic $PS$-cycles shrinking in a certain sense.  
Point-Splitting-Homology $PSH(M | U)$ is defined by the isotopy classes of nice $PS$-cycles ($U$ denotes a smooth Euler Structure of $M$ related to the Annulus Datum).

In Section 4 we associate to each $PS$-cycle $Z$ its partition function $\mathfrak{P}(Z)$.  We integrate on $Z$ a collection of differential forms which are a generalization of the one appearing in \cite{CS}. 
The Stokes theorem yields the Quantum Master Equation for  $\mathfrak{P}(Z)$.  The cycle condition (for isotopies of) $PS$-cycles leads to the cancellations of the boundary terms.  
Compared to the standard approach here the analysis is simpler since it does not involve the vanishing lemmas.

In Section 5 we associate to any frame $\mathbf{fr}$ of $M$ a class $PSH$-class $[Z^{\mathbf{fr}, U}] \in PSH(M | U)$  (the Euler Structure $U$ is assumed compatible with $\mathbf{fr}$).  We characterize $[Z^{\mathbf{fr}, U}] $ introducing the notion of flat-approaching-$PS$-cycle. The frame $\mathbf{fr}$ is used in the formulation of the convergence condition of flat-approaching $PS$-cycle.  We refer to these classes as coherent $PSH$-classes. 

In  Section 6 we introduce the reduced-$PSH$, defined as a suitable quotient of $PSH$. Its computation takes a particular simple form.  We show that reduced $PSH$ is a rank one free modulo over the ring of formal power series on two variables $\Q[[\hbar,a]]$.  
A coherent $PSH$-class is a generator of $PS$-Homology as $\Q[[\hbar,a]]$-module.




\section{Point Splitting Cycles}

\subsection{Decorated Graphs}

A decorated graph  $G $ is defined by the data 
\begin{equation} \label{array-point-splitting}
   (    Comp_0 , (V_c,d_c,g_c)_c , (H_v )_v  , E )
\end{equation} 

where 
\begin{itemize}
\item $Comp_0$ is a finite set;
\item for each $c \in Comp_0 $, $V_c$ is a finite set and $d_c, g_c \in \Z_{\geq 0}$;

Set 
$$V(G)=  \bigsqcup_{c \in Comp_0} V_c ,  d(G)=  \sum_c d_c .$$ 

\item for each $v \in V$,  $H_v$ is a finite cyclically ordered set; 

Set 
$$H(G)= \bigsqcup_{v \in V(G)} H_v .$$
\item $E$ is a partition of $H$ in subset of cardinality two or one. The sets of cardinality two are called internal edges $E^{in}(G)$, the sets of cardinality one are called external edges $E^{ex}(G)$;
\end{itemize}
For each $c \in Comp_0$, set 
$$\chi_c= 2- 2g_c - |V_c|  - d_c.$$
We say that $G$ is stable if
$$ \chi_c - \frac{1}{2} |H_c| <0 \quad \forall c \in Comp_0.$$

There is an obvious notion of isomorphism of decorated graphs. 
Denote by $ \mathfrak{G}$ the set of isomorphism classes of stable decorated graphs.



Set 
$$ \chi(G) =  \sum_{c \in Comp_0} \chi_c - |E^{in}(G)| .$$
Denote by $\mathfrak{G}(\chi)$ the set of stable graphs $G$ with $\chi(G)= \chi$.


\subsubsection{Operation $\delta_v$, $v \in V(G)$}

 Let $G \in \mathfrak{G}$ and $v \in V(G)$ with $H_v = \emptyset$. Let  $c \in Comp_0(G)$ such that $v \in V_{c}$. The graph $\delta_{v} G$ is
  defined discarding $v$ from $V_{c}$ and replacing $d_{c}$ with $d_{c} +1$.
 

\subsubsection{Operation $\delta_e$,  $e \in E^{in}(G) $}

For $G \in \mathfrak{G}$  and $e = \{ h_1,h_2 \} \in E^{in}(G) $ the graph $\delta_e G$  is defined as follows:


\begin{itemize}
\item
Assume $h_1 \in H_{v_1}$ and $h_2 \in H_{v_2}$ with $v_1 , v_2 \in V(G), v_1 \neq v_2$. $V(\delta_e G)$ is obtained from $V(G)$ discarding $v_1,v_2$  and adding a  new vertex. Denote by $v_0$ the new vertex. 
Set $I_1= H_{v_1} \setminus \{ h_1 \},  I_2= H_{v_2} \setminus \{ h_2 \}.$  
The cyclic order of $H_{v_1 }$ and  $H_{v_2 }$ gives an ordering  of the sets  $I_1$ and $ I_2$.
$H_{v_0}$ is defined as 
$ H_{v_0} =  \{ I_1 , I_2  \}$
where the identity is considered as identity of cyclically ordered sets.

We have the following sub-cases:
\begin{itemize}
\item $v_1 \in V_{c_1}$, $v_2 \in V_{c_2}$ for  $c_1 , c_2 \in Comp_0(G)$ with $c_1 \neq c_2$: $Comp(\delta_e G)$ is obtained from $Comp(G)$ replacing the components $c_1$ and $c_2$ with a unique component. Denote by $c_0$ the new component. $V_{c_0}$ is obtained from $V_{c_1} \sqcup V_{c_2} $ replacing the vertices $v_1$ and $v_2$ with $v_0$. We set 
$d_{c_0}=d_{c_1} + d_{c_2}$, $  g_{c_0} = g_{c_1}+ g_{c_2}$.

\item $v_1,v_2 \in V_{c_0}$ for some $c_0  \in Comp_0(G)$:  $Comp_0(\delta_e G) =Comp_0(G)$,
$V_{c_0}$ is obtained replacing the vertices $v_1$ and $v_2$ with $v_0$, the genus $g_{c_0}$ is increased by one, 
 $d_{c_0}$ does not change.

\end{itemize}
\item
Assume $h_1,h_2 \in H_{v_0}$ for $v_0 \in V_{c_0}$ with  $c_0 \in Comp_0(G)$. $ V(\delta_e G)$ is obtained from $V(G)$ replacing $v_0$ with two new vertices $v_1, v_2$. Let  $I_1$ and $I_2$  be the unique ordered sets such that the identity  
$$H_{v_0 }= \{ h_1 , I_1 , h_2, I_2 \},$$ 
holds as identity of  cyclicaly ordered sets .  We set 
$$ H_{v_1} =  I_1, \text{    }  H_{v_2} =  I_2 ,$$
where the identity is considered as identity of cyclically ordered sets.
 $Comp_0(\delta_e G) = Comp_0(G)$, $V_{c_0}(\delta_e G)$ is obtained replacing $v_0$ with $v_1, v_2$,    $g_{c_0}$ and $d_{c_0}$ do not change. 

\end{itemize}

\subsubsection{Marked Graphs} \label{marking-section}
 
Define the set of  decorated graphs $\mathfrak{G}_l$  as the set of pairs $(G,m)$, where $G \in \mathfrak{G}$ and  
$$m= \{ E_0, E_1,..., E_l \}$$ 
with $E_0 \subset E_1 \subset ... \subset E_l \subset E(G)$ be an increasing finite sequence of subsets of $E(G)$ with $E^{ex}(G) \subseteq E_0$.
The definition the operator $\delta_e$  works on $\mathfrak{G}_l$ for $e \in E^{in}(G) \setminus E_l$.  

 
 Denote
 $$ \mathfrak{G}_* = \sqcup_l  \mathfrak{G}_l .$$

Let  $cut_{E_0}G$ be the graph obtained cutting the edges $E_0$. So $E^{in }(cut_{E_0}G)= E^{in}(G) \setminus E_0^{in}, H(cut_{E_0}G)= H(G)$.  
We consider the decorated graphs up to the equivalence relation generated by
$$ (G,m ) \sim ( cut_{E_0}G, cut_{E_0}m   ) $$ 
where 
$$cut_{E_0}m =   \{ E^{ex}(cut_{E_0}G), E_1,..., E_l \}.$$


\subsection{Point Splitting Cycles}

We now introduce the space of Point Splitting Cycles. The definition depends on the choice of  what we call an Annulus Data, which enter in the forgetful compatibility condition. 

In the following we use the notations for the chains on global orbit-fold of Subsection   \ref{orbitfold}. We also use chains on (multi)-loops as in Section \ref{chain-loop-section}.

\subsubsection{Annulus data} \label{annulus-data-section}



Let $Disk0$ be the graph  with exactly one component (which we denote by $c$) characterized as follows
$$Disk0: |V_c|=1, d_c= 0, |H_c| =3 ,  |E^{in}_c|=  0.$$

Define the chain $\underline{Z}_{Disk0    }   \in  C_*(\mathfrak{L}(M) , \mathfrak{o}(H(Disk0))  $  by the maps
$$w: \mathbb{S}^1 \times M \rightarrow M $$
$$ t_{v}: H_{v} \rightarrow   \mathbb{S}^1 $$
where, $w$ is the projection on the second factor, $v$ is the only vertex of $Disk0$, $t_{v}$ is any injective map between $H_v$ and $ \mathbb{S}^1 $ respecting the cyclic order. 
Up to triangulations, $\underline{Z}_{Disk0    }$ does not depend on the choice of  $t_{v}$.

Let $Ann0 = \{  G_0, G_1,G_2 \}$  be the set of  decorated graphs with exactly one component (which we denote by $c$) characterized as follows
\begin{itemize} \label{unstable-cases}
\item $G_0$: $|V_c|=1, d_c=0$,  $|H_c| =3 $,  $|E^{in}_c|=1$;  \label{disk-edge}
\item  $G_1$: $|V_c|=2, d_c=0$ , $|H_c| = 1 $;   \label{annelus} 
\item    $G_2$: $|V_c|=1, d_c=1$ ,$|H_c| = 1 $.   \label{annelus-deg}
\end{itemize}

An Annulus Data $\underline{Z}^{\circledcirc}$ consists in a collection of chains $(\underline{Z}_G^{\circledcirc})_{G \in Ann0}$ 
\begin{equation} \label{chains-loops-annulus}
\underline{Z}_{G}^{\circledcirc} \in  C_*(\mathfrak{L}_{\hat{V}(G)}(M) , \mathfrak{o}(H(G)))^{\text{Aut}(G)}  .
\end{equation}
such that
\begin{itemize}
\item \begin{equation} \label{cycle-condition-annulus}
 \partial ( \underline{Z}_{G_1} ^{\circledcirc} +   \underline{Z}_{G_2}^{\circledcirc}) =   \delta_{e} \underline{Z}_{G_0}^{\circledcirc};
 \end{equation}
 (see subsection \ref{operator-delta} for the definition of the operator $\delta_e$.) 
\item
$\underline{Z}^{\circledcirc}_{G_0} $ is close in the $C^0$ topology to $\underline{Z}_{Disk0}$. 
\end{itemize}


\subsubsection{Point Splitting Cycles} \label{PSC-subsection}

Fix an Annulus Data $\underline{Z}^{\circledcirc}$. 

A point splitting cycle  $Z$ with annulus data $\underline{Z}^{\circledcirc}$ consists on a collection of chains
\begin{equation} \label{collection-chains-definition}
( Z_{G,m} )_{(G,m) \in \mathfrak{G}_*}
\end{equation}
with
$${Z}_{G,m} \in C_{|H(G)| + |m|} ( M^{H(G)} , \mathfrak{o}_{H(G)})^{\text{Aut}(G,m)}  .$$
such  that 
\begin{itemize}
\item  $  Z_{G,m}= Z_{cut_{E_0}G,m}$ for each $(G,m)$.
\item  the transversality condition is satisfied (see  subsection \ref{transversality-section});
\item  the forgetful compatibility is satisfied  with annulus data $\underline{Z}^{\circledcirc}$ (see subsection  \ref{forget-PS-section});
\item the cycle condition is satisfied (see subsection \ref{cycle-condition-section});
\end{itemize}

We denote by $ \mathcal{Z}(M|  {Z}^{\circledcirc}) $ the vector space of $PS$-cycles with annulus data $\underline{Z}^{\circledcirc}$.




\subsubsection{Transversality Condition} \label{transversality-section}

We say that the transversality condition is satisfied for $(G,m)$ if $ Z_{G,m} $ is transversal  to   $ \sqcap_{e \in E'} \pi_e^{-1}(Diag)$ for each subset $E' \subset E^{in }(G )  \setminus E^l$;

\subsubsection{Forgetful Compatibility} \label{forget-PS-section}

Let $(G,m) \in \mathfrak{G}_l$ and $e \in E^{ex}(G)$. The graph $(G',m')= forget_e(G,m) \in \mathfrak{G}_l$  is obtained from $(G,m) $ removing the edge $e$.  

The definition of $(G',m')$   is straightforward in the case that $G$ is stable after removing $e$ .  

Assume that $G$ becomes unstable after removing $e$.
Let $v \in V(G)$ with $e \in H_v$ and let $c \in Comp(G)$ with $v \in V_c$. 
Let $G_e$ be the decorated graph defined by
$$Comp(G_e)= \{  c \}, \beta(G_e)= \beta_c ,  V(G_e)= V_c, g_{G_e}= g_c, d_{G_e}= d_c,$$
$$ H(G_e)=H_c, E^{in}(G_e) = \{ e \in E^{in}(G)| e \subset H_c \} .$$
The assumption on $e$ tells us that $G_e \in Ann0$ or $G_e=Disk0$.

We say that $e$ is not removable if $G_e = G_0$   and  $E^{in}_c \subset E_l$ . In all the other cases we say that $e$ is removable.


When $G_e = Disk0$ define $G'$ removing the component $c$ and gluing the two elements of $H_v \setminus \{ e \}$. More precisely let $H(G_e)= \{ h_1, h_2, e \}$.  If there exists $h_2' \in H(G)$ such that $h_1 \in E^{ex}(G)$, $\{ h_2, h_2' \} \in E^{in}(G)  $, declare $h_2' \in E^{ex}(G')$. 
If there exists $h_1', h_2' \in H(G)$ such that $\{ h_1, h_1' \} \in E^{in}(G)$, $\{ h_2, h_2' \} \in E^{in}(G)  $
set $\{ h_1', h_2' \} \in E^{in}(G')$.
If $\{ h_1, h_1' \} \in E_i,  \{ h_2, h_2' \} \in E(G) \setminus E_{i-1} $  set   $\{ h_1', h_2' \} \in E_i$. 

When  $G_e \neq Disk0$,  $G'$ is defined removing $G_e$ from $G$.
Note that in this case $G_e$ is a connected component of $G$. 



We say that the collection of chains  (\ref{collection-chains-definition}) satisfies the forget compatibility condition if there exists a collection of chains 
\begin{equation} \label{chains-loops}
\underline{Z}_{G,m} \in  C_{|H(G)| + |m|}(\mathfrak{L}_{\hat{V}(G)}(M) , \mathfrak{o}(H(G)))^{\text{Aut}(G,m)}  .
\end{equation}
such that the following happen:
\begin{itemize}
\item  for each $(G,m)$
\begin{equation} \label{chains-ev}
Z_{G,m}= \text{ev}_* ( \underline{Z}_{G,m}) 
\end{equation}
\item If $G$ is stable after removing $e $  we require
    \begin{equation} \label{stable-forget-condition}    
     \underline{Z}_{G,m} = \underline{Z}_{G',m'} \times_{\textbf{forget}_e} \mathfrak{L}_{\hat{V}(G)}( M)  .
     \end{equation}
      \item in the case $G$ is unstable after removing $e $ and $e$ is removable  we require
 \begin{equation} \label{removible-forget-condition}
 \underline{Z}_{G,m} = \underline{Z}_{G',m'} \times \underline{Z}_{G_e}^{\circledcirc}  
 \end{equation}     
 when $G_e \in Ann0$, and
 \begin{equation} \label{removible-forget-condition2}
 \underline{Z}_{G,m} = \underline{Z}_{G',m'} \times \underline{Z}_{Disk0}  
 \end{equation}  
 when $G_e= Disk0$.
      \item In the case  $e$ is not removable we require that $ \underline{Z}_{G,m}$ is close in the $C^1$-topology to $ \underline{Z}_{G',m'} \times \underline{Z}_{Disk0}^{}$.  We write this as
 \begin{equation} \label{unstable-forget-condition}
 \underline{Z}_{G,m} \approx \underline{Z}_{G',m'} \times \underline{Z}_{Disk0}^{} .
 \end{equation}
  \end{itemize}

\begin{remark}
The collection of chains (\ref{chains-loops}) satisfying the above conditions, if it exists, is uniquely determined by the collection of chains (\ref{collection-chains-definition}). 
\end{remark}

\subsubsection{Cycle Condition} \label{cycle-condition-section}




 

Let  $v \in V(G)$ with $H_{v} = \emptyset$. Let $c \in Comp_0$ with $v \in V_c$.
Denote with $\delta_v G$ the decorated graph obtained from $G$ discarding the vertex $v$ and increasing $d_c$ by one.


According to our definition of chain on loop spaces provided in section \ref{chain-loop-section}, we have special boundary faces  corresponding to the face  $N_{0}$ in the generator  (\ref{delta-loop0}) , 
Becouse of these special boundary faces  $\partial   {\underline{Z}}_{G,m} $ does not define an element of  $C_*(  \mathfrak{L}_{\hat{V}(G)}(M) )$. 
Denote by $\partial^{red}  {\underline{Z}}_{G,m} \in C_*(  \mathfrak{L}_{\hat{V}(G)}(M) )$ the chain obtained discarding from $\partial   {\underline{Z}}_{G,m} $ the special boundary faces.

The projection  $  \mathfrak{L}_{\hat{V}(G)}(M)  \rightarrow  \mathfrak{L}_{\hat{V}(\delta_v G)}(M) $ applied to the special boundary faces  $N_{0,v}$ yields an element of  $ C_*(  \mathfrak{L}_{ \hat{V}{\delta_v G}}(M) )$. Denote by $ \delta_v {\underline{Z}}_{G,m} \in C_*(  \mathfrak{L}_{ \hat{V}(\delta_v G})(M) )$ the chain obtained from $ {\underline{Z}}_{G,m}$ in this way. 

We say that the cycle condition is satisfied for $(G,m) \in \mathfrak{G}_*$  if
\begin{equation} \label{cycle-condition-loop}
 \partial^{red}  {\underline{Z}}_{G,m}  =   \sum_{0 \leq i \leq |m|} (-1)^i {\underline{Z}}_{G, \partial_i m} + \sum_{(G',m', e')}  \delta_{e'} {\underline{Z}}_{G',m'}  +  \sum_{(G',m',v')} \delta_{v'} {\underline{Z}}_{G',m'} ,
  \end{equation}
where the second sum is over the triples  $(G',m', e')$ with $\delta_{e'}(G',m')=(G,m)$ and the last sum is over the triples  $(G',m', v')$ with $\delta_{v'}(G',m')=(G,m)$.

 We say that $ {\underline{Z}}$ satisfies the cycle condition if the cycle condition is satisfied for each  $(G,m) \in \mathfrak{G}_*$. 
 


\subsubsection{Operator $\delta_e$} \label{operator-delta}

Let us first fix the conventions about orientations. Let $N, X$ be oriented manifold. 
Let $f: N \rightarrow X$ be a smooth map and $Y$ be a submanifold of $X$. Assume that $f$ is transversal to $Y$, i.e,  $df_p : T_pN \rightarrow T_{f( p  )}X/T_{f(  p  )}Y$ is surjective for each $p \in f^{-1}(Y)$. The orientation of $f^{-1}(Y)$ is defined according to the (not canonical) isomorphism 
 \begin{equation} \label{orientation-intersection}
 TN \cong TX/TY \oplus T(f^{-1}(Y)) . 
 \end{equation}

We need to compare the orientation of $\partial (f^{-1}(Y)) $ with  $f_{\partial}^{-1}(Y)$, 
 where $f_{\partial}$ is the restriction of $f$ to $\partial N$. We assume that the codimension of $Y$ on $X$ is odd.

By definition
$$T (\partial N)=  TX/TY \oplus T(f_{\partial}^{-1}(Y)) .$$
and so
$$ TN= \R \oplus T (\partial N) =\R \oplus   TX/TY \oplus T(f_{\partial}^{-1}(Y))= - TX/TY \oplus  \R \oplus T(f_{\partial}^{-1}(Y)).$$
Comparing with (\ref{orientation-intersection}) we obtain $T( f^{-1}(Y))  = -   \R \oplus T(f_{\partial}^{-1}(Y))$, that is,
\begin{equation} \label{comparing-orientation}
T( \partial (f^{-1}(Y)))   = -    T(f_{\partial}^{-1}(Y)).
  \end{equation}

Now take $G \in \mathfrak{G}$ and $e \in E^{in}(G)$. 
For a chain $C \in C_*(M^{H(G)}, \mathfrak{o}_{H(G)})^{\text{Aut}((G,m)) }$ we define $\delta_e C= C \cap Diag_e$ as follows. We need to be careful about the orientation and the automorphism groups.



Denote by $\text{Aut}((G,m),e)$ the group of automorphims of $(G,m)$  fixing the edge $e$.
Consider the chain $C \cap Diag_e$  as an element of  $C_*(M^{H(G)}, \mathfrak{o}_{H(G)})^{\text{Aut}((G,m),e) }$.


The orientation of $C \cap  Diag_e $ is defined on each simplex $(N,f)$ according to (\ref{orientation-intersection}) where $X=M^{H(G)}, Y=Diag_e$.


There is an homomorphism of groups  $\text{Aut}((G,m),e) \rightarrow \text{Aut}((\delta_e G,m))$ which, together the projection $M^{H(G)} \rightarrow M^{H(\delta_e(G))}$, induces a map of global orbitfolds
$$pr: M^{H(G)}/\text{Aut}((G,m),e) \rightarrow M^{H(\delta_e G)}/ \text{Aut}((\delta_e G,m)).$$

Set
 $$\delta_e C = pr_*( C \cap Diag_e ).$$
 where we use the notations of (\ref{maps-orbitfold}) .
 


From  (\ref{comparing-orientation}) we obtain
$$ \delta_e \partial = - \partial  \delta_e . $$

In order to extend $\delta_e$ to chain on loops (\ref{chains-loops}) we observe that, the forgetful compatibility property implies that there exists a unique collection of chains
\begin{equation} \label{chains-loops-edge}
\{ \delta_e \underline{Z}_{G,m} \}_{G,m,e}  
\end{equation}
$$\delta_e \underline{Z}_{G,m} \in  C_*(  \mathfrak{L}_{\hat{V}{\delta_e G}}(M) )$$
where $e \in E(G) \setminus E_l$, 
such that 
\begin{equation} \label{chains-delta}
\delta_e {Z}_{G,m}= \text{ev} ( \delta_e \underline{Z}_{G,m}) 
\end{equation}
for each $(G,m,e)$.
Relation (\ref{chains-delta}) defines the operator $\delta$ on the chains of loops (\ref{chains-loops}). This operator can be seen as the higher genus generalization of the topological string bracket.

\subsubsection{Chains on Global Orbitfolds}   \label{orbitfold}

Let $X$ be a manifold, and let $\Gamma$ be a finite group acting on $X$. If $\mathfrak{o}$ is a fined set on which $\Gamma$ acts,  
 $\mathfrak{o}$ defines  a  local system on the global orbit-fold $X/\Gamma$.  The chains on $X/\Gamma$ with coefficients on  $\mathfrak{o}$ may be identified with the $\Gamma$-invariant chains on $X$:
$$C_*(X/\Gamma, \mathfrak{o}) = C_*(X , \mathfrak{o})^{\Gamma}.$$
Let $h: \Gamma_1 \rightarrow \Gamma_2$ be an homomorphism of groups. An $h$-equivariant   smooth map  $f : (X_1, \mathfrak{o_1} ) \rightarrow (X_2, \mathfrak{o_2})$ defines  a map of orbit-folds $(X_1, \mathfrak{o_1} )   /\Gamma_1 \rightarrow (X_2, \mathfrak{o_2} )/\Gamma_2$. The induced map on the chains is given by 
\begin{equation} \label{maps-orbitfold}
  C_*(X_1, \mathfrak{o_1})^{\Gamma_1} \rightarrow  C_*(X_2, \mathfrak{o_2})^{\Gamma_2} 
\end{equation}
$$ C  \mapsto \frac{1}{|\Gamma_1|} \sum_{g_2 \in \Gamma_2} (g_2)_*(f_*(C)).$$

For finite  set $S$,  denote by $\mathfrak{o}_S$ the set of ordering of $S$ up to parity: 
$$\mathfrak{o}(S) = \frac{\text{ordering of }S}{\text{even permutations of  }  S}.$$

 The main case  we are interested in this paper is the following:  $X=M^{H(G)}$, $\Gamma= \text{Aut}(G,m)$ and $\mathfrak{o}= \mathfrak{o}_{H(G)}$  is the set of ordering of $H(G)$ up to parity, where  $(G,m)$ is a decorated graph. 

\subsection{Isotopies} \label{isotopies-section}

Let $\tilde{Z}^{\circledcirc}$ be an isotopy between annulus data ${Z}^{\circledcirc,0}$ and  ${Z}^{\circledcirc,1}$. Let $Z^0 \in \mathcal{Z}(M|  {Z}^{\circledcirc,0})$ and $Z^1 \in \mathcal{Z}(M|  {Z}^{\circledcirc,1})$. An isotopy  $\tilde{Z} \in \tilde{\mathcal{Z}}$ of $PS$-cycles  between $Z^0$ and $Z^1$  consists in a collection of Borel-Moore chains
\begin{equation} \label{collection-chains-isotopy}
\tilde{Z}=  ( \tilde{Z}_{G,m} )_{G,m} ,
\end{equation}
with
$$\tilde{Z}_{G,m} \in C_{|H(G)| + |m|+1} ( M^{H(G)} \times \R  , \mathfrak{o}_{H(G)})^{\text{Aut}(G,m)}  .$$
We require that
\begin{itemize}
\item there exists  $\lambda^{\dagger} >0$  big enough such that
$$ \tilde{Z}^{<-\lambda^{\dagger}} =  Z^0  \times  \R_{<-\lambda^{\dagger}}  , \quad \tilde{Z}^{>\lambda^{\dagger}} =   Z^1  \times \R_{>\lambda^{\dagger}}   ;$$

 \item 
the transversality condition is satisfied;


\item the forgetful compatibility condition is satisfied.

\item the cycle condition is satisfied.


\end{itemize}

We denote the space of isotopies of $PS$-cycles with $ \tilde{\mathcal{Z}}(M| \tilde{Z}^{\circledcirc}) $.

\subsubsection{Transversality Condition}

We say that the transversality condition is satisfied for $(G,m)$ if  $\tilde{Z}_{G,m} $ is transversal  to   
$ (\sqcap_{e \in E'}  \pi_e^{-1}(Diag )) \times \R$  for each subset $E' \subset E^{in }(G)  \setminus E^l$.

\subsubsection{Forgetful Compatibility} \label{forget-PS-section-iso}



We say that the collection of chains  (\ref{collection-chains-definition}) satisfies the forget compatibility condition if there exists a collection of chains 
\begin{equation} \label{chains-loops-iso}
\underline{\tilde{Z}}_{G,m} \in  C_*(\mathfrak{L}_{\hat{V}(G)}(M) \times \R , \mathfrak{o}(H(G)))^{\text{Aut}(G,m)}  .
\end{equation}
such that the following happen:
\begin{itemize}
\item  for each $(G,m)$
\begin{equation} \label{chains-ev-iso}
\tilde{Z}_{G,m}= \text{ev}_* ( \underline{\tilde{Z}}_{G,m}) 
\end{equation}
\item If $G$ is stable after removing $e $  we require
    \begin{equation} \label{stable-forget-condition-iso}    
     \underline{\tilde{Z}}_{G,m} = \underline{\tilde{Z}}_{G',m'} \times_{\textbf{forget}_e} \mathfrak{L}_{\hat{V}(G)}( M)  .
     \end{equation}
      \item in the case $G$ is unstable after removing $e $ and $e$ is removable  we require
 \begin{equation} \label{removible-forget-condition-iso}
 \underline{\tilde{Z}}_{G,m} = \underline{\tilde{Z}}_{G',m'} \times_{\R} \underline{\tilde{Z}}_{G_e}^{\circledcirc} 
 \end{equation}     
 if $G_e \in Ann0$, and
 \begin{equation} \label{removible-forget-condition2-iso}
 \underline{\tilde{Z}}_{G,m} = \underline{\tilde{Z}}_{G',m'} \times\underline{{Z}}_{Disk0}
 \end{equation} 
 if $G_e = Disk0$;
     \item In the case  $e$ is not removable we require that $ \underline{\tilde{Z}}_{G,m}$ is close in the $C^1$-topology to $ \underline{\tilde{Z}}_{G',m'} \times \underline{Z}_{Disk0}^{\circledcirc}$:
 \begin{equation} \label{unstable-forget-condition-iso}
 \underline{\tilde{Z}}_{G,m} \approx \underline{\tilde{Z}}_{G',m'} \times \underline{Z}_{Disk0}^{\circledcirc} .
 \end{equation}
  \end{itemize}

\subsubsection{Cycle Condition} \label{cycle-condition-section-iso}




Define $\partial^{red}   \tilde{\underline{Z}}_{G,m} $ analogously to subsection \ref{cycle-condition-section} discarding from  $\partial   \tilde{\underline{Z}}_{G,m} $ the contribution coming from  the boundary faces coming from  $N_{0,v}$  appearing in  (\ref{delta-loop0}). For  $v \in V(G)$ with $H_{v} = \emptyset$, $ \delta_v \tilde{\underline{Z}}_{G,m} $   is defined similarly to subsection \ref{cycle-condition-section}.


We say that the cyclic condition is satisfied for $(G,m) \in \mathfrak{G}_*$  if
\begin{equation} \label{cycle-condition-loop-isotopy}
 \partial  \tilde{\underline{Z}}_{G,m}  =   \sum_i (-1)^i \tilde{\underline{Z}}_{G, \partial_i m} + \sum_{(G',m', e')}  \delta_{e'} \tilde{\underline{Z}}_{G',m'}  +  \sum_{(G',m',v')} \delta_{v'} \tilde{\underline{Z}}_{G',m'} .
 \end{equation}



\subsubsection{Isotopy of Isotopies}

The definition of isotopy of isotopies of $PS$-cycles  is analogous to the definition of isotopies where we consider a two-parameters  families of chains instead of one-parameter families.

For simplicity  of notations below we omit  the dependence on the annulus data. We fix an isotopy of isotopies of annulus data.

\begin{definition} 
Let $ Z^{0, 0},   Z^{1, 0}  ,    Z^{ 0, 1}  ,    Z^{ 1 , 1} $ be $PS$-cycles. Let $ Z^{0, \bullet}$ be an isotopy between $Z^{0, 0}$ and $  Z^{ 0, 1}$,  $ Z^{1, \bullet}$ isotopy between $Z^{1, 0}$ and $  Z^{ 1, 1}$, $ Z^{ \bullet,0}$ isotopy between $Z^{0, 0}$ and $  Z^{ 1, 0}$,  $ Z^{ \bullet,1}$ isotopy between $Z^{0, 1}$ and $  Z^{ 1, 1}$.

An isotopy of isotopies of $PS$-cycles between $ Z^{0, \bullet},   Z^{1, \bullet}  ,    Z^{ \bullet, 0}  ,    Z^{ \bullet , 1} $ consists in a collection 
of Borel-Moore chains
\begin{equation} \label{collection-chains-isotopy-isotopy}
\tilde{ \tilde{Z}}=  (\tilde{\tilde{Z}}_{G,m} )_{G,m} ,
\end{equation}
with
$$\tilde{\tilde{Z}}_{G,m} \in C_{|H(G)| + |m|+2} (M^{H(G)} \times \R_{\lambda} \times \R_{\lambda'}  , \mathfrak{o}_{H(G)})^{\text{Aut}(G,m)}  .$$
such that
\begin{itemize}


\item
There exists  $\lambda^{\dagger} \gg 0$ big enough such that:
\begin{itemize}
\item $\tilde{\tilde{Z}}^{\lambda' <-\lambda^{\dagger} } =  Z_{ G,m}^{0, \bullet} \times \R_{\lambda'<-\lambda^{\dagger}}$;
\item $\tilde{\tilde{Z}}^{\lambda' >\lambda^{\dagger} } =  Z_{ G,m}^{1, \bullet} \times \R_{\lambda'>\lambda^{\dagger}} $;
\item  $\tilde{\tilde{Z}} ^{\lambda <-\lambda^{\dagger} } =  Z_{ G, m}^{ \bullet, 0 } \times \R_{\lambda < -\lambda^{\dagger}}$;
\item $\tilde{\tilde{Z}}^{ \lambda >\lambda^{\dagger} } =  Z_{ G,m}^{ \bullet, 1} \times  \R_{\lambda >\lambda^{\dagger}}$  .
\end{itemize}


\item  the transversality condition is satisfied;
\item forgetful compatibility condition is satisfied;

\item the cycle condition is satisfied.


\end{itemize}

We denote by $  \tilde{\tilde{ \mathcal{Z}}}$ the vector space if isotopies of isotopies.

\end{definition}

\subsubsection{Extension Lemmas}

 In order to construct $PS$-cycles, we shall often use inductive argument on the set of decorated graphs. 
 



We  consider the partial order on $\mathfrak{G}$ generated by the following relations:  
$G'\prec G$ if
\begin{itemize}
\item  $\delta_{e'}G' = G$ for some $e' \in E^{in} (G')$;
\item  $\delta_{v'}G' = G$, for some $v' \in V(G')$ with $H_{v'}= \emptyset$;
\item $forget_e G=G'$ for some $e \in E^{ex} (G)$.
\end{itemize}

This partial order admits a natural extension to  $\mathfrak{G}_*$:
$(G',m') \prec (G,m)$ if
\begin{itemize}
\item  $\delta_{e'}(G',m' ) = (G,m)$ for some $e' \in E^{in} (G')$;
\item  $\delta_{v'}(G',m' ) = (G,m)$, for some $v' \in V(G')$ with $H_{v'}= \emptyset$;
\item $(G',m')= (G, \partial_i m)$ for some $i$ with $0 \leq i \leq |m|$;
\item $forget_e(G,m)=(G',m')$ for some $e \in E^{ex} (G)$.
\end{itemize}

We denote by $\mathcal{Z}_{\preccurlyeq G_0}$ the set of  collection of chains (\ref{chains-loops}) truncated to $G \preccurlyeq G_0$ satisfying the  
the same conditions as before for each   $G \preccurlyeq G_0$.
Observe that it make sense to consider the conditions on the truncation.

Similarly we denote by $\mathcal{Z}_{\prec G_0}$   the truncation of  the collection of chains (\ref{chains-loops}) to the set of graphs with $G \prec G_0$.




\begin{lemma} \label{not-removible-lemma}

Let $G \in \mathfrak{G}$ with $ E^{ex}(G) \neq \emptyset$.    For any $Z_{\prec G} \in \mathcal{Z}_{\prec G}$ there exists  $Z_{\preccurlyeq G}  \in \mathcal{Z}_{\preccurlyeq G}$ extending $Z_{\prec G}$.

\end{lemma}
\begin{proof}

Fix $l \in \Z_{\geq 0}$, let $(G,m ) \in \mathfrak{G}_l$ and assume by induction we have defined $\underline{Z}_{G, m'}$ for each $m'$ with $|m'| <l$.

If there is  $e \in  E^{ex}(G)$ removable use   (\ref{stable-forget-condition}) ,  (\ref{removible-forget-condition}) or   (\ref{removible-forget-condition2}) to define $\underline{Z}_{G, m}$. 

Assume that all the external edges are not removable.
Take  $e \in  E^{ex}(G)$. 
Set
$$ \underline{Z}_{G,  m}^{prel}=  \underline{Z}_{ forget_e (G,m)} \times \underline{Z}_{Disk0}.$$

Using the cycle assumption for $ forget_e (G,m) $ and condition (\ref{unstable-forget-condition}) for $(G',m')$ with $(G',m') \prec (G,m)$ we obtain 
$ \partial  \underline{Z}_{G,  m}^{prel}$ is close in the $C^0$ topology to $\sum_i (-1)^i \underline{Z}_{G, \partial_i m } + \sum_{G',e'}  \delta_{e'} \underline{Z}_{G',m' }  $.
From a standard transversality argument  we may define
 $ \underline{Z}_{G,  m }$ close in the $C^0$-topology to $  \underline{Z}_{G,  m}^{prel}$ such that the transversality condition and 
 the cycle condition hold for $(G,m)$.

\end{proof}
 
\begin{lemma} \label{nice-extension}

Let $G \in \mathfrak{G}$ with $H(G) \neq \emptyset$. For any $Z_{\prec G} \in \mathcal{Z}_{\prec G}$ there exists  $Z_{\preccurlyeq G}  \in \mathcal{Z}_{\preccurlyeq G}$ extending $Z_{\prec G}$. 

\end{lemma}

\begin{proof}

When  $ E^{ex}(G) \neq \emptyset$ use  Lemma \ref{not-removible-lemma}.

Assume  $ E^{ex}(G) = \emptyset$.
Let $l \in \Z_{\geq 0}$, and assume that we have defined $\underline{Z}_{G, m}$ and $\underline{Z}_{G,m \sqcup \{ E(G)\}}$ for each $|m| <l$ with $E(G) \notin m$.

Take $m$ with $|m|=l$ and  $E(G) \notin m$.

Set 
$$  \underline{Z}_{G,  m  }^{prel} =  (-1)^l(   \sum_i (-1)^i \underline{Z}_{G, \partial_i m \sqcup \{ E(G) \}} + \sum_{(G',e')}  \delta_{e'} \underline{Z}_{G', m' \sqcup \{ E(G)\}} +  \sum_{(G',v')}  \delta_{v'} \underline{Z}_{G', m' \sqcup \{ E(G)\}} .$$
The cyclic condition for $(G, \partial_i m \sqcup \{ E(G) \})$ and $(G', m' \sqcup \{ E(G)\})$ implies
$$\partial^{red} \underline{Z}_{G,  m  }^{prel} = \sum_i (-1)^i \underline{Z}_{G, \partial_i m } +\sum_{(G',e')}  \delta_{e'} \underline{Z}_{G', m' } +  \sum_{(G',v')}  \delta_{v'} \underline{Z}_{G', m'}.$$

From the standard transversality argument there exists $\underline{Z}_{G,  m  }$ close in the $C^0$-topology to $\underline{Z}_{G,  m  }^{prel}$ such that   the transversality condition is satisfied and $\partial^{red} \underline{Z}_{G,  m  } = \partial^{red} \underline{Z}_{G,  m  }^{prel}$.

The transversality argument involves  a finite dimensional family of elements of $C_*(\mathfrak{L}_{\hat{V}(G)}(M))$ such that the evaluation on the punctures labelled by $H(G) $  is submersive in the family. The argument produces a chain $\underline{Z}_{G,m}$ cobordant to   $\underline{Z}_{G,  m  }^{prel}$. We may assume that  $\partial^{red} \underline{Z}_{G,  m  }= \partial^{red} \underline{Z}_{G,  m  }^{prel}$ and that the associated restriction of the cobordism is trivial.
We denote by $\underline{Z}_{G,m \sqcup \{ E(G)\}}^{prel}$ the chain associated to this cobordism multiplied by a sign $(-1)^{l+1}$. We have 
$$\partial \underline{Z}_{G,  m \sqcup \{ E(G) \}}^{prel}  = (-1)^{l+1}( \underline{Z}_{G,  m  } -   \underline{Z}_{G,  m  }^{prel}).$$

From a standard transversality argument we may  define  $\underline{Z}_{G,m \sqcup \{ E(G)\}}$ such that  the transversality condition is satisfied
and $\partial^{red} \underline{Z}_{G,m \sqcup \{ E(G)\}}=\partial^{red} \underline{Z}_{G,m \sqcup \{ E(G)\}}^{prel}  $.


\end{proof}


\begin{lemma} \label{not-removible-lemma-isotopies}


Let $Z_{\preccurlyeq G}^0, Z_{\preccurlyeq G}^1 \in  \mathcal{Z}_{\preccurlyeq G} $ and  let  $\tilde{Z} _{\prec G} \in  \tilde{ \mathcal{Z}}_{\prec G} $ be an isotopy between
 $Z_{\prec G}^0$ and $ Z_{\prec G}^1$. 
 There exists  
 $\tilde{Z} _{\preccurlyeq G} \in  \tilde{ \mathcal{Z}}_{\preccurlyeq G} $ isotopy between 
 $Z_{\preccurlyeq G}^0$ and $Z_{\preccurlyeq G}^1 $ extending $\tilde{Z} _{\prec G} $.  

\end{lemma}
\begin{proof}


Let $(G,m ) \in \mathfrak{G}_l$ and assume  we have defined $\tilde{\underline{Z}}_{G, m'}$ for each $m'$ with $|m'| <l$.

If there is  $e \in  E^{ex}(G)$ removable use (\ref{removible-forget-condition}) and   (\ref{stable-forget-condition})  to define $\tilde{\underline{Z}}_{G, m}$. 

Assume that each external edge is not removable.
Take  $e \in  E^{ex}(G)$. 
Set 
$$  \tilde{\underline{Z}}_{G,  m}^{prel}=   \tilde{\underline{Z}}_{ forget_e (G,m) } \times  \tilde{\underline{Z}}_{Disk0}.$$

By assumptions we have $\tilde{\underline{Z}}_{G,  m}^{prel, <-\lambda^{\dagger}} \approx {\underline{Z}}_{G,  m}^0 \times \R_{<-\lambda^{\dagger}}$ and  $\tilde{\underline{Z}}_{G,  m}^{prel, >\lambda^{\dagger}} \approx {\underline{Z}}_{G,  m}^1 \times \R_{>\lambda^{\dagger}}$, for $T \gg 0$.

The cycle property for $forget_e (G,m)$ and the inductive assumption (\ref{unstable-forget-condition}) imply
$$ \partial   \tilde{\underline{Z}}_{G,  m}^{prel} \approx \sum_i (-1)^i  \tilde{\underline{Z}}_{G, \partial_i m } + \sum_{G',e'}  \delta_{e'}  \tilde{\underline{Z}}_{G',m' }  .$$

From a standard transversality argument  we may define
 $  \tilde{\underline{Z}}_{G,  m }$ close in the $C^0$-topology to $   \tilde{\underline{Z}}_{G,  m}^{prel}$ such that the transversality and cycle condition hold
  and $\tilde{\underline{Z}}_{G,  m}^{ <-\lambda^{\dagger}} = {\underline{Z}}_{G,  m}^0 \times \R_{<-\lambda^{\dagger}}, \tilde{\underline{Z}}_{G,  m}^{ >\lambda^{\dagger}} = {\underline{Z}}_{G,  m}^1 \times \R_{>\lambda^{\dagger}}$, for $T \gg 0$.

\end{proof} 
 
  \begin{lemma} \label{nice-extension-isotopy}
 Assume $H(G) \neq \emptyset$.
 
 Let $Z_{\preccurlyeq G}^0, Z_{\preccurlyeq G}^1 \in  \mathcal{Z}_{\preccurlyeq G} $ and  let  $\tilde{Z} _{\prec G} \in  \tilde{ \mathcal{Z}}_{\prec G} $ be an isotopy between
 $Z_{\prec G}^0$ and $ Z_{\prec G}^1$. 
 There exists  
 $\tilde{Z} _{\preccurlyeq G} \in  \tilde{ \mathcal{Z}}_{\preccurlyeq G} $ isotopy between 
 $Z_{\preccurlyeq G}^0$ and $Z_{\preccurlyeq G}^1 $ extending $\tilde{Z} _{\prec G} $.     
 \end{lemma}
 
\begin{proof}

When  $ E^{ex}(G) \neq \emptyset$ use  Lemma \ref{not-removible-lemma-isotopies}.

Assume  $ E^{ex}(G) = \emptyset$.
Let $(G,m ) \in \mathfrak{G}_l$ and assume  we have defined $\tilde{\underline{Z}}_{G, m'}$ for each $m'$ with $|m'| <l$.
 and assume that we have defined $\tilde{\underline{Z}}_{G, m}$ and $\tilde{\underline{Z}}_{G,m \sqcup \{ E(G)\}}$ for each $|m| <l$.


Take a chain $\tilde{\underline{Z}}_{G,m \sqcup \{ E(G) \} }^{\dagger}$ which agrees with $ \underline{Z}_{G,m \sqcup \{ E(G) \} }^0 \times \R_{<-\lambda^{\dagger}}$ and  $ \underline{Z}_{G,m \sqcup \{ E(G) \} }^1 \times \R_{>\lambda^{\dagger}}$  for $T \gg 0$.

Set 
$$  \tilde{\underline{Z}}_{G,  m  }^{prel} =  (-1)^l(   \sum_i (-1)^i \tilde{\underline{Z}}_{G, \partial_i m \sqcup \{ E(G) \}} + \sum_{(G',e')}  \delta_{e'} \tilde{\underline{Z}}_{G', m' \sqcup \{ E(G)\}} + \sum_{(G',v')}  \delta_{v'} \tilde{\underline{Z}}_{G', m' \sqcup \{ E(G)\}}- \partial^{red} \tilde{\underline{Z}}_{G,m \sqcup \{ E(G) \} }^{\dagger}) .$$
The inductive assumptions imply
$$\partial^{red} \tilde{\underline{Z}}_{G,  m  }^{prel} = \sum_i (-1)^i \tilde{\underline{Z}}_{G, \partial_i m } + \sum_{G', e'}  \delta_{e'} \tilde{\underline{Z}}_{G',m' } + \sum_{(G',v')}  \delta_{v'} \tilde{\underline{Z}}_{G', m' },$$
 $\tilde{\underline{Z}}_{G,  m}^{ prel, <-\lambda^{\dagger}} = {\underline{Z}}_{G,  m}^0 \times \R_{<-\lambda^{\dagger}}, \tilde{\underline{Z}}_{G,  m}^{prel,  >\lambda^{\dagger}} = {\underline{Z}}_{G,  m}^1 \times \R_{>\lambda^{\dagger}}$, for $T \gg 0$.

From a standard transversality argument there exists $\tilde{\underline{Z}}_{G,  m  }$ close in the $C^0$-topology to $\tilde{\underline{Z}}_{G,  m  }^{prel}$ such that  the transversality condition is satisfied , 
$\partial^{red} \tilde{\underline{Z}}_{G,  m  } = \partial^{red} \tilde{\underline{Z}}_{G,  m  }^{prel}$ and
  $\tilde{\underline{Z}}_{G,  m}^{  <-\lambda^{\dagger}} = {\underline{Z}}_{G,  m}^0 \times \R_{<-\lambda^{\dagger}}, \tilde{\underline{Z}}_{G,  m}^{ >\lambda^{\dagger}} = {\underline{Z}}_{G,  m}^1 \times \R_{>\lambda^{\dagger}}$, for $T \gg 0$.

Now let   $ \tilde{\underline{Z}}^{cob}_{G,m \sqcup \{ E(G)\}  }$ be a chain obtained from a small cobordism between  $ \tilde{\underline{Z}}^{}_{G,m}$ and $  {\underline{Z}}_{G,  m }^{ prel}$. We assume the cobordism trivial between   $\partial^{red} \underline{Z}_{G,  m  }$ and $ \partial^{red} \underline{Z}_{G,  m  }^{prel}$, and in the regions $\mathfrak{L}_{\hat{V}(G)}(M) \times  \R_{<T}, \mathfrak{L}_{\hat{V}(G)}(M) \times \R_{>\lambda^{\dagger}}$.
 Set
$$ \tilde{\underline{Z}}^{ prel}_{G,m \sqcup \{ E(G)\} } = (-1)^l  \tilde{\underline{Z}}^{cob}_{G,m \sqcup \{ E(G)\}  } + { \tilde{\underline{Z}}}_{G,m \sqcup \{ E(G)\}}^{\dagger}.$$

We have
$$\partial^{red}  \tilde{\underline{Z}}_{G,  m \sqcup \{ E(G)\}  }^{ prel} = (-1)^l (   \tilde{\underline{Z}}_{G,  m  }^{prel }  -  \tilde{\underline{Z}}_{G,  m}^{} )+ \partial^{red} { \tilde{\underline{Z}}}_{G,m \sqcup \{ E(G)\}}^{\dagger}  ,$$
$$\tilde{\underline{Z}}_{G,m \sqcup \{ E(G)\} }^{ prel,< -T} = \underline{Z}_{G,m \sqcup \{ E(G)\} }^{0} \times \R_{<-\lambda^{\dagger}}, \quad \tilde{\underline{Z}}_{G,m \sqcup \{ E(G)\} }^{prel, > T} =  \underline{Z}_{G,m \sqcup \{ E(G)\}}^{1} \times \R_{>\lambda^{\dagger}}$$  
for $T \gg 0$. 





From q standard transversality argument  we may define $\tilde{\underline{Z}}_{G,  m   \sqcup \{ E(G)\}}$ close in the $C^1$-topology to $\tilde{\underline{Z}}_{G,  m  \sqcup \{ E(G)\} }^{prel}$ such that  the transversality condition is satisfied,
$ \partial^{red} \tilde{\underline{Z}}_{G,m \sqcup \{ E(G)\}} = \partial^{red}  \tilde{\underline{Z}}_{G,m \sqcup \{ E(G)\}}^{prel}$ ,  
$\tilde{\underline{Z}}_{G,  m   \sqcup \{ E(G)\}}^{  <-\lambda^{\dagger}} = {\underline{Z}}_{G,  m   \sqcup \{ E(G)\}}^0 \times \R_{<-\lambda^{\dagger}}, \tilde{\underline{Z}}_{G,  m   \sqcup \{ E(G)\}}^{ >\lambda^{\dagger}} =\tilde{\underline{Z}}_{G,  m   \sqcup \{ E(G)\}}^1 \times \R_{>\lambda^{\dagger}} \text{  for  }T \gg 0.$

\end{proof}

From the last Lemma we deduce the following
  \begin{lemma} \label{extension-isotopies-inductive}
 Let $\underline{Z}_{\prec G}^0 \in  \mathcal{Z}_{\prec G} $,  $\underline{Z}_{\prec G}^1 \in  \mathcal{Z}_{\prec G} $ and    $\tilde{Z} _{\prec G} \in  \tilde{ \mathcal{Z}}_{\prec G} $  isotopy between $\underline{Z}_{\prec G}^0$ and $\underline{Z}_{\prec G}^1$.
 Let $\underline{Z}_{\preccurlyeq G}^0    \in  \mathcal{Z}_{\preccurlyeq G} $ extending $\underline{Z}_{\preccurlyeq G}^0 $.
 

There exist
$\underline{Z}_{\preccurlyeq G}^1 \in  \mathcal{Z}_{\preccurlyeq G} $ and 
 $\tilde{Z} _{\preccurlyeq G} \in  \tilde{ \mathcal{Z}}_{\preccurlyeq G} $  isotopy between $\underline{Z}_{\preccurlyeq G}^0$ and $\underline{Z}_{\preccurlyeq G}^1  $ such that
\begin{itemize} 
\item $\underline{Z}_{\preccurlyeq G}^1  $ extends  $\underline{Z}_{\prec G}^1  $;
\item  $\tilde{Z} _{\preccurlyeq G}  $ extends $\tilde{Z} _{\prec G}  $. 
 \end{itemize}
 \end{lemma}
\begin{proof}
When $H(G) \neq \emptyset$ the lemma follows from Lemma \ref{nice-extension-isotopy} and Lemma \ref{nice-extension}.

When $H(G) = \emptyset$ , set $\underline{Z}_{G}^{1}=\underline{Z}_{G}^{0}+ pr_*(\delta_{e'} \tilde{\underline{Z}}_{G'}  )$, where $pr: M \times \R  \rightarrow  M$ is the the projection on the first factor. With this definition of $\underline{Z}_{G}^{1}$  is immediate to check that there exists 
$\tilde{\underline{Z}}_G$ such that the Lemma is satisfied. 

\end{proof}

The next Lemma may be considered as a one parameter version of Lemma  \ref{extension-isotopies-inductive}  , and it can be proved in a similar way.
 \begin{lemma} \label{extension-isotopies-isotopies}

Let  $ Z_{ \preccurlyeq G}^{0, 0},   Z_{ \preccurlyeq G}^{1, 0}  ,    Z_{ \preccurlyeq G}^{ 0, 1}  ,    Z_{ \preccurlyeq G}^{ 1 , 1} $ be $PS$-cycles. Let $ Z_{ \preccurlyeq G}^{0, \bullet}$ be an isotopy between $Z_{ \preccurlyeq G}^{0, 0}$ and $  Z_{ \preccurlyeq G}^{ 0, 1}$,  $ Z_{ \preccurlyeq G}^{ \bullet,0}$ be an isotopy between $Z_{ \preccurlyeq G}^{0, 0}$ and $  Z_{ \preccurlyeq G}^{ 1, 0}$,  $ Z_{ \preccurlyeq G}^{ \bullet,1}$ be an isotopy between $Z_{ \preccurlyeq G}^{0, 1}$ and $  Z_{ \preccurlyeq G}^{ 1, 1}$,
$ Z_{ \prec G}^{1, \bullet}$ be an isotopy between $Z_{ \prec G}^{1, 0}$ and $  Z_{ \prec G}^{ 1, 1}$.

For each $\tilde{\tilde{Z}} _{\prec G} \in  \tilde{\tilde{ \mathcal{Z}}}_{\prec G} $  isotopy of isotopies between  $Z_{ \prec G}^{ \bullet ,0 },  Z_{ \prec G}^{ \bullet ,1 }, Z_{ \prec G}^{0 ,\bullet} ,Z_{ \prec G}^{1 ,\bullet} $
there exists $Z_{ \preccurlyeq G}^{1, \bullet} \in \tilde{ \mathcal{Z}}_{\preccurlyeq G} $  and  $\tilde{\tilde{Z}}_{\preccurlyeq G} \in  \tilde{\tilde{ \mathcal{Z}}}_{\preccurlyeq G} $ isotopy of isotopies between  $Z_{ \preccurlyeq G}^{ \bullet ,0 },  Z_{ \preccurlyeq  G}^{ \bullet ,1 }, Z_{ \preccurlyeq  G}^{0 ,\bullet} ,Z_{ \preccurlyeq  G}^{1 ,\bullet} $ extending
$Z_{ \prec G}^{1, \bullet}$ , $\tilde{\tilde{Z}} _{\prec G} $ respectively. 


 \end{lemma}





\subsection{Appendix: Chains on Loops Space} \label{chain-loop-section}

Denote by $Map(\mathbb{S}^1,M) $ the set of piecewise smooth maps between the circle $\mathbb{S}^1$ and $M$. 
 Given a  cyclic ordered finet set $S$, denote by
 $Conf_S(\mathbb{S}^1)$ the set of injective maps between $S$ and $\mathbb{S}^1$ respecting the cyclic order.

The set $  \mathfrak{L}_S(M) $ of loops with marked points labeled by $S$ is defined by 
$$  \mathfrak{L}_S(M) = (  Map(\mathbb{S}^1,M) \times Conf_S(\mathbb{S}^1) )/(Diff^+(\mathbb{S}^1) ) .$$

We denote by 
$$ ev_S : \mathcal{L}_S M \rightarrow M^S $$
 the evaluation map on the marked  points.

When $S= \emptyset$, we simply denote $  \mathfrak{L}(M)  =   \mathfrak{L}_{\emptyset}(M)$.


In the following we define the space of chains (up to triangulations)  on $ \mathcal{L}_S M$. 

\subsubsection{Adding a marked point}

In this subsection we consider the operation of adding a marked point $\star$ to $S$ and we introduce some notation which will be useful in the following. 

Denote by $Conf_S^+(\mathbb{S}^1)'$ the space of injective maps between $S \sqcup \{  \star  \}$  to $\mathbb{S}^1$ respecting the cyclic order of $S$, and by $\overline{Conf}_S^+(\mathbb{S}^1)'$ its compactification. 
 $\overline{Conf}_S^+(\mathbb{S}^1)'$  is a manifold with corners whose  connected components are in  correspondence to the position of $\star$ in $\mathbb{S}^1 \setminus S$.

The connected components of  $\overline{Conf}_S^+(\mathbb{S}^1)'$ may be attached in a natural way to obtain a connected manifold $\overline{Conf}_S^+(\mathbb{S}^1)$.
The forget map $\overline{Conf}_S^+(\mathbb{S}^1) \rightarrow \overline{Conf}_S(\mathbb{S}^1)$ is a $\mathbb{S}^1$-fibration, trivial when $S \neq \emptyset$:
$$Conf_S^+(\mathbb{S}^1)  \cong Conf_S(\mathbb{S}^1) \times \mathbb{S}^1 \text{  if  } S \neq \emptyset.$$

\subsubsection{$k$-chains on $ \mathfrak{L}_S(M)$}



Consider arrays $(N,(w, (t_i)_i ))$ where
\begin{itemize}
    \item  $N$ is a compact oriented $k$-manifold with corners;
    \item  $(t_i)_{i \in S}: N \rightarrow \overline{Conf}_S(\mathbb{S}^1) \label{delta-configuration}$;
    \item $w: \overline{Conf}_S^+(\mathbb{S}^1) \times_{ \overline{Conf}_S(\mathbb{S}^1)   } N \rightarrow M .$ \label{delta-loop}
\end{itemize}
    We assume that the maps $(w, (t_i)_i ) $ are continuous and piecewise smooth. We assume that $w$ not constant on each fiber of the projection $\overline{Conf}_S^+(\mathbb{S}^1) \times_{ \overline{Conf}_S(\mathbb{S}^1)   } N  \rightarrow N$.

A $k$-chain on $ \mathfrak{L}_S(M)$ should be defined as a formal linear combination of arrays as above, modulo certain equivalence relations.
However for our purpose  it is necessary to add some extra structure to this definition.

Consider first the case $S \neq \emptyset$.  We consider  arrays
\begin{equation} \label{array-chains}
((N,N_0),( w, (t_i)_i  )  )
\end{equation}
 where
\begin{itemize}
\item  $N$ is a compact oriented $k$-manifold with corners;
    \item  $N_0 $ is a codimension two submanifold of $N$ intersecting transversally the boundary of $N$;
     \item $(t_i)_i : \hat{N} \rightarrow \overline{Conf}_S(\mathbb{S}^1) \label{delta-configuration1}$,  where $\hat{N}$ denotes the differential blow-up of $N$ along $N_0$. ;
    \item  $w : \overline{Conf}_S^+(\mathbb{S}^1) \times_{ \overline{Conf}_S(\mathbb{S}^1)   } \hat{N} \rightarrow M $\label{delta-loop1}.
\end{itemize}

Let $\hat{N}_0 $ be the pre-image of $N_0$ by the blow-down map $\hat{N} \rightarrow N$.
By definition $ \hat{N}_0$ is equipped by an $ \mathbb{S}^1$-action with $\hat{N}_0/  \mathbb{S}^1 = N_0$.
We assume that the restriction on  $\hat{N}_0 $ of $(w, (t_i)_i )$ is  $\mathbb{S}^1$-equivariant when $M$ is equipped with the trivial $\mathbb{S}^1$-action.

From the evaluation map associated to $(N,(w, (t_i)_i ))$ we obtain a map 
$$N  \rightarrow M^S$$
$$z \mapsto ( w(t_i(z), z) )_i$$ 
where we have used the  
 $\mathbb{S}^1$-equivariance  in order to blow-down the map from $\hat{N}$ to $N$.

We define a $k$-chain  on  $ \mathfrak{L}_S(M)$  as a formal linear combinations of arrays (\ref{array-chains}).



We consider the following equivalence relation.
Assume $ N= N^1 \sqcup_{P} N^2$, with $P$ transversal to $N_0$. Set 
$$N_0^1= N^1 \cap N_0, N_0^2= N^2 \cap N_0, $$
$$ (w^1, (t_i^1)_i ))= (w, (t_i)_i )|_{N^1} , (w^2, (t_i^2)_i ))= (w, (t_i)_i )|_{N^2}.$$ 
We declare
$$((N,N_0),(w, (t_i)_i )) \sim ((N^1,N_0^1),(w^1, (t_i^1)_i )) +((N^2, N_0^2),(w^2, (t_i^2)_i )).$$    


Now we consider the case $S = \emptyset$. We consider arrays
\begin{equation} \label{array-chains-zero}
((N,N_0), w  )
\end{equation}
 where
$N_0 $ is a boundary face of $N$, 
\begin{equation} \label{delta-loop0}
w: N \times \mathbb{S}^1 \rightarrow M .
\end{equation}
We assume that the restriction of $w$ to 
$N_0 \times \mathbb{S}^1 $ is constant along the $\mathbb{S}^1$-direction.

We define a $k$-chain  on  $ \mathfrak{L}_0(M)$  as a formal linear combinations of arrays (\ref{array-chains-zero}).



We consider the following  equivalence relation.
Assume $ N= N^1 \sqcup_{P} N^2$, with $P$ transversal to $N_0$. Set 
$$N_0^1= N^1 \cap N_0, N_0^2= N^2 \cap N_0,$$ 
$$ w^1= w|_{N^1} , w^2= w|_{N^2}.$$ 
We declare
$$((N,N_0), w  ) \sim ((N^1,N_0^1), w^1  ) +((N^2,N_0^2), w^2  ).$$

\subsubsection{Forgetting map}

 The definition of the equivalence relation on the $k$-chains is based on the notion of pull-back of a $k$-chain under the forget map which we define in this sub-section.

Let $S'$ be the set obtained adding a point to $S$. Denote by $s_0 \in S'$ the added point.

We have a natural forgetting map
$$ \mathbf{forget}_{s_0}:  \mathfrak{L}_{S'}(M) \rightarrow   \mathfrak{L}_{S}(M)  $$ 
which would like to consider as a fibration whose fibers are closed intervals if $S' \neq \emptyset$ or circles $\mathbb{S}^1$ if $S' = \emptyset$.
The precise meaning of this statement is that to each $k$-chain in $\mathcal{L}_{S'}(M) $ corresponds a pull-back $(k+1)$-chain on $\mathcal{L}_{S'}(M) $.
This chain can be explicitly defined as follows.

We need to associate to a generator  $(N,(w, (t_i)_i ))$ of $k$-chains on $ \mathfrak{L}_{S}(M)$ a generator  $(N',(w', (t_i')_i ))$ of $k+1$-chains on  $ \mathfrak{L}_{S'}(M)$. This is done as follows.

Consider first the case $N_0 = \emptyset$.
Set
\begin{equation} \label{pull-back-delta}
N' = \overline{Conf}_{S'}(\mathbb{S}^1) \times_{ \overline{Conf}_S(\mathbb{S}^1)   } N  .
\end{equation}
The map 
$$(t_i')_i  :   N' \rightarrow  \overline{Conf}_{S'}(\mathbb{S}^1)  $$
 is  the projection on the first factor.
The map
$$w': \overline{Conf}_{S'}^+(\mathbb{S}^1) \times_{ \overline{Conf}_{S'}(\mathbb{S}^1)   } N'  \rightarrow M $$
 is defined  from $w$ applying
  the isomorphism
 $$ \overline{Conf}_{S'}^+(\mathbb{S}^1) \times_{ \overline{Conf}_{S'}(\mathbb{S}^1)   } N' = \overline{Conf}_{S'}^+(\mathbb{S}^1) \times_{ \overline{Conf}_S(\mathbb{S}^1)   }  N$$ 
and the forget  map 
  $ \overline{Conf}_{S'}^+(\mathbb{S}^1) \rightarrow   Conf_{S}^+(\mathbb{S}^1) . $

Now consider the case $N_0 \neq \emptyset, S \neq \emptyset$. Set  
\begin{equation} \label{pull-back-delta-2}
\hat{N}' = \overline{Conf}_{S'}(\mathbb{S}^1) \times_{ \overline{Conf}_S(\mathbb{S}^1)   } \hat{N}  .
\end{equation}
The submanifold
$$\hat{N}_0'  := \overline{Conf}_{S'}(\mathbb{S}^1) \times_{ \overline{Conf}_S(\mathbb{S}^1)   } \hat{N}_0  $$
is equipped  with the $\mathbb{S}^1$-action induced from the $\mathbb{S}^1$-action on $ \overline{Conf}_{S'}(\mathbb{S}^1)$ and  $\hat{N}_0$.
The maps 
\begin{subequations} 
\begin{align}
(t_i')_i : \hat{N}' \rightarrow \overline{Conf}_{S'}(\mathbb{S}^1) \label{delta-configuration-pull-back-2} \\
w':  \overline{Conf}_{S'}^+(\mathbb{S}^1) \times_{ \overline{Conf}_{S'}(\mathbb{S}^1)   } \hat{N}'  \rightarrow M \label{delta-loop-pull-back-2}
\end{align}
\end{subequations}
are defined 
as before.

Define $N'$ as the quotient of $\hat{N}'$ with respect the $\mathbb{S}^1$-action on $\hat{N}_0'  $. 
 The maps (\ref{delta-configuration-pull-back-2}) ,  (\ref{delta-loop-pull-back-2}) restricted to $\hat{N}_0'$ are  $\mathbb{S}^1$-equivariant.

Finally consider the case $S= \emptyset$ (\ref{delta-loop0}). Set
\begin{equation} \label{pull-back-delta-0}
\hat{N}' = \overline{Conf}_{s_0}(\mathbb{S}^1) \times N.
\end{equation}
Define $N'$  quotient its subspace  $ \overline{Conf}_{s_0}(\mathbb{S}^1) \times N_0 $  by  the $\mathbb{S}^1$-action . 
The definition of $w'$ is similar to the case above.



\subsubsection{Isotopies}

The definition of chains given above may be extended  to chains  on  $ \mathfrak{L}_S(M) \times [0,1]$ straighforwardly. 

For semplicity we consider the case $N_0 = \emptyset$. 
A generator consists in an array $(N, (\tilde{w},(\tilde{t}_{i})_i, \lambda ))$ where
\begin{itemize}
    \item  $N$ is a compact oriented $k$-manifold with corners;
    \item  $(\tilde{t}_{i})_i: N \rightarrow Conf_S(\mathbb{S}^1) $ \label{delta-configuration-iso};
    \item $\tilde{w}_{}: \overline{Conf}_S^+(\mathbb{S}^1) \times_{ \overline{Conf}_S(\mathbb{S}^1)   } N \rightarrow M $ \label{delta-loop-iso};
    \item $\lambda : N \rightarrow [0,1]$.
    \end{itemize}

We have the natural evaluation map 
$$\text{ev}:   \mathfrak{L}_S(M) \times [0,1] \rightarrow M^S \times [0,1] $$
and the natural forgetting map
$$ \mathbf{forget}_{s_0}:  \mathfrak{L}_{S'}(M)  \times [0,1] \rightarrow   \mathfrak{L}_{S}(M)   \times [0,1]$$ 
where as before $S' = S \sqcup {s_0} $.

The definition of pull-back chain defined above may be adapted  straightforwardly.

\subsubsection{Multi-loop space}

Now  we consider chains on the (multi-)loop space. We consider arrays 
$$\hat{V} = (V , (H_v))_{v \in V})$$
where $V$ is a finite set and $H_v$ is a cyclic order set for each $v \in V$.

To $\hat{V}$ we associate the multi-loop space
$$\mathfrak{L}_{\hat{V}}(M) =  \prod_{v \in V} \mathfrak{L}_{H_v}(M) .$$
Chains on $\mathfrak{L}_{\hat{V}}(M)$ are linear combinations  of arrays $(N, ( N^0_v )_v , (w_v, t_v )_v)$ , where, for each $v \in V$, $(N,N^0_v, (w_v,t_v))$ is a generator of  $\mathfrak{L}_{ H_v}(M)$ as above.
We assume that the the manifolds $ \{ N^0_v \}_v$ are transversal.

Similarly we consider isotopies, i.e., chains on $  \mathfrak{L}_{\hat{V}}(M) \times [0,1]$. A chain on  $  \mathfrak{L}_{\hat{V}}(M) \times [0,1]$ is formal linear combination of  arrays $(N, ( N^0_v )_v , (w_v, t_v )_v,  \lambda)$ , where, for each $v \in V$, $(N,N^0_v, (w_v,t_v), \lambda)$ is a generator  $C_*(\mathfrak{L}_{ H_v}(M) \times [0,1])$ as above.
We assume that the the manifolds $ \{ N^0_v \}_v$ are transversal.

\section{ Point-Splitting-Homology} \label{PSH-section}

In this section we define Point-Splitting-Homology (PSH) as isotopy equivalence classes of nice $PS$-cycles. We start defining the notion of nice $PS$-cycle.

\subsection{Nice $PS$-Cycles}    \label{nice-section}

\subsubsection{Limit Support}

Let $G$ be a decorated graph.

Let $\underline{C}= ( {\underline{C}}^{(n)} )_n$ be a sequence of chains on $\mathfrak{L}_{\hat{V}(G)}M$.  We say that the limit support $\lim supp(\underline{C}) $ of $\underline{C}$
is empty, if for each $\epsilon$, for  $n$ big enough we have
\begin{equation} \label{support-bound}
 Diameter(\bigsqcup_{v \in V_c} supp(w_v)) < \epsilon 
 \end{equation}
  for each  $(w_v)_v \in supp (\underline{C}^{(n)})$ and $c \in Comp_0$.

\subsubsection{Nice Annulus Data}

\begin{definition} 
 A nice annulus data $\underline{Z}^{\circledcirc, \diamond}$ consists in a sequence of annulus data $(\underline{Z}^{\circledcirc, (n)})_n$ such that 
 \begin{itemize}
\item $ \lim Supp(\underline{Z}^{\circledcirc, \diamond}_G)= \emptyset$  for each $G \in Ann0$;
\item there exists   $(\tilde{\underline{Z}}^{\circledcirc, int, (n)})_n$, where $ \tilde{\underline{Z}}^{\circledcirc, int, (n)}$  is an isotopy of annulus data between $ \tilde{\underline{Z}}^{\circledcirc,  (n)} $ and $ \tilde{\underline{Z}}^{\circledcirc,  (n-1)}$   with
$ \lim Supp( pr_*(\tilde{\underline{Z}}^{\circledcirc, int}) =  \emptyset.$
\end{itemize} 
\end{definition}

\subsubsection{Nice $PS$-Cycles}

  \begin{definition} \label{nice-PSC-definition}
 Fix  a  nice annulus data $\underline{Z}^{\circledcirc, \diamond}$.  
  A nice $PS$-cycle  $Z^{\diamond } $ with annulus data $\underline{Z}^{\circledcirc, \diamond}$  consists in a sequence of $PS$-cycles $( Z^{(n)} )_n$ such that 
\begin{itemize}
\item $Z^{(n)} \in \mathcal{Z}(M| \underline{Z}^{\circledcirc, (n)})$ for each $n$;
\item   $\lim Supp(Z_{G,m}) = \emptyset$ for each $(G,m) \in \mathfrak{G}_*$;
\item there exists a sequence $(\tilde{Z}^{\circledcirc, int, (n)})_n$,  where $\tilde{Z}^{int, (n)}  \in \tilde{\mathcal{Z}}(M| \tilde{\underline{Z}}^{\circledcirc, int, (n)})$  isotopy of  $PS$-cycles   between $Z^{(n)} $ and $Z^{(n-1)}$, with  $\lim supp( pr_*(\tilde{Z}_{G,m}^{int}))= \emptyset$ for each $(G,m)$. (here $pr: M \times \R \rightarrow M$ is the projection on the first factor.)
\end{itemize} 
\end{definition}
We denote by $\mathcal{Z}(M| \underline{Z}^{\circledcirc, \diamond})$ the space of nice $PS$-cycles.





\begin{lemma} \label{extension-nice}
For each $\underline{Z}_{\prec G}^{\diamond} \in  \mathcal{Z}_{\prec G}^{\diamond} $ there exists $\underline{Z}_{\preccurlyeq G}^{\diamond} \in  \mathcal{Z}_{\preccurlyeq G}^{\diamond} $ extending $\underline{Z}_{\prec G}^{\diamond}$.
\end{lemma}
\begin{proof}
Take  $\underline{Z}_{\preccurlyeq G}^{(0)} \in  \mathcal{Z}_{\preccurlyeq G}$ extending $\underline{Z}_{\preccurlyeq G}^{(0)} \in  \mathcal{Z}_{\preccurlyeq G}$. Using induction on $n$, assume we have defined $\underline{Z}_{\preccurlyeq G}^{(n)}$. Apply Lemma  \ref{extension-isotopies-inductive} 
to $\underline{Z}_{\preccurlyeq G}^{(n)}, \tilde{ \underline{Z}}^{int, (n)}_{\prec G}$ to obtain $\underline{Z}_{\preccurlyeq G}^{(n+1)}, \tilde{ \underline{Z}}^{int, (n)}_{\preccurlyeq G}$.
The proof of Lemma \ref{nice-extension-isotopy} and Lemma \ref{nice-extension} tells us that we may construct $\underline{Z}_{\preccurlyeq G}^{(n+1)}, \tilde{ \underline{Z}}^{int, (n)}_{\preccurlyeq G}$ such that limit support is empty.
\end{proof}

\subsubsection{Isotopies}

Take a sequence of chains $\tilde{\underline{C}}= ( {\tilde{\underline{C}}}^{(n)} )_n$ on $ \mathfrak{L}_{\hat{V}(G)}M \times \R$.  We say that the limit support $\lim \tilde{supp}(\tilde{\underline{C}}) $ of $\tilde{\underline{C}}$ is empty if
for each $\epsilon>0$, for  $n$ big enough we have
 $$Diameter(\bigsqcup_{v \in V_c} supp(w_v)) < \epsilon $$
  for each $c \in Comp_0$ and $((w_v)_v, t) \in supp (\tilde{\underline{C}}^{(n)})$.


\begin{definition}
Let $\tilde{\underline{Z}}^{\circledcirc, \diamond}$ be an isotopy of nice annulus data  between nice annulus data  $\underline{Z}^{0, \circledcirc, \diamond}$ and $\underline{Z}^{1,\circledcirc, \diamond}$.


An isotopy  of nice $PS$-cycle $\tilde{Z}^{(n)}$ between nice $PS$-cycles   $Z^{0,\diamond} \in \mathcal{Z}(M|   Z^{0,\circledcirc,\diamond}    )$ and  $Z^{1 , \diamond} \in \mathcal{Z}(M|   Z^{1, \circledcirc ,\diamond}    )$ with annullus data $\tilde{\underline{Z}}^{\circledcirc, \diamond}$ consists in a sequence $\tilde{Z}^{\diamond }= ( \tilde Z^{(n)} )_n$, where $\tilde{Z}^{(n)} \in \tilde{\mathcal{Z}}(M| \tilde{\underline{Z}}^{\circledcirc, (n)})$   isotopy of $PS$-cycles between $Z^{0, (n)}$ and $Z^{1,(n)}$, with
$\lim \tilde{Supp} (\tilde Z_{G,m}^{(n)})_n = \emptyset$   for each $(G,m)$.
\end{definition}

We denote by $ \tilde{\mathcal{Z}}(M| \tilde{\underline{Z}}^{\circledcirc, \diamond})$ the vector space of isotopies of nice $PS$-cycle  with annullus data $\underline{Z}^{\circledcirc, \diamond}$.

\subsubsection{Obstructions} \label{obstruction-section}

Let $Z^{0,\diamond} , Z^{1 , \diamond} \in \mathcal{Z}(M|   Z^{ \circledcirc ,\diamond}    )$ . 
  
  Let $\tilde{\underline{Z}}_{\prec G} \in  \tilde{\mathcal{Z}}_{\prec G}(M| \tilde{\underline{Z}}^{\circledcirc, \diamond})$ be an isotopy between   $Z_{\prec}^{0,\diamond} , Z_{\prec G}^{1 , \diamond}$ .  
   In this subsection we determine the conditions for which there exists an isotopy $\tilde{\underline{Z}}_{\preccurlyeq G}$ extending $\tilde{\underline{Z}}_{\prec G}$.

In the case $H(G) \neq \emptyset$,  Lemma \ref{nice-extension-isotopy} tells us that there exists $\tilde{\underline{Z}}_{\preccurlyeq G}^{(n)}$ extending $\tilde{\underline{Z}}_{\prec G}^{(n)}$  . The construction may be made such that the limit support is trivial.

Assume  $H(G) = \emptyset$. If there exists $v \in V(G)$, for $n$ big enough, $\tilde{\underline{Z}}_{ G}^{(n)}$ may be defined shrinking $w_v$ in the chains 
$\underline{Z}^{0,(n)}, \underline{Z}^{1,(n)}$ and $   \sum_{(G', e')}  \delta_{e'} \tilde{\underline{Z}}_{G'}^{(n)}  +  \sum_{(G', v')} \delta_{v'} \tilde{\underline{Z}}_{G'}^{(n)} $.

In the case  $H(G) = \emptyset$, $V(G) = \emptyset$ we define the obstruction
\begin{equation} \label{definition-obstruction}
\mathfrak{obs}_G^{(n)} = \underline{Z}^{1,(n)} -  \underline{Z}^{0,(n)} + pr_*(   \sum_{(G', v')} \delta_{v'} \tilde{\underline{Z}}_{G'}^{(n)})  \in \Q.
\end{equation}

\begin{lemma} \label{obstruction-lemma}
The obstruction (\ref{definition-obstruction}) does not depend on $n$. The extension of $\tilde{\underline{Z}}$ to the graph $G$ exists if and only if $\mathfrak{obs}_G=0$.
\end{lemma}
\begin{proof}

Let $n \gg 0$.
Using recursively Lemma  \ref{extension-isotopies-isotopies} we may construct an isotopy of isotopies $\tilde{\tilde{\underline{Z}}}_{\prec G}^{(n)}$ between $\tilde{\underline{Z}}_{\prec G}^{(n-1)}, \tilde{\underline{Z}}_{\prec G}^{(n)} , \tilde{\underline{Z}}_{ \prec G}^{0, int, (n)}, \tilde{\underline{Z}}_{\prec G}^{1, int, (n)}$. We may define the obstruction to the existence of 
$\tilde{\tilde{\underline{Z}}}_{ G}^{(n)}$
analogously to (\ref{definition-obstruction}):
\begin{equation} \label{definition-obstruction-2}
\tilde{\mathfrak{obs}}_G^{int, (n)} = \tilde{\underline{Z}}^{1,int, (n)} -  \tilde{\underline{Z}}^{0, int, (n)} + pr_*(   \sum_{(G', v')} \delta_{v'} \tilde{\tilde{\underline{Z}}}_{G'}^{(n)})  \in C_1(\R),
\end{equation}
that is  a one-chain on $\R$.
It is easy to check that
$$\partial \tilde{\mathfrak{obs}}_G^{int,(n)} =0, $$
and
$$ \tilde{\mathfrak{obs}}_G^{int, (n), <-\lambda^{\dagger}}= \mathfrak{obs}_G^{(n-1)} \times \R_{<-\lambda^{\dagger}},  \tilde{\mathfrak{obs}}_G^{int, (n), > \lambda^{\dagger}}= \mathfrak{obs}_G^{(n)} \times \R_{>\lambda^{\dagger}}.$$
We conclude that $ \mathfrak{obs}_G^{(n-1)}= \mathfrak{obs}_G^{(n)}$.

\end{proof}

\begin{remark}
Let $\tilde{\underline{Z}} \in  \tilde{\mathcal{Z}}(M| \tilde{\underline{Z}}^{\circledcirc, \diamond})$ be an isotopy between   $Z^{0,\diamond} , Z^{1 , \diamond}$ .  It is easy to see that, for each $n$, there exists   an isotopy of isotopies $\tilde{\tilde{\underline{Z}}}^{(n)}$ between $\tilde{\underline{Z}}^{(n-1)}, \tilde{\underline{Z}}^{(n)} , \tilde{\underline{Z}}^{0, int, (n)}, \tilde{\underline{Z}}_{\prec G}^{1, int, (n)}$. This claim may be proved applying recursively Lemma  \ref{extension-isotopies-isotopies} . It is easy to see that the obstruction to the existence of $\tilde{\tilde{\underline{Z}}}^{(n)}_G$  vanishes automatically when $H(G) = \emptyset$  . 

\end{remark}

\subsection{$PS$-Homology} \label{PSH-subsection}

We are now in position to define Point-Splitting-Homology:

\begin{definition}
Let  $\underline{Z}^{\circledcirc, \diamond}$ be a nice annulus data. The Point-Splitting-Homology with annulus data $\underline{Z}^{\circledcirc, \diamond}$ ,  $PSH(M| \underline{Z}^{\circledcirc, \diamond})$,  is defined as the space of nice $PS$-cycles  $ \mathcal{Z}(M| \underline{Z}^{\circledcirc, \diamond})$ modulo isotopies  of nice $PS$-cycles  $ \tilde{\mathcal{Z}}(M| \underline{Z}^{\circledcirc, \diamond})$ .
\end{definition}





\subsubsection{Transfer Map} 
 
 The definition of $PSH$ depends on the choice of an annulus data. The next Lemma tells us how to relate  $PSH$  for isotopic annulus data .

  \begin{lemma} \label{transfer-annulus}
 Take $\tilde{Z}^{\circledcirc, \diamond}$ isotopy between nice annulus data  $Z^{\circledcirc, 0 , \diamond }$  and  $Z^{\circledcirc, 1 , \diamond}$.




 
 For each $Z^{0, \diamond} \in \mathcal{Z}(M|   Z^{\circledcirc, 0 ,\diamond}    )$ there exists $Z^{1, \diamond} \in \mathcal{Z}(M|   Z^{\circledcirc, 1 ,\diamond}    )$ and $\tilde{Z}^{\diamond} \in \tilde{ \mathcal{Z}}(M|   \tilde{Z}^{\circledcirc , \diamond}    )$  isotopy between $Z^{0, \diamond}$ and $Z^{1, \diamond}  $.
 
 The homology class $[Z^{1, \diamond}] \in PSH(M|   Z^{\circledcirc, 1 ,\diamond}    )$ is determined by the homology class $[Z^{0, \diamond}] \in PSH(M|   Z^{\circledcirc, 0,\diamond}    )$. We obtain a map
 \begin{equation} \label{transfer-map-annulus}
 transfer_{  \tilde{Z}^{\circledcirc , \diamond} }   : PSH(M| Z^{\circledcirc, 0 , \diamond}  )  \rightarrow PSH(M| Z^{\circledcirc, 1 , \diamond }  ) .
  \end{equation}
Take $\tilde{Z}^{\circledcirc, \prime }$ another isotopy between nice annulus data  $Z^{\circledcirc, \diamond, 0 }$  and  $Z^{\circledcirc, 1 }$. If $\tilde{Z}^{\circledcirc, \diamond, \prime }$  is related to  $\tilde{Z}^{\circledcirc , \diamond}$ by an isotopy of isotopies we have
$$transfer_{  \tilde{Z}^{\circledcirc , \diamond} }  =  transfer_{  \tilde{Z}^{\circledcirc , \diamond, \prime} } .$$
 
 \end{lemma}

 \begin{proof}

We proceed by induction on  graphs. Assume we have defined $\tilde{\underline{Z}}_{\prec G}^{(n)}, {\underline{Z}}_{\prec G}^{1,(n)}$ for each $n$.
Apply  Lemma  \ref{extension-isotopies-inductive} to ${\underline{Z}}_{\preccurlyeq G}^{0,(n)}, \tilde{\underline{Z}}_{\prec G}^{(n)}$ to obtain 
${\underline{Z}}_{\preccurlyeq G}^{1,(n)}, \tilde{\underline{Z}}_{\preccurlyeq G}^{(n)}$.
The construction may be made such that the limit support is trivial.



Let $Z^{1,\prime} \in \mathcal{Z}(M|   Z^{\circledcirc, \diamond, 1}    )$ and  $\tilde{Z}^{\prime} \in \tilde{ \mathcal{Z}}(M|   \tilde{Z}^{\circledcirc }    )$ be a $PS$-cycle and an isotopy of $PS$-cycle obtained from $\tilde{Z}^{\circledcirc, \prime }$.  Applying inductively Lemma \ref{extension-isotopies-isotopies}  we may define an isotopy of isotopies between $\tilde{Z}$ and $\tilde{Z}^{\prime}$. In particular this yields  an isotopy between $Z^1$ and $Z^{1,\prime}$. 


\end{proof}

\subsection{Euler Structures}

Recall that an Euler Structures on a compact oriented three manifold $M$ is defined by a nowhere vanishing vector field on $M$.  Two Euler Structures on $M$ are said homologous if they are isotopic outside a ball as nowhere vanishing vector fields. We denote by $\mathfrak{Eul}(M)$  the set of homology classes of Euler Structures. It is an elementary fact that $\mathfrak{Eul}(M)$ is a torsor on $H_1(M, \Z)$.


\subsubsection{An alternative definition}

Consider the set of  pairs $U^{\blacktriangle}=(U_0 , U_1)$ where
\begin{itemize}
\item
 $U_0$ is a three chain on $M \times M$ close in the $C^0$ topology to $Diag_M$ and transversal to $Diag_M$;
\item $U_1$ is a one-chain on $M$ with
$$\partial U_1 = U_0 \cap Diag_M.$$
Here we consider $U_0 \cap Diag_M$ as a chain on $Diag_M \cong M$.  
\end{itemize}

An isotopy $\tilde{U}^{\blacktriangle}$ between $U^{\blacktriangle,0}=(U_0^0 , U_1^0)$ and $U^{\blacktriangle,1}=(U_0^1 , U_1^1)$ is defined by a pair $(\tilde{U}_0 , \tilde{U}_1)$ where

\begin{itemize}
\item $\tilde{U}_0$ is a four chain on $ M \times M \times \R$ close in the $C^0$ topology to $ Diag_M \times \R$ and transversal to $ Diag_M \times \R $;
\item $\tilde{U}_1$ is a two-chain on $ M \times \R$ with
$$\partial \tilde{U}_1 = \tilde{U}_0 \cap ( Diag_M \times \R).$$
\item 

$$\tilde{U}_0^{< -\lambda^{\dagger}} = U_0^0     \times \R_{<-\lambda^{\dagger}} , \quad  \tilde{U}_1^{< -\lambda^{\dagger}} = U_1^0   \times   \R_{<-\lambda^{\dagger}},$$

$$\tilde{U}_0^{>\lambda^{\dagger}}= U_0^1  \times  \R_{>\lambda^{\dagger}}    , \quad  \tilde{U}_1^{>\lambda^{\dagger}} =  U_1^1  \times   \R_{>\lambda^{\dagger}}.$$

\end{itemize}


Let $\mathfrak{Eul}(M)^{\blacktriangle}$  be the set of isotopy classes of the objects $U^{\blacktriangle}$ .  The group $H_1(M, \Z)$ acts  on $\mathfrak{Eul}(M)^{\blacktriangle}$ by addiction on $U_1$, making $\mathfrak{Eul}(M)^{\blacktriangle}$ a torsor on $H_1(M, \Z)$.

It is easy to see that
\begin{equation}    \label{equivalence-Euler}
  \mathfrak{Eul}(M) \cong \mathfrak{Eul}(M)^{\blacktriangle}.
\end{equation}
To see this observe that
when $V$ is a vector field on $M$ with  small norm, using the exponential map, we can consider $graph(V)$ as three chain on $M \times M$ close in the $C^0$ topology to $Diag_M$. Hence  $ (graph(V), 0)$
defines an element of $\mathfrak{Eul}(M)^{\blacktriangle}$.
The isomorphism (\ref{equivalence-Euler}) follows since the map 
$$V \mapsto (graph(V), 0)$$
is compatible with the action of $H_1(M,\Z)$,
 

\subsubsection{Relation with Annulus Data}

 To an annulus data $\underline{Z}^{\circledcirc}=(\underline{Z}^{\circledcirc}_G)_{G \in Ann0}$ it is associated  an Euler Structure $(U_0,U_1)$ as follows.
 
Since $\delta_e(pr_e(Z_{G_0}))$ is close in the $C^0$-topology to $\delta_e Z_{G_0}$ and $\delta_e Z_{G_0} + \partial Z_{G_1} + \partial Z_{G_2}=0$, there exists 
$Z_{G_1}'$ close in the $C^0$-topology to $Z_{G_1}$ such that 
$$\delta_e(pr_e Z_{G_0}) + \partial Z_{G_1}' + \partial Z_{G_2}=0.$$ 
It follows that
\begin{equation}  \label{Ann-Euler-map0}
U^{\blacktriangle}= (pr_e(Z_{G_0}), Z_{G_1}' + Z_{G_2}) 
\end{equation}
defines an Euler Structure. 
$U^{\blacktriangle}$ is not uniquely determinate from $\underline{Z}^{\circledcirc}$, however the set of its choice is contractible. 


Assume that   $Z^{\circledcirc, 0 }$  and  $Z^{\circledcirc, 1 }$ are annulus data and that $U^{\blacktriangle,0}$ and $U^{\blacktriangle,1}$ are the Euler Structures associated $Z^{\circledcirc, 0 }$  and  $Z^{\circledcirc, 1 }$ respectively.
The one parameter version of the argument above shows that to each isotopy $\tilde{Z}^{\circledcirc}$ between annulus data  $Z^{\circledcirc, 0 }$  and  $Z^{\circledcirc, 1 }$ we may associate an isotopy $\tilde{U}^{\blacktriangle}$ between $U^{\blacktriangle,0}$ and $U^{\blacktriangle,1}$, well defined up to isotopy of isotopies.

Set
$$ PSH(M |   U^{\blacktriangle} ) =PSH(  M|   Z^{\circledcirc} ) .$$
The definition makes sense since Lemma \ref{transfer-annulus} tells us that we may identify canonically the spaces 
$PSH(  M|   Z^{\circledcirc} ) $ for different choices of $Z^{\circledcirc}$ corresponding to the same Euler Structure $U^{\blacktriangle}$ .

To each isotopy $\tilde{U}^{\blacktriangle} $  between  the Euler Structures  $U^{\blacktriangle,0} $ and $U^{\blacktriangle,1}    $ it is associated a map
$$ Transfer_{\tilde{U}^{\blacktriangle}}:  PSH(M |   \tilde{U}^{\blacktriangle,0}) \rightarrow PSH(M |  \tilde{U}^{\blacktriangle,1})  .$$

\subsection{Factorization Property}

 In the standard  approach to Perturbative Chern-Simons the configuration space of points has a trivial factorization property, namely the configuration space of points of a disjoint union of graphs  is the product of the configuration of points of each graph.  This property leads to write  the partition function as the exponential of the Chern-Simons potential defined in terms of  connected graphs.
  
 In the situation of the present paper
 the factorization property is not compatible with transversality condition in the trivial sense.
 In this subsection we introduce   a more sophisticate version of the factorization condition for $PSH$-classes .

\subsubsection{Factorization $PS$-cycles}

Set
$$\mathfrak{G}^{!}=  \mathfrak{G} \times \mathfrak{G}.$$
Denote by $G^1 \oplus G^2$ the disjoint union of the decorated graphs $G^1,G^2 \in \mathfrak{G}$.

Let $\mathfrak{G}^{!}_*$ be the set of pairs $((G^1,G^2) , m)$ with $(G^1,G^2) \in \mathfrak{G}^{!}$ and $m$ is a marking  for $G^1 \oplus G^2$   as in  Subsection \ref{marking-section} . Note that 
define $m$ is equivalent to define  a marking $m^1$ for $G^1$ and a marking $m^2$ for $G^2$. We write $m= m^1 \oplus m^2$.

Fix an  annulus data $\underline{Z}^{\circledcirc}$.
A factorization $PS$-cycles $Z^!$ with annulus data $\underline{Z}^{\circledcirc}$  consists on a collection of chains
$$( Z^!_{(G^1,G^2) , m)} )_{((G^1,G^2),m) \in \mathfrak{G}_*^!} $$
with
$$ Z^!_{(G^1,G^2) , m)}  \in C_{|H(G^1)| + |H(G^2)|+ |m|} ( M^{H(G^1) \sqcup H(G^2)} , \mathfrak{o}_{H(G^1) \sqcup H(G^2)})^{\text{Aut}(G^1,G^2) , m)}  .$$
such that
\begin{itemize}
\item the transversality condition holds; 
\item  the forgetful compatibility holds  with annulus data $\underline{Z}^{\circledcirc}$ ;
\item the cycle condition holds .
\end{itemize}

The definitions  for factorization $PS$-cycles of the transversality, the forgetful compatibility and the cycle conditions are a straightforward extension of the definitions provided in  Subsection \ref{PSC-subsection}.




 We denote by $\mathcal{Z}^!$  the vector space of factorization $PS$-cycles.

Isotopies $\tilde{\mathcal{Z}}^!$ of  factorization $PS$-cycles can be defined analogously  modifying the definition of isotopies of  subsection ( \ref{isotopies-section}).


\subsubsection{Factorization Map}

Define the factorization map between the usual space of $PS$-cycles and the factorization $PS$-cycles:
\begin{equation} \label{factorization-map}
 \mathbf{fact} :  \mathcal{Z} \rightarrow  \mathcal{Z}^!
\end{equation}
\begin{equation} \label{factorization-formula-definition}
 \mathbf{fact}(Z)_ {(( G^1,G^2),m  ) }  := Z_{(G^1 \oplus G^2,  m) }. 
 \end{equation}

\begin{lemma}
For each $PS$-cycle $Z$,  $   \mathbf{fact}(Z)$ is a factorization $PS$-cycle.
\end{lemma}
\begin{proof} It is an easy check.
\end{proof}

The factorization map  for isotopies of $PS$-cycles 
\begin{equation} \label{factorization-map-iso}
 \mathbf{fact} : \tilde{\mathcal{Z}} \rightarrow  \tilde{\mathcal{Z}}^!.
\end{equation}
is defined from
\begin{equation} \label{factorization-formula-definition-iso}
 \mathbf{fact}(\tilde{Z})_ {(( G^1,G^2),m  ) }  := \tilde{Z}_{(G^1 \oplus G^2,  m) }. 
 \end{equation}


\subsubsection{Product}

Define the  product of $PS$-cycles
\begin{equation} \label{product-factorization}
  \boxtimes : \mathcal{Z}  \times  \mathcal{Z} \rightarrow \mathcal{Z}^! ,
  \end{equation}
$$   (Z^1, Z^2) \rightarrow  Z^1\boxtimes Z^2$$
from
\begin{equation} \label{cup-formula}
(Z^1\boxtimes Z^2)_{(G^1,G^2), m } := \sum_{0 \leq r \leq l}  Z^1_{(G^1, m^1_{[0,r]})} \times  Z^2_{(G^2, m^2_{[r,l]})}. 
\end{equation}
Here we have used the notation $m_{[a,b]} := \{ E_i \}_{a \leq i \leq b}$ for $m = \{ E_i \}_{0 \leq i \leq l}$ and $0 \leq a \leq b \leq l$.

\begin{lemma}
For each pair of $PS$-cycles $(Z^1, Z^2)$,  $ Z^1\boxtimes Z^2$ is a factorization $PS$-cycle.
\end{lemma}
\begin{proof} It is an easy check.
\end{proof}

The definition of product is extended to isotopies
\begin{equation} \label{product-factorization-isotopies}
  \boxtimes : \tilde{\mathcal{Z}}  \times  \tilde{\mathcal{Z}} \rightarrow \tilde{\mathcal{Z}}^! ,
  \end{equation}
$$   (\tilde{Z}^1, \tilde{Z}^2) \rightarrow  \tilde{Z}^1\boxtimes \tilde{Z}^2$$
from
\begin{equation} \label{cup-formula-iso}
(\tilde{Z}^1\boxtimes \tilde{Z}^2)_{(G^1,G^2), m } := \sum_{0 \leq r \leq l}  \tilde{Z}^1_{(G^1, m^1_{[0,r]})} \times_{\R}  \tilde{Z}^2_{(G^2, m^2_{[r,l]})}. 
\end{equation}

\subsubsection{Nice Factorization PS-cycles}

Using  factorization $PS$-cycles  in place of $PS$-cycles, we may define Nice Factorization $PS$-cycles $\mathcal{Z}^{!, \diamond}$ as in section  \ref{PSH-section}. The isotopies $\tilde{Z}^{int,(n)}$ used in the definition of section  \ref{PSH-section} now are assumed to be isotopies of  factorization $PS$-cycles.

 Factorization $PS$-homology $PSH^!(M  | \underline{Z}^{\circledcirc, \diamond}  )$ is defined taking equivalence classes $\mathcal{Z}^{!, \diamond}$  .


 \begin{lemma}
 For ${Z}^{ \diamond}  = (Z^{(n)})_n \in \mathcal{Z}^{\diamond}$  
 $$ \mathbf{fact}(Z^{\diamond}):=(\mathbf{fact}(Z^{(n)}))_n$$
  defines a nice factorization $PS$-cycle. 
\end{lemma}
\begin{proof}
In order to show that $( \mathbf{fact}( Z^{(n)}))_n$ defines a  nice factorization $PS$-cycle   we may use  $( \mathbf{fact}(\tilde{Z}^{int,(n) }))_n$.
\end{proof}

 \begin{lemma}
For ${Z}^{ 1, \diamond}  = (Z^{1, (n)})_n , {Z}^{ 2, \diamond}  = (Z^{2, (n)})_n \in   \mathcal{Z}^{ \diamond}$  , 
$$Z^{1, \diamond} \boxtimes Z^{2, \diamond}:= ( Z^{1,(n)} \boxtimes Z^{2,(n)})_n$$
 defines a nice factorization $PS$-cycle. 
\end{lemma}
\begin{proof}
 In order to show that $( Z^{1,(n)} \boxtimes Z^{2,(n)})_n$   defines a  nice factorization $PS$-cycle we  may use  $(\tilde{Z}^{1,int, (n)} \boxtimes \tilde{Z}^{2,int, (n)})_n$. 

\end{proof}

\begin{lemma}
The product  (\ref{cup-formula}) induces  a commutative product in $   PSH(M| \underline{Z}^{\circledcirc, \diamond})$
\begin{equation} \label{product-nice-homology}
    PSH(M| \underline{Z}^{\circledcirc, \diamond}) \boxtimes PSH(M| \underline{Z}^{\circledcirc, \diamond}  ) \rightarrow PSH^!(M  | \underline{Z}^{\circledcirc, \diamond}  ).
\end{equation}

\end{lemma}
\begin{proof}

 It is immediate to check that the product of equivalent nice $PS$-cycles  leads to equivalent factorization $PS$-cycles.  

The commutativity of the product in $PSH$ follows imitating the argument used in the proof of the commutativity ( up to sign ) of the cap product of usual cohomology. In our case the sign does not appear since it is cancelled by the change of orientation associated to the exchange of  the order in the product of chains appearing in (\ref{cup-formula}).
\end{proof}

\subsubsection{Factorization condition}

Now we are in position to state the factorization condition:

\begin{definition}
We say that   $[Z^{\diamond}] \in    PSH(M| \underline{Z}^{\circledcirc, \diamond})$
 satisfies the factorization property if 
\begin{equation} \label{factoriation-equation}
  \mathbf{fact}([Z^{\diamond}]) = [Z^{\diamond}] \boxtimes [Z^{\diamond}] .
  \end{equation}
\end{definition}

\section{Chern-Simons Partition Function}

Let $M$ be a compact oriented three-manifold and let $\mathfrak{g}$ be a principal flat vector bundle over $M$.  Denote by $H^*(M, \mathfrak{g})$  the cohomology  of $M$ with coefficients in $\mathfrak{g}$ endowed with the odd symplectic form induced by the pairing
$$\langle \alpha \otimes X,  \alpha' \otimes X' \rangle =
 \int_M \alpha \wedge \alpha'  \langle  X,  X' \rangle_{\mathfrak{g}} .$$
In this section we define  the  Point Splitting Perturbative Chern-Simons partition function $\mathfrak{P}([Z])  \in \R [[ \hbar,x  ]]$ associated  to a $PSH$-class $[Z] \in PSH(M| U)$, which is a formal power series on the formal variable $x \in (H^*(M, \mathfrak{g})[1])^*$.  $\mathfrak{P}([Z]) $ is determined up to  up to Quantum Master Equation isotopies.

In this section we first define the partition function for each $PS$-cycle. 
The partition function is defined on a particular class of nice $PS$-cycles which we call smooth nice $PS$-cycles.  
This is enough to define the partition function on $PSH$.
If the factorization property is satisfied we show  that the Chen-Simons potential is defined.



The partition function will depend not only on the choice of flat connection of the $\mathfrak{g}$-bundle, but also on the choice of a representation
$$ \rho :  \mathfrak{g} \rightarrow \mathfrak{gl}(n, \C) .$$
In the following we implicitly identify $\mathfrak{g}$ with a sub-Lie-algebra of $\mathfrak{gl}(n, \C)$ using $\rho$.

\subsection{ Chern-Simons Propagator}

In this subsection  we introduce the main ingredient  of the definition of the partition function of $PCPCS$, i.e.,  the Chern-Simons Propagator.
The construction is mostly similar to the construction of the propagator   in the case of the standard perturbative Chern Simons of \cite{CS}  . The main difference consists on the conditions which are required to $i_{\partial}^*( P)$.
 In order to make control of the anomalies, in  \cite{CS} was necessary to make a more specific condition on $i_{\partial}^*( P)$, which in particular  depends on the choice of a connection on the tangent space of $M$. In $PCPCS$ we may use a more general propagator.
  




Take   a finite dimensional vector subspace $\mathcal{A} < \Omega^*(M, \mathfrak{g})$ of the space of closed differential forms of $M$ with value in $\mathfrak{g}$ such that its map in cohomology of $M$ yields an isomorphism:
\begin{equation} \label{harmonic-isomorphism}
   \mathcal{A} \xrightarrow{\,\raisebox{-0.4ex}{$\sim$}\,} H^*(M, \mathfrak{g}) .
   \end{equation}
Equivalently we have  a linear map
$$ r_{\mathcal{A}}: H^*(M, \mathfrak{g}) \rightarrow \Omega^*(M, \mathfrak{g})$$
such that 
$$ r_{\mathcal{A}}([\alpha])= \alpha \quad \forall \alpha \in \mathcal{A} .$$





Denote by  $\text{Conf}_2(M)$ the compactification of the configuration space of two non-overlapping points of $M$ , i.e. , the differential geometric blow-up of $M \times M$ along the diagonal $Diag \subset M \times M$. 
$\text{Conf}_2(M)$  is a manifold with boundary with the following properties:
\begin{itemize}
    \item 
    $$    \overset{\circ}{\text{Conf}}_2(M)  =  M \times M \setminus Diag  ;$$
    \item 
$$ \partial  \text{Conf}_2(M) = S(N(Diag \subset M \times M)) ,$$
where the LHS is the sphere bundle of the normal bundle of the diagonal;
\item there are natural projections $ \pi_i :   \text{Conf}_2(M)  \rightarrow M $, $i=1,2$ . 
\end{itemize}


The switch diffeomomorphism $(x,y) \rightarrow (y,x)$ on $M^2$ induces  a diffeomorphism 
$$sw : \text{Conf}_2(M,   \pi_1^*(\mathfrak{g}) \otimes \pi_2^*(\mathfrak{g}) )  \rightarrow \text{Conf}_2(M, \pi_2^*(\mathfrak{g}) \otimes \pi_1^*(\mathfrak{g})  ) .$$ 
The restriction of $sw$ to $ \partial  \text{Conf}_2(M) $ agrees with the reverse diffeomorphism  on the fibers of the spherical bundle $S(N(Diag \subset M \times M)) \rightarrow M$.
We denote by $H^2(\partial \text{Conf}_2(M) , \Z )^-$  the cohomology anti-invariant under the reverse diffeomorhism. We have
$$H^2(\partial \text{Conf}_2(M) , \Z )^- = H^2(S^2, \Z)  =  \Z .$$


Let $I_{\mathfrak{g}} \in \mathfrak{g} \otimes \mathfrak{g}$ be the tensor dual of the pairing on $\mathfrak{g}$. 
 $I_{\mathfrak{g}}$ defines a parallel section of $\pi_1^*(\mathfrak{g}) \otimes \pi_2^*(\mathfrak{g})$ over $Diag$, and so, by pull-back,    a differential form on $\partial \text{Conf}_2(M)$.  
 
We denote by
$$PD(Diag)  \in H^3(M \times M  ,   \pi_1^*(\mathfrak{g}) \otimes \pi_2^*(\mathfrak{g})    )$$
the Poincare dual of the diagonal of $M \times M$.

 
 
\begin{lemma}  \label{construction-propogator}
Take  $\mathcal{A} < \Omega^*(M, \mathfrak{g})$ as above. Let $Vol_{S^2}$ be a differential form representing the generator of $H^2(\partial \text{Conf}_2(M) , \Z )^-$. 

There exists a differential form 
$$P \in \Omega_2(\text{Conf}_2(M),   \pi_1^*(\mathfrak{g}) \otimes \pi_2^*(\mathfrak{g})  )$$
 such that 
\begin{itemize}
\item 
\begin{equation}     \label{propagator-differential}
dP = r_{\mathcal{A}} \otimes  r_{\mathcal{A}} (PD(Diag));
\end{equation}
\item   Denote by $i_{\partial}: \partial \text{Conf}_2(M) \rightarrow \text{Conf}_2(M) $ the inclusion map.
\begin{equation}   \label{propagator-singularity}
 i_{\partial}^*( P) = Vol_{S^2} \otimes I_{\mathfrak{g}} \in \Omega^*(\partial \text{Conf}_2(M)  ,\pi_1^*(\mathfrak{g}) \otimes \pi_2^*(\mathfrak{g}))^-.   
\end{equation}

\item 
\begin{equation} \label{parity}
 \text{sw}^*(P) = -P  ;
 \end{equation}
\item
\begin{equation} \label{fixing-class-condition}
\langle P,\pi_1^*(\alpha') \wedge \pi_2^*(\alpha'') \rangle =0
 \end{equation}
   for each   $ \alpha', \alpha'' \in \mathcal{A}$ . 
\end{itemize}

The differential form $P$ is uniquely determinate up to an element of 
\begin{equation} \label{propagator-unicity}
d \Omega^1(\text{Conf}_2(M), \partial \text{Conf}_2(M),  \pi_1^*(\mathfrak{g}) \otimes \pi_2^*(\mathfrak{g})  ).
\end{equation}


  \end{lemma}

\begin{proof}



Let $U$ be a small tubular neighborhood of the diagonal $Diag$ on $M \times M$.  We can associate to $U$ a collar $\hat{U}$ of  $\partial  \text{Conf}_2(M)$ .
Taking the pull back from $\hat{U} \rightarrow \partial \text{Conf}_2(M)$ we may consider the differential two-form $Vol_{S^2}$   as a closed differential two-form  on  $\hat{U}$.

  $I_{\mathfrak{g}}$ may be extended to parallel section of $ \Omega^0(\hat{U}, \pi_1^*(\mathfrak{g}) \otimes \pi_2^*(\mathfrak{g}))$ which we denote in the same way:  
  $$I_{\mathfrak{g}} \in \Omega^0(\hat{U}, \pi_1^*(\mathfrak{g}) \otimes \pi_2^*(\mathfrak{g})).$$

Let $\rho$ be a cutoff function on $M \times M$ equal to one in a neighborhood of $Diag$ and zero outside a compact subset of $U$. We assume that $\rho$ invariant under the switch diffeomorphism $sw$.
The differential form $d (\rho Vol_{S^2}) \otimes I_{\mathfrak{g}}$ defines a smooth closed differential form on $M \times M$, whose cohomology class agrees with the Poincare dual of the diagonal  
$$ [d (\rho  Vol_{S^2}) \otimes  I_{\mathfrak{g}}  ] = PD (Diag) \in H^3(M \times M  ,   \pi_1^*(\mathfrak{g}) \otimes \pi_2^*(\mathfrak{g})    ).$$ 
So there exists  $\psi \in \Omega^2(M \times M,  \pi_1^*(\mathfrak{g}) \otimes \pi_2^*(\mathfrak{g})  )$ such that 
$$ r_{\mathcal{A}} \otimes r_{\mathcal{A}}  (PD(Diag)) = d (\rho Vol_{S^2})  \otimes  I_{\mathfrak{g}} + d \psi. $$
We may  assume that  $\psi $ is anti-invariant under the switch diffeomorphism $sw$.

Set
$$P^{prel}= \rho Vol_{S^2} \otimes  I_{\mathfrak{g}} + \psi .$$ 
$P^{prel}$ satisfies   conditions 
 (\ref{propagator-differential}), (\ref{propagator-singularity}) and (\ref{parity}) .
We may add to $P^{prel}$ a suitable closed differential two-form 
in order to satisfy condition (\ref{fixing-class-condition}). In this way we obtain the required differential form $P$.

In order to prove the last claim
 observe that
the difference between two differential forms satisfying the above conditions defines a differential form 
$$\theta  \in \Omega^2(\text{Conf}_2(M),   \pi_1^*(\mathfrak{g}) \otimes \pi_2^*(\mathfrak{g})  )$$
 such that
 \begin{equation} \label{conditions-difference-propagators}
 d \theta=0, i_{\partial}^*( \theta)=0,  \text{sw}^*(\theta) = -\theta, 
 \end{equation}
 \begin{equation} \label{conditions-difference-propagators2}
   \langle \theta ,\pi_1^*(\alpha') \wedge \pi_2^*(\alpha'') \rangle =0 
  \end{equation}
    for each   $ \alpha', \alpha'' \in \mathcal{A}$ .


There exists $\zeta \in  \Omega^1(\text{Conf}_2(M), \partial \text{Conf}_2(M),  \pi_1^*(\mathfrak{g}) \otimes \pi_2^*(\mathfrak{g})  )$ such that the differential form $\theta':= \theta + d \zeta $ has  support away from $ \partial \text{Conf}_2(M)$, and so
it defines an element of $ \Omega^2(M \times M,   \pi_1^*(\mathfrak{g}) \otimes \pi_2^*(\mathfrak{g})  )$. Condition ( \ref{conditions-difference-propagators2}) implies that $\theta'$ is exact in $ \Omega^2(M \times M,   \pi_1^*(\mathfrak{g}) \otimes \pi_2^*(\mathfrak{g})  )$.





\end{proof}


Instead to consider $P$ as a differential forms on $\text{Conf}_2(M)$ we can think $P$ as an element of $\Omega^2(M \times M \setminus Diag)$, with singularity along the diagonal.
The second condition implies that
$\text{lim}_{\epsilon \rightarrow 0} \int_{S^2_{\epsilon}} P =  1 $          
where $S^2_{\epsilon}$ is a family of spheres of radius $\epsilon$ surrounding the diagonal of $M \times M$. 

Using the propagator $P$ we  define the operator 
$$ \mathbf{P} : \Omega^*(M, \mathfrak{g}) \rightarrow \Omega^*(M, \mathfrak{g}) $$ 
$$ \mathbf{P}(\omega) = \pi_* (P \wedge \pi_2^* (\omega) )  $$ 
which satisfies the relation
\begin{equation} \label{diff-P}
d \mathbf{P} - \mathbf{P} d = Id - \pi_{\mathcal{A}} .
\end{equation}

\subsubsection{One parameter version} \label{parameter-propagator-section}

The propagator $P$  defined in Lemma \ref{construction-propogator} depends on some choices, in particular the data  $\mathcal{A},Vol_{S^2}$.   In this subsection we show that different choices lead to isotopic propagators. 

We start with the following two immediate Lemmas:
\begin{lemma} \label{isotopy-harmonic-lemma}
Let $\mathcal{A}_0, \mathcal{A}_1$ be subspaces of 
closed $\mathfrak{g}$-value differential forms on $M$ such that
$$   \mathcal{A}_0 \xrightarrow{\,\raisebox{-0.4ex}{$\sim$}\,} H^*(M, \mathfrak{g}) , \quad  \mathcal{A}_1 \xrightarrow{\,\raisebox{-0.4ex}{$\sim$}\,} H^*(M, \mathfrak{g}). $$
There exists a sub-vector space $\tilde{ \mathcal{A}} < \Omega^*(M \times [0,1], \mathfrak{g})$ of closed differential forms
such that  
$$i_0^*(\tilde{ \mathcal{A}}) = \mathcal{A}_0, \quad   i_1^*(\tilde{ \mathcal{A}}) = \mathcal{A}_1 .$$


\end{lemma}

To $\tilde{ \mathcal{A}} $ as in the Lemma it is associated a linear map
$$ r_{\tilde{ \mathcal{A}} }: H^*(M, \mathfrak{g}) \rightarrow \Omega^*(M \times [0,1], \mathfrak{g})$$
such that 
$$ r_{\tilde{ \mathcal{A}} }([\alpha])= \alpha \quad \forall \alpha \in \tilde{ \mathcal{A}}  .$$

\begin{lemma} \label{isotopy-volume-lemma}
Let $Vol_{S^2,0}, Vol_{S^2,1} $ be closed differential forms representing  the generator of $H^*(\partial \text{Conf}_2(M) , \Z )^-$ .
 There exists a closed differential form $\tilde{Vol}_{S^2} \in   \Omega^2( \partial \text{Conf}_2(M) \times [0,1] )^- $ such that 
 $$i_0^* (\tilde{Vol}_{S^2})= Vol_{S^2,0}, \quad i_1^* (\tilde{Vol}_{S^2})= Vol_{S^2,1} .$$
 \end{lemma}
 

\begin{proposition}  \label{parameter-propagator-lemma}
 Let $P_0$ and $P_1$ be propagators satisfying the conditions of  Lemma \ref{construction-propogator} for the data $\mathcal{A}_0,  Vol_{S^2,0}$ and $\mathcal{A}_1,Vol_{S^2,1}$ respectively.

Take $\tilde{ \mathcal{A}}$ and $\tilde{Vol}_{S^2} $ as in Lemma \ref{isotopy-harmonic-lemma} and Lemma \ref{isotopy-volume-lemma}.


 There exists a differential form $\tilde{P} \in \Omega^2(\text{Conf}_2(M) \times [0,1],   \pi_1^*(\mathfrak{g}) \otimes \pi_2^*(\mathfrak{g})  ) )$ such that 
\begin{itemize}

\item 
\begin{equation} \label{initial-final}
i_0^*(\tilde{P})= P_0 , i_1^*(\tilde{P})= P_1 ;
\end{equation}
\item 
\begin{equation}     \label{propagator-differential-iso}
d \tilde{P} = r_{\tilde{\mathcal{A}}}  \otimes  r_{\tilde{\mathcal{A}}} (PD(Diag));
\end{equation}
\item 
\begin{equation}   \label{propagator-singularity-iso}
 i_{\partial}^*( \tilde{P}) = \tilde{Vol}_{S^2} \otimes I_{\mathfrak{g}} \in \Omega^*(\partial \text{Conf}_2(M) \times [0,1]  ,\pi_1^*(\mathfrak{g}) \otimes \pi_2^*(\mathfrak{g}))^-.   
\end{equation}

\item 
\begin{equation} \label{parity-iso}
 \text{sw}^*(\tilde{P}) = -\tilde{P}  ;
 \end{equation}
\item
\begin{equation} \label{fixing-class-condition-iso}
\langle \tilde{P},\pi_1^*(\tilde{\alpha}') \wedge \pi_2^*(\tilde{\alpha}'') \rangle =0 \in \Omega^*([0,1])
 \end{equation}
   for each   $ \tilde{\alpha}', \tilde{\alpha}'' \in \tilde{\mathcal{A}}$ .

\end{itemize}
\end{proposition}
In  (\ref{fixing-class-condition-iso}) the coupling is defined using the push forward of differential forms throughout the projection $ \text{Conf}_2(M) \times [0,1] \rightarrow [0,1]$. 

\begin{proof}
The construction of $\tilde{P}$ follows the same lines of  the argument   of the proof of Lemma  \ref{construction-propogator}.


 We may consider the differential two-form $\tilde{Vol}_{S^2}$   as a closed differential two-form  on  $\hat{U} \times [0,1]$.
The differential form $d (\rho \tilde{Vol}_{S^2})\otimes  I_{\mathfrak{g}} $ defines a smooth differential form on $M \times M \times [0,1]$. Its homology class agrees with the Poincare dual of the diagonal  
$$ [d (\rho  \tilde{Vol}_{S^2})  \otimes  I_{\mathfrak{g}}] = PD (Diag  \times [0,1] ) \in H^3(M \times M  \times [0,1]  ,\mathfrak{g}).$$ 
There exists  $\tilde{\psi} \in \Omega^2(M \times M \times [0,1],  \pi_1^*(\mathfrak{g}) \otimes \pi_2^*(\mathfrak{g})  )$ such that 
$$    r_{\tilde{\mathcal{A}}}  \otimes  r_{\tilde{\mathcal{A}}} (PD (Diag  \times [0,1] ))=d (\rho \tilde{Vol}_{S^2}  \otimes  I_{\mathfrak{g}})+ d \tilde{\psi}.$$
We may take  $\tilde{\psi} $ anti-invariant under the switch diffeomorphism $sw$.

Set
$$\tilde{P}^{prel}= \rho \tilde{Vol}_{S^2} \otimes  I_{\mathfrak{g}} + \tilde{\psi} .$$ 
$\tilde{P}^{prel}$ satisfies   conditions (\ref{propagator-differential-iso}), (\ref{propagator-singularity-iso}) and (\ref{parity-iso}) .

From  $(\ref{propagator-differential-iso})$ and $( \ref{propagator-singularity-iso})$ it follows that 
$$d \langle \tilde{P}^{prel},\pi_1^*(\tilde{\alpha}') \wedge \pi_2^*(\tilde{\alpha}'') \rangle =0 \in \Omega^*([0,1]) $$
 for each   $ \tilde{\alpha}', \tilde{\alpha}'' \in \tilde{\mathcal{A}} .$
 So we may add a suitable closed differential form  to $\tilde{P}^{prel}$  in order to satisfy condition (\ref{fixing-class-condition-iso}).

Finally , using   (\ref{propagator-unicity}), there exists $\psi_0 ,  \psi_1$ such that
 $i_0^*(\tilde{P}^{prel})- P_0 = d \psi_0, i_1^*(\tilde{P}^{prel})- P_1= d \psi_1$.  Take $\tilde{\psi}$ vanishing on $\partial Conf_2(M) \times [0,1]$ such that $i_0^*(\tilde{\psi})= \psi_0,  i_1^*(\tilde{\psi})= \psi_1$.  The differential form 
$$\tilde{P}=\tilde{P}^{prel} + d  \tilde{\psi}$$
 also satisfies condition (\ref{fixing-class-condition-iso}) .





\end{proof}

\subsection{Pertubative Chern-Simons Partition Function}

In this subsection we define the partition function of a $PS$-cycle as a formal power series on the formal variable $\hbar$ and the variable $x \in (H^*(M, \mathfrak{g})[1])^*$.

Applying  ( \ref{harmonic-isomorphism}) to the first factor of   $Identity \in H^*(M, \mathfrak{g})[1] \otimes (H^*(M, \mathfrak{g})[1])^*$ we obtain an element 
$$\mathfrak{x} \in \mathcal{A} \otimes \R[x] .$$

Let $G \in \mathfrak{G}$. 
For each $h \in H(G)$ denote by $\pi_h : M^{H(G)} \rightarrow M $  the projection on the factor $h$.  Set
$$ \mathfrak{g}_h = \pi_h^* (\mathfrak{g}) $$
and let $\mathfrak{g}_{H(G)}$ be the bundle over $M^{H(G)}$ defined as
$$ \mathfrak{g}_{H(G)} = \bigwedge (\bigoplus_{h \in H(G)} \mathfrak{g}_h) .$$



Let $U$ be an open set of $M^{H(G)}$ with the following property:  
\begin{equation} \label{close-points-trace}
 (z_h)_h \in U \Longrightarrow Diameter ( \{ z_h \}_{h \in H_v}) < \text{ injective radius  }, \forall v \in V(G).
 \end{equation}

Consider the trace on $\mathfrak{g}_{H(G)} $ 
\begin{equation} \label{trace}
Tr : \mathfrak{g}_{H(G)}  \rightarrow C_0(U, \mathfrak{o}(H(G)))
\end{equation}
defined by the following conditions:
\begin{itemize}
\item
$$ \text{Tr}(\bigwedge_{s \in S} X_{s}) =0$$
when $X_s \in  \mathfrak{g}_{h(s)}$ for some $h: S \rightarrow H(G)$ which is not a bijection;
\item
$$ \text{Tr}(\bigwedge_{h \in H(G)} X_h) = \prod_{v \in V(G)} \text{Tr}(\prod_{h \in H_v}^{cyclic}X_h). $$
Here, for each $v$, the order of the  elements $( X_h )_{h \in H_v}$ in the argument of $\text{Tr}$ is taken compatible with the cyclic order of $H_v$. 
For each point of $U$,  and $v \in V(G)$,  we use condition ( \ref{close-points-trace}) and a local trivialization of $\mathfrak{g}$  to identify $X_h$ with an element of the Lie-algebra for each $h \in H_v$.
\end{itemize}

\subsubsection{Partition Function}

Let $P$ be a propagator as in Lemma \ref{construction-propogator}.
As in \cite{AS1}, in order to make the signs simpler it is useful to introduce the super-propagator $P_s$ defined as the image of $P$ by the inclusion
$$ \pi_1^*(\mathfrak{g}) \otimes \pi_2^*(\mathfrak{g})   \rightarrow \bigwedge(\pi_1^*(\mathfrak{g}) \oplus \pi_2^*(\mathfrak{g})) . $$
Property $(\ref{parity})$  implies 
$$ tw^* (P_s) =P_s .$$

For $e \in E(G) \setminus E_l$ let $Conf_e(M)$  be he compactification of the configuration space of two points labeled by the half-edges of $e$. Define
$$Conf_{G,m}(M)= M^{H_l}   \times  \prod_{e \in E(G) \setminus E_l}Conf_e(M)  .$$
This space comes with the projections
$$ \pi_e : Conf_{G,m}(M)  \rightarrow  Conf_{h_1,h_2}(M) \text{    for   } e=\{h_1,h_2\} \in E^{in}(G) \setminus E_l,$$
$$ \pi_e : Conf_{G,m}(M)  \rightarrow  M^e \text{    for   } e \in E_l,$$
where $M^e \cong M^2$ if $e \in E^{in}$, $M^e \cong M$ if $e \in E^{ex}$.

To a chain $C \in C_*(M^{H(G)})$ which satisfies the transversality condition corresponds  a chain $\hat{C} \in C_*(Conf_{G,m}(M))$, which is defined blowing up $C$ along the diagonals. In the boundary of $ \hat{C}$ there are  extra boundary faces associated to the boundary of $Conf_{G,m}(M)$:
\begin{equation} \label{boundary-blow-up}
\partial \hat{C} = \widehat{\partial C} - \sum_{e } S(C \cap Diag_e),
\end{equation}
where $S(C \cap Diag_e)$ is the pre-image of $C \cap Diag_e$ from the map $Conf_{G,m}(M) \rightarrow M^{H(G)}$.

For $e \in E(G) \setminus E_l$, 
$$ \pi_e^*(P_s) \in \Omega^2(Conf_{G,m}(M), \wedge( \mathfrak{g}_h \oplus \mathfrak{g}_{h'}))   .$$

Assume 
\begin{equation} \label{condition-m}
m=\{ E_0,E_1,...,E_l  \} \text{   with     }|E_i|= |E_{i-1}| +1. 
\end{equation}
 So there exists $e_i \in E^{in}$, $1 \leq i \leq l $, such that $E_i= E_0 \sqcup \{e_1, . . . ,e_i   \}$.


To the graph $(G,m)$ we associate the differential form 
$$\underline{\Omega}_{G,m} \in  \R[x] \otimes \Omega^*(Conf_{G,m}(M), \mathfrak{g}_{H(G)})   $$
defined by
\begin{equation} \label{graph}
\underline{\Omega}_{G,m}   = (-1)^{|m|} \bigwedge_{e \in E(G) \setminus E_l} \pi_e^*(P_s) \wedge \bigwedge_{i} \pi_{e_i}^*(dP_s) \wedge \bigwedge_{e \in E^{ext}(G)}  \pi_e^*( \mathfrak{x} )  
\end{equation}

Applying  the trace (\ref{trace})  set
\begin{equation} \label{omega-graphs}
    \Omega_{G,m} :=  \text{Tr} ( \underline{\Omega}_{G,m} )\in   \R[x] \otimes \Omega^*(U, \mathfrak{o}(H(G))).    
\end{equation}




The partition function of the $PS$-cycles $Z$ is defined as 
\begin{equation} \label{partition-definition}
    \mathfrak{P} (Z) :=  \sum_{G,m}  \hbar^{- \chi(G)}  N^{d(G)} \langle   \Omega_{G,m} ,   Z_{G,m} \rangle  .
\end{equation}
in the last formula we have denoted by $ \langle   \omega ,   C \rangle \in \R[x]$ the natural integration pairing between $\omega \in     \R[x] \otimes \Omega^*(U) \otimes \mathfrak{o}(H(G))$ and $C \in C_*(U, \mathfrak{o}(H(G)))$ .


$\mathfrak{P} (Z)$ defines a formal Laurent series in $ \R((\hbar))[[x]]  $. Write
$$ \mathfrak{P} (Z) = \sum_{l,k} \hbar^l p_{k,l}(x) $$
where,   for each $k \geq 0$   , $p_{k,l}(x)$ is an homogeneus polynomial of degree $k$,  The stability condition on the graphs  tells us
$$ p_{k,l} \neq 0 \Rightarrow k + 3l \geq 0. $$




\subsubsection{Isotopies}

Take an isotopy of propagators $\tilde{P}$ and $\tilde{\mathcal{A}}$ as in section  \ref{parameter-propagator-section}.  As in (\ref{graph}), to each  graph $(G,m)$ we associate the differential form
$$\tilde{\Omega}_{G,m} \in  \R[x] \otimes \Omega^*(Conf_{G,m}(M) \times \R, \mathfrak{g}_{H(G)})   $$
defined by
\begin{equation} \label{graph-isotopy}
\tilde{\underline{\Omega}}_{G,m}   = (-1)^{|m|} \bigwedge_{e \in E(G) \setminus E_l} \tilde{\pi}_e^*(\tilde{P}_s) \wedge \bigwedge_{i} \tilde{\pi}_{e_i}^*(d \tilde{P}_s) \wedge \bigwedge_{e \in E^{ext}(G)}  \tilde{\pi}_e^*( \tilde{\mathfrak{x}} )  
\end{equation}
where now 
$$ \tilde{\pi}_e : Conf_{G,m}(M) \times \R  \rightarrow  Conf_e(M) \times \R \text{    for   } e \in E^{in}(G) \setminus E_l,$$
$$ \tilde{\pi}_e : Conf_{G,m}(M) \times \R  \rightarrow  M^e \times \R   \text{    for   } e \in E_l .$$

Applying  the trace (\ref{trace})   we set
\begin{equation} \label{omega-graphs-isotopy}
    \tilde{\Omega}_{G,m} :=  \text{Tr} ( \tilde{\underline{\Omega}}_{G,m}   )\in   \R[x] \otimes \Omega^*(U \times \R, \mathfrak{o}(H(G)).    
\end{equation}


 To an isotopy $ \tilde{Z} = (\tilde{Z}_{G,m})_{G,m}$ of  $PS$-cycles it is associated the isotopy of partition functions defined by
\begin{equation} \label{partition-definition-isotopies}
    \mathfrak{P} (\tilde{Z}) :=  \sum_{G,m}  \hbar^{- \chi(G)}  N^{d(G)} \langle   \tilde{\Omega}_{G,m} ,   \tilde{Z}_{G,m}  \rangle .
\end{equation}

$ \mathfrak{P} (\tilde{Z})$  defines a formal Laurent series with coefficients piecewise smooth differential forms on $\R$. We may write
$$ \mathfrak{P} (\tilde{Z}) = \sum_{l,k} \hbar^l \tilde{p}_{k,l}(x) $$
where now $\tilde{p}_{k,l}(x)$ is an homogeneus polynomial of degree $k$ with coefficients in piecewise smooth differential forms on $\R$.  The stability condition on the graphs  tells us
$$ \tilde{p}_{k,l} \neq 0 \Rightarrow k \geq -3l . $$



We may write
$$ \mathfrak{P} (\tilde{Z})= \mathfrak{Q} (\tilde{Z}) + \mathfrak{K} (\tilde{Z}) d \lambda , $$
where $\mathfrak{Q} (\tilde{Z})$ and $\mathfrak{K} (\tilde{Z}) $ are formal Laurent series with coefficients piecewise smooth functions on $\R$.

$ \mathfrak{Q} (\tilde{Z})$ may be consider as the function
$$\mathfrak{Q} (\tilde{Z})= \{ \lambda \mapsto \mathfrak{P} ({Z}^{\lambda}) \}$$
defined for generic value of $\lambda$.
Here $ \underline{{Z}}^{\lambda}$  is  the $PS$-cycle defined by the restriction of $\tilde{Z}$ to $\lambda$.

\begin{lemma}
$ \mathfrak{Q} (\tilde{Z})$ is a continuos function on the variable $\lambda$. 
\end{lemma}
\begin{proof}
Let $\lambda_0 \in \R$. Take generic $\lambda^-,\lambda^+ \in \R$  close to $\lambda_0$,  with $\lambda^- < \lambda_0 < \lambda^+$. We may assume that $  \underline{{Z}}_{G,m}^{\lambda^+}$ and $   \underline{{Z}}_{G,m}^{\lambda^-} $ are defined.

Set 
$$\underline{\tilde{Z}}_{G,m}^{[\lambda^-, \lambda^+]} = pr_*(\underline{\tilde{Z}}_{G,m}|_{ [\lambda^-, \lambda^+]}).$$  

The cycle condition for $\tilde{Z}$ implies
$$  \underline{{Z}}_{G,m}^{\lambda^+} -  \underline{{Z}}_{G,m}^{\lambda^-}  = \partial  \underline{\tilde{Z}}_{G,m}^{[\lambda^-, \lambda^+]} -\sum  \delta_{e'} \underline{\tilde{Z}}_{G',m}^{[\lambda^-, \lambda^+]}- \sum_i (-1)^i \underline{\tilde{Z}}_{G,\partial_i m}^{[\lambda^-, \lambda^+]}.$$

We also have
$$ \langle  \underline{{\Omega}}_{G,m} ,  \partial  \underline{\tilde{Z}}_{G,m}^{[\lambda^-, \lambda^+]}   \rangle  =  \langle  d \underline{{\Omega}}_{G,m} ,    \underline{\tilde{Z}}_{G,m}^{[\lambda^-, \lambda^+]}   \rangle + \langle  \underline{{\Omega}}_{G,m} ,  \sum_e  \delta_{e} \underline{\tilde{Z}}_{G,m}^{[\lambda^-, \lambda^+]}   \rangle, $$
where  the last term comes from formula ( \ref{boundary-blow-up}) and property  (\ref{propagator-singularity}) of the propagator.

Making the sum over the graphs we obtain
$$\mathfrak{P}(Z^{\lambda^+})  - \mathfrak{P}(Z^{\lambda^-})=   \sum_{G,m}  \hbar^{- \chi(G)}  N^{d(G)} (  \langle  d \underline{{\Omega}}_{G,m} ,    \underline{\tilde{Z}}_{G,m}^{[\lambda^-, \lambda^+]}   \rangle -      \langle   \underline{{\Omega}}_{G,m} , \sum_i (-1)^i \underline{\tilde{Z}}_{G,\partial_i m}^{[\lambda^-, \lambda^+]} \rangle ).$$
Note that the terms with $\delta_e$ are cancelled. It is easy to see that each term in the sums goes to zero when $\lambda^-,\lambda^+ \rightarrow \lambda_0$.


\end{proof}

\begin{proposition} \label{QME-proposition}
$\mathfrak{P} (\tilde{Z})$ satisfies the $QME$:
\begin{equation} \label{master-equation}
    d_{\lambda} \mathfrak{P} (\tilde{Z}) + \hbar \Delta \mathfrak{P} (\tilde{Z}) =0
\end{equation}
\end{proposition}

\begin{proof}


Applying the formula for the differential of the push-forward of a differential form we obtain the identity
\begin{equation} \label{differential-pairing-isotopy}
d \langle  \underline{\tilde{\Omega}}_{G,m} ,   \underline{\tilde{Z}}_{G,m}  \rangle =(-1)^{H(G)+ |m|}( \langle  d \underline{\tilde{\Omega}}_{G,m} ,   \underline{\tilde{Z}}_{G,m}  \rangle -
 \langle   \underline{\tilde{\Omega}}_{G,m} , \partial   \underline{\tilde{Z}}_{G,m}  \rangle +
\sum_{e \in E^{in}(G) \setminus E_l} \langle  \underline{\tilde{\Omega}}_{\delta_e G,m} ,  \delta_e \underline{\tilde{Z}}_{G,m}  \rangle ).
\end{equation}
In the last term we have used formula ( \ref{boundary-blow-up}) and property  (\ref{propagator-singularity}) of the propagator.

Let $|m|=l$. For $0< i < l$, define $sw_i(m)$ from $m$ switching  $e_i$ with $e_{i+1}$. Since $\partial_i m = \partial_{i} sw_i(m)$ and $\underline{\tilde{\Omega}}_{G,m}= - \underline{\tilde{\Omega}}_{G,sw_i(m)} $ we have  
\begin{equation} \label{twist-cancellation}
\langle  \underline{\tilde{\Omega}}_{G,m} ,   \underline{\tilde{Z}}_{G, \partial_i m}  \rangle  +  \langle  \underline{\tilde{\Omega}}_{G,sw_i(m)} ,    \underline{\tilde{Z}}_{G, \partial_i (sw_i(m))}  \rangle =0 .
\end{equation}

It is immediate to check
\begin{equation} \label{differential-omega-relation}
d   \underline{\tilde{\Omega}}_{G,m}  = (-1)^{|m|} \sum_{m' | \partial_{l+1} m' =m }  \underline{\tilde{\Omega}}_{G,m'}  
\end{equation}

Using   (\ref{differential-pairing-isotopy}) ,
 (\ref{twist-cancellation}),  (\ref{differential-omega-relation}) and the cycle condition for $\tilde{Z}$ we obtain 
\begin{equation} \label{quantum-master-preliminary}
d \mathfrak{P}(\tilde{Z}) = \sum_{G,m}   (-1)^{|H(G)| + |m|+1} \hbar^{-\chi(G)}   \langle   \underline{\tilde{\Omega}}_{G,m} ,    \underline{\tilde{Z}}_{G, \partial_0 m}  \rangle .
\end{equation}
Note that the terms in $\delta_e, \delta_v$ are cancelled. 

It is easy to see that the RHS of  (\ref{quantum-master-preliminary}) equal to $ - \hbar \Delta  \mathfrak{P}(\tilde{Z})$. This follows from the identity
\begin{equation} \label{delta-omega-relation}
 \Delta   \underline{\tilde{\Omega}}_{G',m'}  =(-1)^{E^{ex}(G')+|m'|} \sum_{ (G,m)| (G, \partial_0 m )=(G',m') }  \underline{\tilde{\Omega}}_{G,m} .
 \end{equation}
The factor $\hbar$ appears in (\ref{master-equation})  since $\chi(G') = \chi(G) +1$.
 


\end{proof}

\section{Coherent Point-Splitting-Cycles}

 Denote by $ \mathfrak{Fr}(M  )$ the set of  frames  of $M$.
Given  $\mathbf{fr}  = (\mathbf{fr}_1 , \mathbf{fr}_2  , \mathbf{fr}_3 ) \in \mathfrak{Fr}(M  )$ we say that un Euler Structure $U$ is compatible with $\mathbf{fr}   $ if there exists $\mathbf{a}=(a_1,a_2,a_3) \in \R^3 \setminus \{ 0,0,0  \}$    such that  
$U =  a_1 \mathbf{fr}_1 +  a_1 \mathbf{fr}_2  +  a_3 \mathbf{fr}_3 $.
Clearly the space of Euler Structures compatible with  $\mathbf{fr}$ is homeomorphic to $ \R^3 \setminus \{ 0,0,0  \}$. This tells us that each pair of Euler Structures compatible with  $\mathbf{fr}$ are isotopic with isotopy canonically  determined up to isotopy of isotopies, and so,
Lemma \ref{transfer-annulus} yields a canonical identification of
$PSH(  M|  U ) $ for different choices of $U$  compatible with  $\mathbf{fr}$. Let $PSH(M |   \mathbf{fr})$ denote this vector space.


The transfer map (\ref{transfer-map-annulus}) yields the transfer map 
$$ Transfer_{\tilde{\mathbf{fr}}}:  PSH(M |   \mathbf{fr}^0) \rightarrow PSH(M |   \mathbf{fr}^1)  $$
for any $\tilde{\mathbf{fr}}$  isotopy between  $\mathbf{fr}^0$ and $\mathbf{fr}^1 \in \mathfrak{Fr}(M  )$

\begin{theorem}
To each $\mathbf{fr} \in \mathfrak{Fr}(M  )$ it is associated canonically a $PS$-homology class 
$$ [Z^{\mathbf{fr} }   ] \in  PSH(M |   \mathbf{fr}).$$

For $\mathbf{fr}^0,\mathbf{fr}^1   \in \mathfrak{Fr}(M  )  $ and $\tilde{\mathbf{fr}}$ isotopy between   $\mathbf{fr}^0$ and $\mathbf{fr}^1$ we have
$$Transfer_{\tilde{\mathbf{fr}}}( [Z^{\mathbf{fr}^0 }   ]  )   =   [Z^{\mathbf{fr}^1 }   ]$$
\end{theorem}

To prove the theorem we introduce the notion of flat-approaching $PS$-cycle. We first define flat $PS$-cycles and show that there exists a special flat $PS$-cycle  canonically defined up to isotopy (see Proposition \ref{existence-flat-proposition}  and Proposition \ref{unicity-flat-proposition}) . 
Given a flat $PS$-cycle $Z^{\flat}$ and a frame $\mathbf{fr}$, we define the notion of flat-approaching $PS$-cycle $Z^{\wr}$ converging to $(Z^{\flat}, \mathbf{fr})$ . We show that for each pair $(Z^{\flat}, \mathbf{fr})$ there exists a  flat-approaching $PS$-cycle  $Z^{\wr}$ converging to $(Z^{\flat}, \mathbf{fr})$ (Proposition \ref{existence-asymptotic-proposition} ) and moreover it is unique  up to isotopies
(Proposition \ref{unicity-asymptotic-proposition}). 

These results yield  a  flat-approaching $PS$-cycle $Z^{\wr}$ canonically defined up to isotopy. The coherent $PS$-cycle $Z^{\mathbf{fr} } $ is defined  from $Z^{\wr}$ as
$$Z^{\mathbf{fr}}_{(G,m)} = Z^{\wr}_{G,m , \mathcal{P}^{max}}$$
where $ \mathcal{P}^{max}$ is the maximal partition of $G$, i.e., the trivial partition $ \mathcal{P}^{max}= \{ V(G) \}$.


\subsection{Flat $PS$-cycles}

We denote by $\mathfrak{G}^{part}$  the set of pairs $(G, \mathcal{P})$ where $G \in \mathfrak{G}$ and $ \mathcal{P}$ is a partition of $V(G)$.
We denote by $\mathfrak{G}^{part}_*$ the set of triples $(G, m,\mathcal{P})$ with $(G,m)  \in  \mathfrak{G}_*$ and $ (G, \mathcal{P}) \in \mathfrak{G}^{part}$.

On the set of partitions we consider the usual partial order, i.e., $\mathcal{P'} \preccurlyeq \mathcal{P}$  if  the following condition is satisfied: 
$$v_1 \sim_{\mathcal{P}'} v_2 \Rightarrow v_1 \sim_{\mathcal{P}} v_2 .$$


\begin{definition}
Denote by $G_0$ the decorated graph defined in subsection \ref{annulus-data-section}.

A \emph{flat annulus data} $Z^{\flat, \circledcirc}$   is defined by a Borel-Moore chain 
\begin{equation} \label{flat-annulus-chain}
\underline{Z}_{G_0}^{\flat , \circledcirc} \in  C_*(\mathfrak{L}_{\hat{V}(G_0)}(\R^3) , \mathfrak{o}(H(G_0)))^{\text{Aut}(G_0)}  .
\end{equation}
 close in the $C^0$ topology to $\underline{Z}_{Disk0}$ and invariants  (up to triangulations) under the  translations  of  $\R^3$ .
\end{definition}

Note that comparing with the definition of Subsection \ref{annulus-data-section} here we are setting $  \underline{Z}_{G_1}^{\flat, \circledcirc}  =   \underline{Z}_{G_2}^{\flat, \circledcirc} =0$.




Now fix a flat annulus data $Z^{\flat, \circledcirc}$.
A \emph{ flat $PS$-cycle} $Z^{\flat}$, with annulus data  $Z^{\flat, \circledcirc}$, consists in a collection of Borel-Moore chains 
\begin{equation} \label{flat-complex}
   Z^{\flat}= ( Z^{\flat}_{G,m, \mathcal{P}} )_{(G,m, \mathcal{P} ) \in \mathfrak{G}^{part}_*}
\end{equation} 
$$Z^{\flat}_{G,m, \mathcal{P}} \in C_{|H(G)| + |m|}((\R^3)^{H(G)} , \mathfrak{o}_{H(G)}  )^{\text{Aut}(G,m, \mathcal{P})}$$ 
such that :
\begin{itemize}
\item  transversality condition is satisfied;
\item   forgetful compatibility condition is satisfied;
\item cycle condition is satisfied;
\item  support condition is satisfied;
\item  cluster condition is satisfied;
\item  translation condition is satisfied.
\end{itemize}

We denote by $\mathcal{Z}^{\flat}(Z^{\flat, \circledcirc})$ the space of flat $PS$-cycles with annulus data $Z^{\flat, \circledcirc}$ .

We use the notation 
\begin{equation} \label{flat-G}
   Z^{\flat}_{G,m}= ( Z^{\flat}_{G,m, \mathcal{P}} )_{ \mathcal{P} }.
\end{equation} 

Below we explains the conditions appearing in the definition.

\subsubsection{Transversality Condition}

For each $h \in H(G)$ denote by $v_h$ the vertex to which $h$ is attached.
Set 
$$E^{in}(G, \mathcal{P})  = \sqcup_j \{ \{h,h' \}  \in E^{in}(G) | v_{h} \sim_{\mathcal{P}} v_{h'} \} ,   \quad E(G, \mathcal{P}) = E^{in}(G, \mathcal{P}) \sqcup E^{ex}(G).$$

We say that the transversality condition is satisfied for $(G,m, \mathcal{P})$ if
$Z^{\flat}_{G,m, \mathcal{P}}$ is transversal  to   $ \bigcap_{e \in E'} \pi_e^{-1}(Diag)$ for each subset $E' \subset E^{in }(G , \mathcal{P})  \setminus E^l$;

\subsubsection{Forgetful Compatibility Condition}

For $G \in \mathfrak{G}$, the multi-loop space  $\underline{\digamma}_{G}$  is defined as the set of arrays $({w}_v , t_v)_{v \in V(G)}  $  where   
$${w}_v : \mathbb{S}^1 \rightarrow \R^3, t_v : H_v  \hookrightarrow \mathbb{S}^1, \quad \forall v \in V(G).$$  
We require that $t_v$ is compatible with the cyclic ordering of $H_v$ .






We say that   (\ref{flat-complex}) satisfies the  forgetful compatibility condition if there exists a collections of Borel-Moore chains 
\begin{equation} \label{flat-complex-loop}
    (\underline{Z}^{\flat}_{G,m, \mathcal{P}} )_{(G,m, \mathcal{P}) \in \mathfrak{G}^{part}_*}
\end{equation}
$$\underline{Z}^{\flat}_{G,m, \mathcal{P}} \in C_{|H(G)| + |m|}(\underline{\digamma}_{G} , \mathfrak{o}_{H(G)}  )^{\text{Aut}(G,m, \mathcal{P})}$$
such that
\begin{equation} \label{chains-ev-flat}
ev_*( \underline{Z}^{\flat}_{G,m, \mathcal{P}}    )={Z}^{\flat}_{G,m, \mathcal{P}}
\end{equation}

For each external edge $e \in E^{ex}(G)$ we have:
\begin{itemize}
\item If $G$ is stable after removing $e $  
    \begin{equation} \label{flat-stable-forget-condition}    
     \underline{Z}^{\flat}_{G,m, \mathcal{P}} = \underline{Z}^{\flat}_{forget_e(G,m), \mathcal{P}'} \times_{\textbf{forget}_e} \underline{\digamma}_{G}   .
     \end{equation}
      \item in the case $G$ is unstable after removing $e $ and $e$ is removable  we require
 \begin{equation} \label{flat-removible-forget-condition}
 \underline{Z}^{\flat}_{G,m, \mathcal{P}} = \underline{Z}^{\flat}_{forget_e(G,m), \mathcal{P}'} \times \underline{Z}_{G_e}^{\flat, \circledcirc} .
 \end{equation}     
  when $G_e \in Ann0$, and
 \begin{equation} \label{flat-removible-forget-condition2}
 \underline{Z}^{\flat}_{G,m} = \underline{Z}^{\flat}_{forget_e(G,m)} \times \underline{Z}^{\flat}_{Disk0}  
 \end{equation}  
 when $G_e= Disk0$.
     \item In the case  $e$ is not removible we require that $ \underline{Z}^{\flat}_{G,m, \mathcal{P}}$ is close in the $C^1$-topology to $ \underline{Z}^{\flat}_{forget_e(G,m), \mathcal{P}'} \times \underline{Z}_{Disk0}^{\flat, \circledcirc}$:
 \begin{equation} \label{flat-unstable-forget-condition}
 \underline{Z}^{\flat}_{G,m, \mathcal{P}} \approx \underline{Z}^{\flat}_{forget_e(G,m), \mathcal{P}} \times \underline{Z}_{Disk0}^{\flat} .
 \end{equation}
  \end{itemize}


In the notation above we considered any partition $\mathcal{P}$ of a graph $G$ as a partition also of $forget_e(G)$ using the obvious correspondence.


\subsubsection{Cycle Condition} 

We say that the cycle condition is satisfied if 
\begin{equation} \label{cycle-condition-loop-flat}
 \partial  \underline{Z}_{G, \mathcal{P},m}^{\flat}  =   \sum_i (-1)^i \underline{Z}_{G,  \mathcal{P},\partial_i m}^{\flat} + \sum_{e'}  \delta_{e'} \underline{Z}_{G', \mathcal{P}',m'}^{\flat}.
 \end{equation}
for each $(G,m, \mathcal{P}) \in \mathfrak{G}_*^{part}$.

\subsubsection{Support Condition}


Let $G \in \mathfrak{G}$ and  $\rho >0$ be a positive real number. Define   $ \underline{\digamma}_{G }^{< \rho}$ as the set of 
 $( w_v , t_v)_v \in  \underline{\digamma}_{G}   $  such that
$$ Diameter ( \sqcup_{v \in V_c} w_v(\mathbb{S}^1) ) < \rho$$
for each $c \in  Comp_0$  .

We say that the support condition is satisfied for $(G,m, \mathcal{P})$ if there exists  a real positive number $\rho^{\dagger}$ such that 
$$ Supp(  \underline{Z}_{G,m, \mathcal{P}}^{\flat} )  \subset  \digamma_{G}^{<\rho^{\dagger}}.$$

\subsubsection{Cluster Condition}

To a partition $\mathcal{P}$ of $V(G)$ correspond a submanifold $Diag_{\mathcal{P}}^{\flat} \cong (\R^3)^{\mathcal{P}}$ of $(\R^3)^{V(G)}$, which is the diagonal associated to $\mathcal{P}$ .
We have
$$\mathcal{P}' \prec \mathcal{P} \Longleftrightarrow Diag_{  \mathcal{P}  }^{\flat} \subset Diag_{\mathcal{P}'}^{\flat}  .$$




\begin{definition}

Let $\mathcal{P} $ be a partition of $G \in \mathfrak{G}$. For  $  \mathcal{P}^{\triangleright} $ with  $\mathcal{P}  \prec  \mathcal{P}^{\triangleright} $ and $\eta>0$,  
$ Cut_{\mathcal{P}^{\triangleright} , \eta} ( Diag_{\mathcal{P} }^{\flat})$ is defined as the set of $\mathbf{z}  \in Diag_{\mathcal{P} }^{\flat}$ such that
 $d ( \mathbf{z}, Diag_{\mathcal{P}^{\triangleright  \prime}}^{\flat}  ) > \eta  $ for each $\mathcal{P}^{\triangleright \prime}$ with $\mathcal{P}  \prec \mathcal{P}^{\triangleright\prime} \preccurlyeq \mathcal{P}^{\triangleright} $.
\end{definition}

For $\mathbf{z} = (z_{v})_{v \in V(G) } \in (\R^3)^{V(G)}$, 
denote by $\underline{D}_{\rho, \mathbf{z}}^{\flat}$
 the set of elements $({w}_v,t_v)_v  \in  \underline{\digamma}_{G,\mathcal{P}}$
such that 
\begin{equation} \label{shring-condition-loop0-flat}
{w}_v(\mathbb{S}^1) \subset D_{\rho, z_{ v}} \quad \forall  v \in V(G).
\end{equation}

For $\mathcal{P} \prec \mathcal{P}^{\triangleright} $,
for each $\rho>0$ there exists  $\eta$ big enough such that
\begin{equation} \label{disjoint-disks-nuclear-0}
  \mathbf{z} \in  Cut_{\mathcal{P}^{\triangleright} , \eta} ( Diag_{\mathcal{P} }^{\flat}),  v_1 \nsim_{\mathcal{P}^{\triangleright}} v_2 \Longrightarrow  {D}_{\rho, z_{ v_1}} \cap {D}_{\rho, z_{ v_2}} = \emptyset.
 \end{equation}
 
Property (\ref{disjoint-disks-nuclear-0}) implies that we have an inclusion of sets
\begin{equation} \label{inclusion-disks-partion-loop-flat}
     \underline{D}_{\rho, \mathbf{z}}^{\flat} \hookrightarrow   \underline{\digamma}_{G, \mathcal{P}^{\triangleright} } 
\end{equation}
which is compatible with the cycle condition in the sense that $\delta_e \neq 0$ on   the chains of $ \underline{D}_{\rho, \mathbf{z}}^{\flat}$ when  $e \notin  E(G, \mathcal{P})$.

\begin{definition}

For $(G,m, \mathcal{P}) \in \mathfrak{G}_*^{part} $, $\rho, \eta> 0$, 
we say that (\ref{flat-complex}) satisfies the  cluster condition for $(G,m, \mathcal{P}, \rho, \eta)$ if 
\begin{equation} \label{matching-nuclear}
\underline{Z}^{\flat}_{G,m, \mathcal{P}} \cap \underline{D}_{\rho ,  \mathbf{z}}^{\flat} = \underline{Z}^{\flat}_{G,m, \mathcal{P}^{\triangleright}}  \cap \underline{D}_{\rho , \mathbf{z}}^{\flat} ,
  \end{equation}
for each $\mathcal{P}^{\triangleright} $ with $\mathcal{P} \prec \mathcal{P}^{\triangleright} $ and $ \mathbf{z} \in  Cut_{\mathcal{P}^{\triangleright} , \eta} ( Diag_{\mathcal{P} }^{\flat}) $.

We say that the cluster condition is satisfied for
$(G,m, \mathcal{P} , \rho)$ if  there exists $\eta$ big enough such that the cluster condition is satisfied for
$(G,m,  \mathcal{P}, \rho, \eta)$ .
\end{definition}

The cluster condition  defines $\underline{Z}^{\flat}_{G,m, \mathcal{P}}$ outside a compact set in terms of the partitions  $\mathcal{P}' \prec \mathcal{P}$. It can be considered as    boundary conditions at infinity for  $\underline{Z}^{\flat}_{G,m, \mathcal{P}}$.

\subsubsection{Translation Condition}

There exists an obvious action on $\underline{\digamma}_{G} $ of the group $Transl_{\mathcal{P}}:=(\R^3)^{\mathcal{P}} $ by translation .   
This action yields an action on $ C_{|H(G)| + |m|}(\underline{\digamma}_{G} , \mathfrak{o}_{H(G)}  )^{\text{Aut}(G,m, \mathcal{P})}$.

We say that the Translation Condition is satisfied for $(G,m, \mathcal{P})$ if $\underline{Z}^{\flat}_{G,m, \mathcal{P}}$ is invariant by $Transl_{\mathcal{P}}$ (up to triangulations and reparametrizations).


\subsubsection{Isotopies}

Let $\tilde{Z}^{\flat, \circledcirc}$ be an isotopy between flat annulus data ${Z}^{\flat, \circledcirc,0}$ and  ${Z}^{\flat, \circledcirc,1}$. Take $Z^{\flat,0} \in \mathcal{Z}(M|  {Z}^{\flat, \circledcirc,0})$ and $Z^{\flat,1} \in \mathcal{Z}(M|  {Z}^{\flat, \circledcirc,1})$. An isotopy  
of flat $PS$-cycles
  between $Z^{\flat,0}$ and $Z^{\flat,1}$  consists in a collection of chains
\begin{equation} \label{collection-chains-isotopy-flat}
\tilde{Z}^{\flat}=  ( \tilde{Z}_{G,m, \mathcal{P}}^{\flat} )_{G,m, \mathcal{P}} ,
\end{equation}
with
$$\tilde{Z}_{G,m, \mathcal{P}}^{\flat} \in C_{|H(G)| + |m|+1} ( M^{H(G)} \times \R  , \mathfrak{o}_{H(G)})^{\text{Aut}(G,m, \mathcal{P})}  .$$
We require that
\begin{itemize}
\item there exists  $\lambda^{\dagger} >0$  big enough such that
$$ \tilde{Z}^{\flat, <-\lambda^{\dagger}} =  Z^{\flat,0}  \times  \R_{<-\lambda^{\dagger}}  , \quad \tilde{Z}^{\flat, >\lambda^{\dagger}} =  Z^{\flat,1}  \times \R_{>\lambda^{\dagger}}   ;$$

 \item 
the transversality condition is satisfied;


\item the forgetful compatibility condition is satisfied.

\item the cycle condition is satisfied;

\item the translation conditino is satisfied;

\item the support condition is satisfied;

\item the cluster condition is satisfied.


\end{itemize}

We denote by $\tilde{\mathcal{Z}}^{\flat}(  \tilde{Z}^{\flat, \circledcirc})$ the space of isotopies of flat $PS$-cycles. 



The notions of transversality condition, forgetful compatibility condition and cycle condition are parallel to the corresponding notions in section  \ref{isotopies-section}.

The translation condition is defined extending 
the action of $Transl_{\mathcal{P}}$ to   $\underline{\digamma}_{G,\mathcal{P}} \times \R  $ , where it acts trivially on the $\R$ factor.

The cluster condition is defined considering the one parameter version of  condition  (\ref{matching-nuclear}):
\begin{equation} \label{matching-nuclear-iso}
\underline{\tilde{Z}}^{\flat}_{G,m, \mathcal{P}} \cap(  \underline{D}_{\rho ,  \mathbf{z}}^{\flat} \times \R ) = \underline{\tilde{Z}}^{\flat}_{G,m,\mathcal{P}^{\triangleright}}  \cap (  {\underline{D}}_{\rho , \mathbf{z}}^{\flat}    \times \R ),
  \end{equation}
for each  $ \mathbf{z} \in  Cut_{\mathcal{P}^{\triangleright} , \eta} ( Diag_{\mathcal{P} }^{\flat}) $.
.

\subsubsection{Existence}

Denote by $G^{\heartsuit}$ the empty decorated graph, i.e., $Comp_0(G^{\heartsuit})= \emptyset$.


\begin{proposition} \label{existence-flat-proposition}
There exists a flat $PS$-cycle $ {Z}^{\flat}$ such that 
    $$  {\underline{Z}}^{\flat}_{G^{\heartsuit}} = 1$$
    $$H(G)= \emptyset,  \underline{Z}_G^{\flat} \neq 0 \Rightarrow G = G^{\heartsuit}.$$
    $$d(G) \neq 0 \Rightarrow     \underline{Z}_{G,m}^{\flat}=0.$$    
    \end{proposition}
The proposition follows by induction using the Lemmas below.




Set
$$ \underline{\digamma}_{G, \mathcal{P} }(\rho) = \bigsqcup_{ \mathbf{z} \in  Diag_{\mathcal{P} }^{\flat} }   \underline{D}_{ \rho, \mathbf{z}}^{\flat}  . $$
and
$$ \underline{\digamma}_{G, \mathcal{P} }(\mathcal{P} ', \rho', \eta') = \bigsqcup_{ \mathbf{z}'  \in Cut_{\mathcal{P} , \eta'} ( Diag_{\mathcal{P} '}^{\flat}) }   \underline{D}_{ \rho', \mathbf{z}'}^{\flat}  \subset  \underline{\digamma}_{G, {\mathcal{P}'} }  $$
for any $\mathcal{P}'$ with $\mathcal{P}' \prec \mathcal{P}$.




Relation (\ref{inclusion-disks-partion-loop-flat}) tells us that for $\eta' $ big enough there is a natural inclusion of sets
\begin{equation} \label{inclusion-partion-loop-flat}
     \underline{\digamma}_{G, \mathcal{P} }(\mathcal{P} ', \rho', \eta') \hookrightarrow   \underline{\digamma}_{G, \mathcal{P} } . 
\end{equation}
 compatible with the cycle condition.

\begin{lemma} \label{cluster-flat-lemma}

Fix $(G,m) \in \mathfrak{G}_*$. Let $\rho^{\dagger}$ be a real positive number. 
Assume that $\underline{Z}_{G,m, \mathcal{P}}^{\flat}$ satisfies the support condition  with radius $\rho^{\dagger}$ for each partition $\mathcal{P}$.


If $(\underline{Z}_{G,m, \mathcal{P}}^{\flat})_{\mathcal{P}}$ satisfies the cluster condition there exists a collection of real positive numbers  $(\rho_{\mathcal{P}}, \eta_{\mathcal{P}})_{\mathcal{P}}$ such that

\begin{itemize}

\item  the cluster condition is satisfied for
\begin{equation}  \label{matching-inductive-lemma-flat}
    (G,m, \mathcal{P} , \rho_{ \mathcal{P}}, \eta_{ \mathcal{P}});
\end{equation}

Set
\begin{equation} \label{set-U-nuclear}
    U^{\flat , cluster}_{\mathcal{P}}=  \bigsqcup_{\mathcal{P}' \prec \mathcal{P}}   \underline{\digamma}_{G, {\mathcal{P}} }(\mathcal{P}', \rho_{ \mathcal{P}'},   \eta_{\mathcal{P}'} )     .
\end{equation}

\item
\begin{equation} \label{covering-boundary-condition-nuclear}
  \underline{ \digamma}_{G}^{< \rho_{G, m}^{\dagger}} \subset  \underline{\digamma}_{G, \mathcal{P} }(\rho_{\mathcal{P}})  \sqcup  U^{\flat, cluster} ;
\end{equation}


\item 

   \begin{equation} \label{not-overlapping-condition-nuclear}
     \underline{\digamma}_{G, {\mathcal{P}} }(\mathcal{P}_1, \rho_{\mathcal{P}_1},   \eta_{\mathcal{P}_1} ) \cap \underline{\digamma}_{G, {\mathcal{P}} }(\mathcal{P}_2, \rho_{ \mathcal{P}_2},   \eta_{\mathcal{P}_2} ) 
   = \emptyset
    \end{equation}   
when $\mathcal{P}_1 \prec \mathcal{P}, \mathcal{P}_2 \prec \mathcal{P}, \mathcal{P}_1 \npreceq \mathcal{P}_2 , \mathcal{P}_2 \npreceq \mathcal{P}_1$.

 

\item \begin{equation} \label{disjoint-disks-nuclear2}
  \mathbf{z} \in Cut_{\mathcal{P}^{\triangleright}, \eta_{\mathcal{P}}} ( Diag_{\mathcal{P}}^{\flat}), v_1 \sim_{\mathcal{P}^{\triangleright}} v_2 ,   v_1 \nsim_{\mathcal{P}} v_2 \Longrightarrow  {D}_{\rho_{ \mathcal{P}}, z_{ v_1}} \cap {D}_{\rho_{\mathcal{P}}, z_{ v_2}} = \emptyset .
\end{equation}
for each $\mathcal{P}^{\triangleright}$ with $\mathcal{P} \prec \mathcal{P}^{\triangleright}$.

\end{itemize}

Vice versa ,  the cluster condition is satisfied if  there exists  a collection of positive real numbers $(\rho_{\mathcal{P}}, \eta_{\mathcal{P}})_{\mathcal{P}}$ such that above conditions are satisfied. 


\end{lemma}
\begin{proof}
Let $\mathcal{P}$ be a partition of $G$ and
assume we have defined the real numbers $(\rho_{\mathcal{P}'}, \eta_{\mathcal{P}'})_{\mathcal{P}' \prec \mathcal{P}}$.
Take $\rho_{\mathcal{P}}$ big enough such that condition (\ref{covering-boundary-condition-nuclear}) is satisfied,
then take $\eta_{\mathcal{P}}$ big enough such that conditions  (\ref{matching-inductive-lemma-flat}) , (\ref{not-overlapping-condition-nuclear}) ,  (\ref{disjoint-disks-nuclear2}) are satisfied. 

Now assume we have a collection of real positive numbers $(\rho_{\mathcal{P}}, \eta_{\mathcal{P}})_{\mathcal{P}}$ such that the properties above hold.

Let $ \mathcal{P}^{\triangleright}$ with $\mathcal{P} \prec \mathcal{P}^{\triangleright}$, and let $\rho >0$. Take $\eta$ big enough such that  
\begin{equation} \label{cluster-sets-compatibility}
  \underline{D}_{\rho,\mathbf{z}} \cap   \underline{\digamma}_{G, {\mathcal{P}} }({\mathcal{P}'}, \rho_{{\mathcal{P}'}},   \eta_{{\mathcal{P}'}} ) =    \underline{D}_{\rho,\mathbf{z}} \cap   \underline{\digamma}_{G, {{\mathcal{P}^{\triangleright}}} }({\mathcal{P}'}, \rho_{ {\mathcal{P}'}},   \eta_{{\mathcal{P}'}} ) .
  \end{equation}
 for each $\mathcal{P}' \prec \mathcal{P}$ and 
$ \mathbf{z} \in Cut_{\mathcal{P}^{\triangleright}, \eta } Diag_{\mathcal{P}} $ .

The cluster condition for $\mathcal{P}'$ tells us
$$ \underline{Z}_{G,m, \mathcal{P}}^{\flat} \cap   \underline{\digamma}_{G, {\mathcal{P}} }({\mathcal{P}'}, \rho_{ {\mathcal{P}'}},   \eta_{{\mathcal{P}'}} ) =    \underline{Z}_{G,m, \mathcal{P}'}^{\flat} \cap   \underline{\digamma}_{G, {\mathcal{P}} }({\mathcal{P}'}, \rho_{ {\mathcal{P}'}},   \eta_{{\mathcal{P}'}} ) $$
and
$$ \underline{Z}_{G,m, \mathcal{P}^{\triangleright}}^{\flat} \cap    \underline{\digamma}_{G, {{\mathcal{P}^{\triangleright}}} }({\mathcal{P}'}, \rho_{{\mathcal{P}'}},   \eta_{{\mathcal{P}'}} ) =    \underline{Z}_{G,m, \mathcal{P}'}^{\flat} \cap    \underline{\digamma}_{G, {{\mathcal{P}^{\triangleright}}} }({\mathcal{P}'}, \rho_{\mathcal{P}'},   \eta_{{\mathcal{P}'}} ), $$
and so from ( \ref{cluster-sets-compatibility}) we obtain
\begin{equation} \label{cluster-lemma-basic1}
\underline{Z}_{G,m, \mathcal{P}}^{\flat} \cap  \underline{D}_{\rho,\mathbf{z}} \cap   \underline{\digamma}_{G, {\mathcal{P}} }({\mathcal{P}'}, \rho_{ {\mathcal{P}'}},   \eta_{{\mathcal{P}'}} ) =  \underline{Z}_{G,m, \mathcal{P}^{\triangleright}}^{\flat} \cap  \underline{D}_{\rho,\mathbf{z}} \cap   \underline{\digamma}_{G, {{\mathcal{P}^{\triangleright}}} }({\mathcal{P}'}, \rho_{ {\mathcal{P}'}},   \eta_{{\mathcal{P}'}} ) .
\end{equation}

Taking $\eta> \eta_{\mathcal{P}}$ condition  (\ref{matching-inductive-lemma-flat}) tells us 
\begin{equation} \label{cluster-lemma-basic2}
\underline{Z}_{G,m, \mathcal{P}}^{\flat} \cap  \underline{D}_{\rho_{\mathcal{P}},\mathbf{z}}=  \underline{Z}_{G,m, \mathcal{P}^{\triangleright}}^{\flat} \cap  \underline{D}_{\rho_{\mathcal{P}},\mathbf{z}}  .
\end{equation}

The relation defining the cluster condition 
$$\underline{Z}_{G,m, \mathcal{P}}^{\flat} \cap  \underline{D}_{\rho ,\mathbf{z}}=  \underline{Z}_{G,m, \mathcal{P}^{\triangleright}}^{\flat} \cap  \underline{D}_{\rho ,\mathbf{z}}  $$
follows from (\ref{cluster-lemma-basic1}) and (\ref{cluster-lemma-basic2}) since condition (\ref{covering-boundary-condition-nuclear}) implies
$$ \underline{D}_{\rho,\mathbf{z}}  \cap   \underline{ \digamma}_{G}^{< \rho^{\dagger}}  \subset  \underline{D}_{\rho_{\mathcal{P}},\mathbf{z}} \cup ( \underline{D}_{\rho,\mathbf{z}} \cap U_{\mathcal{P}}^{\flat, cluster}) .$$


\end{proof}


Below we construct flat $PS$-cycles using an inductive argument on the graphs.  It is useful to introduce the following notation.

Fix $\mathcal{P}$ and assume we have defined $(\underline{Z}^{\flat}_{G,m, \mathcal{P}'}, \rho_{\mathcal{P}'},   \eta_{\mathcal{P}'})_{\mathcal{P}' \prec \mathcal{P}}$.
 For the condition  (\ref{matching-inductive-lemma-flat}), for $(G,m, \mathcal{P}' , \rho_{ \mathcal{P}'}, \eta_{ \mathcal{P}'})$  we assume relation $(\ref{matching-nuclear})$ only for   $\mathcal{P}^{\triangleright}$ with $\mathcal{P}^{\triangleright} \prec \mathcal{P}$.

Define $ \rho_{\mathcal{P}},   \eta_{\mathcal{P}}$ as in the proof Lemma  \ref{cluster-flat-lemma}. 

Let $\underline{Z}^{\flat,  cluster}_{G,m ,\mathcal{P}} $ be the chain on $U_{\mathcal{P}}^{\flat, cluster}$ such that   for any $\mathcal{P}' \prec \mathcal{P}$ we have
\begin{equation}  \label{cluster-flat-definition-relation}
 \underline{Z}_{G,  m  , \mathcal{P}}^{\flat, cluster}  =  \underline{Z}_{G,  m  , \mathcal{P}'}^{\flat}  
 \end{equation}
 on $\underline{\digamma}_{G, {\mathcal{P}} }(\mathcal{P}', \rho_{\mathcal{P}'},   \eta_{\mathcal{P}'} )     $ .
Property  (\ref{matching-inductive-lemma-flat}) and  (\ref{not-overlapping-condition-nuclear}) tell us that these conditions are compatible for different $\mathcal{P}'$.
By induction $ \underline{Z}_{G,  m  , \mathcal{P}}^{\flat, cluster}  $ satisfies the translation invariance condition.

If we take  $( \rho_{\mathcal{P}'},   \eta_{\mathcal{P}'})_{\mathcal{P}' \prec \mathcal{P}}$ such that in addition the cluster condition is satisfied for
\begin{equation}  \label{matching-preview-lemma-flat}
    (G',m', \mathcal{P}',  \rho_{\mathcal{P}'}, \eta_{\mathcal{P}'})
\end{equation}
when $(G',m') \prec (G,m)$ and $\mathcal{P}' \prec \mathcal{P}$,
 condition (\ref{disjoint-disks-nuclear2}) the cycle condition lead to
$$\partial \underline{Z}_{G,  m  , \mathcal{P}}^{\flat, cluster} = \sum_i (-1)^i \underline{Z}_{G, \partial_i m , \mathcal{P}}^{\flat} + \sum_{e'}  \delta_{e'} \underline{Z}_{G',m', \mathcal{P} }^{\flat} .$$

\begin{lemma}
Assume $E^{ex}(G) \neq \emptyset$.  For each $Z_{\prec G}^{\flat} \in  \tilde{Z}_{\prec G}^{\flat} $ there exists $Z_{\preccurlyeq G}^{\flat}  \in  \tilde{Z}_{\preccurlyeq G}^{\flat}  $ extending $Z_{\prec G}^{\flat}$ .
\end{lemma}

\begin{proof}

Let $m$ with $|m|=l$ and assume  we have defined  $ \tilde{Z}^{\flat}_{G,  m' }  $ for $|m'| <l$.

If there is  $e \in  E^{ex}(G)$ removable define $\underline{Z}_{G, m}$ using   (\ref{stable-forget-condition})  and (\ref{removible-forget-condition})   . 

Assume that there are no removable external edges.
Take  $e \in  E^{ex}(G)$.


Take $\rho^{\dagger}$ big enough such that the support condition of radius $\rho^{\dagger}$ is satisfied    for each $G'$ with $G' \prec G$.

We want to define $(\underline{Z}^{\flat}_{G,m, \mathcal{P}}, \rho_{\mathcal{P}},   \eta_{\mathcal{P}})_{\mathcal{P}}$ such that  the following conditions are satisfied. 


Set 
$$  \underline{Z}_{G,  m , \mathcal{P} }^{\flat, prel} = \underline{Z}_{forget_e(G,m), \mathcal{P} }^{\flat} \times \underline{Z}_{Disk0}^{\flat} .$$


\begin{itemize} 
\item  $\underline{Z}^{\flat}_{G,m, \mathcal{P}}$ is close in the $C^{1}$-topology to $ \underline{Z}_{G,  m , \mathcal{P} }^{\flat, prel}$;

\item the transversality condition, the translation condition and the cycle condition are satisfied for   $\underline{Z}^{\flat}_{G,m, \mathcal{P}}$;

\item $\rho_{\mathcal{P}}$,   $\eta_{\mathcal{P}}$ satisfies the conditions of Lemma  \ref{cluster-flat-lemma}.

\item The cluster condition (\ref{matching-nuclear}) is satisfied for
\begin{equation}  \label{matching-preview-nuclear-forget}
    (G',m',   \mathcal{P},  \rho_{ \mathcal{P}}, \eta_{ \mathcal{P}})
    \end{equation}
when $(G',m') \preccurlyeq (G,m)$.


\end{itemize}

Assume we have defined $(\underline{Z}^{\flat}_{G,m, \mathcal{P}'}, \rho_{\mathcal{P}'},   \eta_{\mathcal{P}'})_{\mathcal{P}' \prec \mathcal{P}}$.

As in the proof Lemma  \ref{cluster-flat-lemma}, take $\rho_{\mathcal{P}}$ big enough such that condition (\ref{covering-boundary-condition-nuclear}) is satisfied,
then take $\eta_{\mathcal{P}}$ big enough such that conditions  (\ref{matching-inductive-lemma-flat}) , (\ref{not-overlapping-condition-nuclear}) ,  (\ref{disjoint-disks-nuclear2}) are satisfied. Now we ask in addition that
the cluster condition (\ref{matching-preview-nuclear-forget}) is satisfied for $(G',m')$ with $(G',m') \prec (G,m)$.


$ \underline{Z}_{G,  m  , \mathcal{P}}^{\flat, prel} $ satisfies the translation condition. The cycle condition for $forget_e(G,m)$ and the forgetful compatibility condition yield the relation 
$$\partial \underline{Z}_{G,  m  , \mathcal{P}}^{\flat, prel} \approx \sum_i (-1)^i \underline{Z}_{G, \partial_i m , \mathcal{P}}^{\flat} + \sum_{e'}  \delta_{e'} \underline{Z}_{G',m', \mathcal{P} }^{\flat} .$$

Condition  (\ref{matching-preview-nuclear-forget}) implies
$$ \underline{Z}_{G,  m  , \mathcal{P}}^{\flat, prel}  =  \underline{Z}_{G,  m  , \mathcal{P}'}^{\flat, prel}$$
on $ \underline{\digamma}_{G, {\mathcal{P}} }(\mathcal{P}', \rho_{G,m, \mathcal{P}'},   \eta_{G,m,\mathcal{P}'} )     $ for each $\mathcal{P}' \prec \mathcal{P}$.
From the inductive assumption we have
$$ \underline{Z}_{G,  m  , \mathcal{P}}^{\flat, cluster} \approx   \underline{Z}_{G,  m  , \mathcal{P}}^{\flat, prel} $$
on $U^{\flat, cluster}_{\mathcal{P}}$.

The standard transversality argument shows that we may define $\underline{Z}_{G,  m , \mathcal{P} }^{\flat}$ close in the $C^0$-topology to $\underline{Z}_{G,  m , \mathcal{P}   }^{\flat, prel}$ such that   the  transversality condition is satisfied,
$ \underline{Z}_{G,  m  , \mathcal{P}}^{\flat}  =  \underline{Z}_{G,  m  , \mathcal{P}}^{\flat, cluster} $ on $U^{\flat , cluster}_{\mathcal{P}}$,
the cycle condition and the translation  condition are satisfied for $(G,  m  , \mathcal{P})$  .




\end{proof}


\begin{lemma}
Assume $E^{ex}(G) = \emptyset, H(G) \neq \emptyset$. For each $Z_{\prec G}^{\flat} \in  \mathcal{Z}_{\prec G}^{\flat} $ there exists $Z_{\preccurlyeq G}^{\flat}  \in  \mathcal{Z}_{\preccurlyeq G}^{\flat}  $ extending $Z_{\prec G}^{\flat}$ .

\end{lemma}

\begin{proof}



Assume  we have defined  $ \underline{Z}^{\flat}_{G,  m } ,\underline{Z}^{\flat}_{G,  m \sqcup \{E(G) \} } $ for $|m| <l$ and $E(G) \notin m$.

Fix $m$ with $|m|=l$ and $E(G) \notin m$.

Take $\rho^{\dagger}$ big enough such that the support property of radius $\frac{1}{2} \rho^{\dagger}$ is satisfied for
$    (G',m'$ and $( G',  m' \sqcup \{E(G') \})$
when $(G',m') \prec (G,m)$ 

We want to define $(\underline{Z}^{\flat}_{G,m, \mathcal{P}} ,  \underline{Z}^{\flat}_{G,  m \sqcup \{E(G) \}, \mathcal{P} }, \rho_{\mathcal{P}},   \eta_{\mathcal{P}})_{\mathcal{P}}$ satisfying  the following conditions.


Set 
$$  \underline{Z}_{G,  m , \mathcal{P} }^{\flat, prel} =  (-1)^l(   \sum_i (-1)^i \underline{Z}_{G, \partial_i m \sqcup \{ E(G) \} , \mathcal{P}}^{\flat} + \sum_{G'/e'=G}  \delta_{e'} \underline{Z}_{G', m' \sqcup \{ E(G)\} , \mathcal{P}  }^{\flat})  .$$

\begin{itemize} 
\item  $\underline{Z}^{\flat}_{G,m, \mathcal{P}}$ is close in the $C^{0}$-topology to $ \underline{Z}_{G,  m , \mathcal{P} }^{\flat, prel}$;

\item  $ \underline{Z}^{\flat}_{G,  m \sqcup \{E(G) \}, \mathcal{P} }$  is close in the $C^{0}$-topology to $ \underline{Z}^{\flat, prel}_{G,  m \sqcup \{E(G) \}, \mathcal{P} }$,
where $ \underline{Z}^{\flat, prel}_{G,  m \sqcup \{E(G) \}, \mathcal{P} }$  be  a chain defined from a small cobordism between  $\underline{Z}^{\flat}_{G,m, \mathcal{P}}$ and $ \underline{Z}_{G,  m , \mathcal{P} }^{\flat, prel}$;


\item  the transversality condition, the translation invariance condition and the cycle condition are satisfied for   $\underline{Z}^{\flat}_{G,m, \mathcal{P}}$ and $ \underline{Z}^{\flat}_{G,  m \sqcup \{E(G) \}, \mathcal{P} }$;

\item  $\rho_{\mathcal{P}}$,   $\eta_{\mathcal{P}}$ satisfies the conditions of Lemma  \ref{cluster-flat-lemma}.

\item the cluster condition (\ref{matching-nuclear}) is satisfied for
\begin{equation}  \label{matching-preview-nuclear}
    (G',m',   \mathcal{P} , \rho_{ \mathcal{P}}, \eta_{\mathcal{P}}),( G',  m' \sqcup \{E(G') \},  \mathcal{P} ,  \rho_{\mathcal{P}}, \eta_{\mathcal{P}})
\end{equation}
when $(G',m') \preccurlyeq (G,m)$   ;   


\end{itemize}

Assume we have defined $(\underline{Z}^{\flat}_{G,m, \mathcal{P'}} ,  \underline{Z}^{\flat}_{G,  m \sqcup \{E(G) \}, \mathcal{P'} }, \rho_{\mathcal{P'}},   \eta_{\mathcal{P'}})_{\mathcal{P'} \prec \mathcal{P}}$.

Define $ \rho_{\mathcal{P}},   \eta_{\mathcal{P}}$ as in the proof Lemma  \ref{cluster-flat-lemma} asking in addition 
 the cluster condition (\ref{matching-preview-nuclear}) for  $(G',m') $ with $(G',m') \prec (G,m)$  .


$ \underline{Z}_{G,  m  , \mathcal{P}}^{prel}$ satisfies the translation condition.
The cycle condition for $(G', m' \sqcup \{ E(G)\} , \mathcal{P}  )$ with $(G',m') \preccurlyeq (G,m)$ implies
$$\partial \underline{Z}_{G,  m  , \mathcal{P}}^{\flat, prel} = \sum_i (-1)^i \underline{Z}_{G, \partial_i m , \mathcal{P}}^{\flat} + \sum_{e'}  \delta_{e'} \underline{Z}_{G',m', \mathcal{P} }^{\flat} .$$

Condition (\ref{matching-preview-nuclear}) for   $(G',m') \prec (G,m)$ implies 
\begin{equation} \label{matching-flat-prelimary}
 \underline{Z}_{G,  m  , \mathcal{P}}^{prel}  =  \underline{Z}_{G,  m  , \mathcal{P}'}^{prel}
 \end{equation}
on $ \underline{\digamma}_{G, {\mathcal{P}} }(\mathcal{P}', \rho_{\mathcal{P}'},   \eta_{\mathcal{P}'} )     $  for each $\mathcal{P}'  \prec \mathcal{P}$.  Relation  ( \ref{matching-flat-prelimary}) 
implies in particular that  $ \underline{Z}_{G,  m  , \mathcal{P}}^{\flat, cluster}$ is close in the $C^0$-topology to $  \underline{Z}_{G,  m  , \mathcal{P}}^{\flat, prel} $ on $U^{\flat, cluster}_{\mathcal{P}}$.

A standard transversality argument shows that we may define a chain $\underline{Z}_{G,  m , \mathcal{P} }^{\flat}$ close in the $C^0$-topology to $\underline{Z}_{G,  m , \mathcal{P}   }^{\flat, prel}$ such that the transversality condition is satisfied,
$\partial \underline{Z}_{G,  m   , \mathcal{P} }^{\flat} = \partial \underline{Z}_{G,  m   , \mathcal{P} }^{\flat, prel}$,  $\underline{Z}_{G,  m , \mathcal{P} }^{\flat}=\underline{Z}_{G,  m , \mathcal{P} }^{\flat, cluster}$ on  $U^{\flat, cluster}_{\mathcal{P}}$ and the translation invariance  condition  is satisfied.

We now define $\underline{Z}_{G,m \sqcup \{ E(G)\} ,\mathcal{P}}^{\flat}$.
Let $\underline{Z}^{\flat,  prel}_{G,m \sqcup \{ E(G)\} ,\mathcal{P}} $  be a chain obtained from a small cobordism between  $\underline{Z}^{\flat}_{G,m, \mathcal{P}}$ and $ \underline{Z}_{G,  m , \mathcal{P} }^{\flat, prel}$. We may assume $\underline{Z}^{\flat,  prel}_{G,m \sqcup \{ E(G)\} ,\mathcal{P}} $ satisfies the translation condition.
Using property  ( \ref{matching-flat-prelimary}) we may moreover assume  that
\begin{equation} \label{matching-flat-prelimary-2}
 \underline{Z}^{\flat,  prel}_{G,m \sqcup \{ E(G)\} ,\mathcal{P}} = \underline{Z}^{\flat,  prel}_{G,m \sqcup \{ E(G)\} ,\mathcal{P}'} 
 \end{equation} 
  on $ \underline{\digamma}_{G, {\mathcal{P}} }(\mathcal{P}', \rho_{ \mathcal{P}'},   \eta_{\mathcal{P}'} )     $
  for each $\mathcal{P}' \prec \mathcal{P}$. 
Relation ( \ref{matching-flat-prelimary-2})   implies in particular that  $\underline{Z}^{\flat,  prel}_{G,m \sqcup \{ E(G)\} ,\mathcal{P}} $ is close in the $C^0$-topology to $\underline{Z}^{\flat,  cluster}_{G,m \sqcup \{ E(G)\} ,\mathcal{P}} $ on $U^{\flat, cluster}_{\mathcal{P}}$.

A standard transversality argument shows that we may define a chain  $\underline{Z}_{G,m \sqcup \{ E(G)\} ,\mathcal{P}}^{\flat}$ close in the $C^0$-topology to $\underline{Z}^{\flat,  prel}_{G,m \sqcup \{ E(G)\} ,\mathcal{P}} $ such that the  transversality condition is satisfied, 
$\partial \underline{Z}_{G,m \sqcup \{ E(G)\} ,\mathcal{P}}^{\flat}  = \partial \underline{Z}^{\flat,  prel}_{G,m \sqcup \{ E(G)\} ,\mathcal{P}}   $,  $\underline{Z}_{G,m \sqcup \{ E(G)\} ,\mathcal{P}}^{\flat}= \underline{Z}_{G,m \sqcup \{ E(G)\} ,\mathcal{P}}^{\flat, cluster}$ on $U^{\flat, cluster}_{\mathcal{P}}$ and  the translation condition is satisfied.

\end{proof}

\begin{lemma}
Assume  $H(G)=\emptyset, G \neq G^{\heartsuit}$.  For each $Z_{\prec G}^{\flat} \in  \tilde{Z}_{\prec G}^{\flat} $ there exists $Z_{\preccurlyeq G}^{\flat}  \in  \tilde{Z}_{\preccurlyeq G}^{\flat}  $ extending $Z_{\prec G}^{\flat}$ .
We may set $\underline{Z}_G^{\flat} =0$.
\end{lemma}
\begin{proof}

Using the translation invariance condition, from dimensional reasons we have that $\delta_{e'}  {\underline{Z}}_{G'}^{\flat}= 0$  for each $(G' ,e' )$ with $\delta_{e'}G'  =G$ . 

\end{proof}

\subsubsection{Unicity}

\begin{proposition} \label{unicity-flat-proposition}
 The flat $PS$-cycle $Z^{\flat}$ of  Proposition \ref{existence-flat-proposition} is unique up to isotopy.
 \end{proposition}

We start with the isotopic version of Lemma \ref{cluster-flat-lemma}, whose proof is similar.

\begin{lemma} \label{cluster-flat-lemma-iso}

Let $(G,m) \in \mathfrak{G}_*$. Let $\rho^{\dagger}>0$ be a real positive number. 

Assume that $(\tilde{\underline{Z}}_{G,m, \mathcal{P}}^{\flat})$ satisfies the support condition with radius $\rho^{\dagger}$ for each $\mathcal{P}$.


If $(\tilde{\underline{Z}}_{G,m, \mathcal{P}}^{\flat})_{\mathcal{P}}$ satisfies the cluster condition there exists a collection of real positive numbers  $(\rho_{\mathcal{P}}, \eta_{\mathcal{P}})_{\mathcal{P}}$ such that

\begin{itemize}

\item  the cluster condition is satisfied for
\begin{equation}  \label{matching-inductive-lemma-flat-isotopies}
    (G,m,  \mathcal{P}, \rho_{\mathcal{P}}, \eta_{\mathcal{P}}).
\end{equation}

\item conditions  (\ref{covering-boundary-condition-nuclear}), (\ref{not-overlapping-condition-nuclear}), (\ref{disjoint-disks-nuclear2}) are satisfied.

\end{itemize}

Vice versa ,  the cluster condition is satisfied if  there exists  a collection of positive real numbers $(\rho_{\mathcal{P}}, \eta_{\mathcal{P}})_{\mathcal{P}}$ such that the conditions  above are satisfied.


\end{lemma}


Assume we have defined $(\tilde{\underline{Z}}^{\flat}_{G,m, \mathcal{P}'}, \rho_{\mathcal{P}'},   \eta_{\mathcal{P}'})_{\mathcal{P}'}$ for each $\mathcal{P}' \prec \mathcal{P}$. 

Define $ \rho_{\mathcal{P}},   \eta_{\mathcal{P}}$ as in the proof of Lemma  \ref{cluster-flat-lemma}. 

Let $\tilde{\underline{Z}}^{\flat,  cluster}_{G,m ,\mathcal{P}} $ be the chain on $U^{\flat, cluster}_{\mathcal{P}} \times \R$ such that   for each $\mathcal{P}' \prec \mathcal{P}$
\begin{equation} \label{condition-cluster-iso}
 \tilde{\underline{Z}}_{G,  m  , \mathcal{P}}^{\flat, cluster}  =  \tilde{\underline{Z}}_{G,  m  , \mathcal{P}'}^{\flat}
 \end{equation}
 on $\underline{\digamma}_{G, {\mathcal{P}} }(\mathcal{P}', \rho_{\mathcal{P}'},   \eta_{\mathcal{P}'} )   \times \R  $.
Conditions  (\ref{matching-inductive-lemma-flat-isotopies}) and  (\ref{not-overlapping-condition-nuclear}) imply that the conditions (\ref{condition-cluster-iso}) are compatible for different $\mathcal{P}'$.
Obviously $ \tilde{\underline{Z}}_{G,  m  , \mathcal{P}}^{\flat, cluster}  $ satisfies the translation condition.

If we take $( \rho_{\mathcal{P}'},   \eta_{\mathcal{P}'})_{\mathcal{P}'}$ such that in addition the cluster condition is satisfied for
\begin{equation}  \label{matching-preview-lemma-flat-isotopy}
    (G',m',  \mathcal{P}' , \rho_{ \mathcal{P}}, \eta_{\mathcal{P}})
    \end{equation}
 when $(G',m') \prec (G,m)$,    
from  the inductive cycle condition and  (\ref{disjoint-disks-nuclear2}) we obtain
$$\partial \tilde{\underline{Z}}_{G,  m  , \mathcal{P}}^{\flat, cluster} =  \sum_i (-1)^i \tilde{\underline{Z}}_{G, \partial_i m , \mathcal{P}}^{\flat} + \sum_{e'}  \delta_{e'} \tilde{\underline{Z}}_{G',m', \mathcal{P} }^{\flat} .$$

If we take $( \rho_{G,m,\mathcal{P}'},   \eta_{G,m,\mathcal{P}'})_{\mathcal{P}'}$ such that in addition $Z^{\flat, 0}$ and $Z^{\flat, 1 }$ satisfies the cluster condition for
\begin{equation}  \label{matching-inductive-lemma-flat-start-end}
    (G,m, \mathcal{P}' , \rho_{ \mathcal{P}'}, \eta_{ \mathcal{P}'})
\end{equation}
for each $\mathcal{P}' \prec \mathcal{P} $,
we have
$$\tilde{ {\underline{Z}}}_{G,m , \mathcal{P}  }^{\flat, cluster, < -\lambda^{\dagger}} = \underline{Z}_{G,m, \mathcal{P}  }^{\flat, 0} \times \R_{<-\lambda^{\dagger}}, \quad \tilde{\underline{Z}}_{G,m , \mathcal{P} }^{\flat, cluster, > \lambda^{\dagger}} =  \underline{Z}_{G,m }^{\flat,1} \times \R_{>\lambda^{\dagger}}$$
on  $U^{\flat, cluster}_{\mathcal{P}} \times \R_{<\lambda^{\dagger}} , U^{\flat, cluster}_{\mathcal{P}} \times \R_{>\lambda^{\dagger}}$, for $\lambda^{\dagger} \gg 0$.

\begin{lemma}
Assume $E^{ex}(G) \neq \emptyset$. For each $\tilde{Z}_{\prec G}^{\flat}$ there exists $\tilde{Z}_{\preccurlyeq G}^{\flat}$ extending $\tilde{Z}_{\prec G}^{\flat}$ .
\end{lemma}

\begin{proof}

Assume that there are no removable external edges.
Take  $e \in  E^{ex}(G)$.

Assume  we have defined   $ \tilde{\underline{Z}}^{\flat}_{G,  m } $ for $|m| <l$. Let $m$ with $|m|=l$. 

If there is  $e \in  E^{ex}(G)$ removable use   forgetful compatibility   to define $\tilde{\underline{Z}}_{G, m}$.

Take $\rho^{\dagger}$ big enough such that   $ \tilde{\underline{Z}}^{\flat} $ the support condition of radius $\rho^{\dagger}$  for  each $G' \prec G$. Moreover we assume that  $ {\underline{Z}}^{\flat,0} , {\underline{Z}}^{\flat,1} $ satisfies the support condition of radius $\rho^{\dagger}$   for  each $G' \preccurlyeq G$.


We want to define $( \tilde{\underline{Z}}^{\flat}_{G,m, \mathcal{P}}, \rho_{\mathcal{P}},   \eta_{\mathcal{P}})_{\mathcal{P}}$ such that  the following conditions are satisfied.

Set 
$$   \tilde{\underline{Z}}_{G,  m , \mathcal{P} }^{\flat, prel} =  \tilde{\underline{Z}}_{forget_e(G,m) }^{\flat} \times \underline{Z}_{Disk0}^{\flat} .$$


\begin{itemize} 
\item  $ \tilde{\underline{Z}}^{\flat}_{G,m, \mathcal{P}}$ is close in the $C^{0}$-topology to $  \tilde{\underline{Z}}_{G,  m , \mathcal{P} }^{\flat, prel}$;

\item  transversality condition, translation invariance and cycle condition hold for   $ \tilde{\underline{Z}}^{\flat}_{G,m, \mathcal{P}}$;

\item $$\tilde{ {\underline{Z}}}_{G,m , \mathcal{P}  }^{\flat, < -\lambda^{\dagger}} = \underline{Z}_{G,m, \mathcal{P}  }^{\flat, 0} \times \R_{<-\lambda^{\dagger}}, \quad \tilde{\underline{Z}}_{G,m , \mathcal{P} }^{\flat,  >\lambda^{\dagger}} =  \underline{Z}_{G,m }^{\flat,1} \times \R_{>\lambda^{\dagger}}$$
for $\lambda^{\dagger} \gg 0$.

\item $\rho_{\mathcal{P}}$,   $\eta_{\mathcal{P}}$ satisfies conditions of Lemma  \ref{cluster-flat-lemma-iso}.  

\item  $\tilde{\underline{Z}}^{\flat}$ satisfies the cluster condition for
\begin{equation}  \label{matching-preview-nuclear-forget-iso}
    (G',m',  \mathcal{P} , \rho_{ \mathcal{P}}, \eta_{ \mathcal{P}})
    \end{equation}
when $(G' , m') \preccurlyeq (G,m )$    ;


\item $Z^{\flat,0}$, $Z^{\flat,1}$ satisfies the cluster condition 
for
\begin{equation}  \label{matching-inductive-nuclear-forget-iso-boundary}
    (G,m,  \mathcal{P}, \rho_{\mathcal{P}}, \eta_{\mathcal{P}});
\end{equation}

\end{itemize}

Assume we have defined $( \tilde{\underline{Z}}^{\flat}_{G,m, \mathcal{P}'}, \rho_{\mathcal{P}'},   \eta_{\mathcal{P}'})_{\mathcal{P}'}$ for each $\mathcal{P}' \prec \mathcal{P}$.

Define $ \rho_{\mathcal{P}},   \eta_{\mathcal{P}}$ as in the proof Lemma  \ref{cluster-flat-lemma} asking in addition  
condition (\ref{matching-preview-nuclear-forget-iso}) for $(G' , m') \prec (G,m )$ and condition (\ref{matching-inductive-nuclear-forget-iso-boundary}).

 $  \tilde{\underline{Z}}_{G,  m  , \mathcal{P}}^{\flat, prel} $ satisfies the translation condition and condition
$$\tilde{ {\underline{Z}}}_{G,m , \mathcal{P}  }^{\flat, < -\lambda^{\dagger}} = \underline{Z}_{G,m, \mathcal{P}  }^{\flat, 0} \times \R_{<-\lambda^{\dagger}}, \quad \tilde{\underline{Z}}_{G,m , \mathcal{P} }^{\flat,  > \lambda^{\dagger}} = \underline{Z}_{G,m }^{\flat,1} \times \R_{>\lambda^{\dagger}}$$
for $T \gg 0$.
From the cycle condition for $ forget_e(G,m)$ and the forgetful compatibility condition we obtain 
$$\partial  \tilde{\underline{Z}}_{G,  m  , \mathcal{P}}^{\flat, prel} \approx \sum_i (-1)^i  \tilde{\underline{Z}}_{G, \partial_i m , \mathcal{P}}^{\flat} + \sum_{e'}  \delta_{e'}  \tilde{\underline{Z}}_{G',m', \mathcal{P} }^{\flat} .$$

Condition  (\ref{matching-preview-nuclear-forget-iso}) implies
$$  \tilde{\underline{Z}}_{G,  m  , \mathcal{P}}^{\flat, prel}  =   \tilde{\underline{Z}}_{G,  m  , \mathcal{P}'}^{\flat, prel}$$
on $ \underline{\digamma}_{G, {\mathcal{P}} }(\mathcal{P}', \rho_{ \mathcal{P}'},   \eta_{\mathcal{P}'} )  \times \R   $ for each $\mathcal{P}' \prec \mathcal{P}$. From the induction assumption we have
$$  \tilde{\underline{Z}}_{G,  m  , \mathcal{P}}^{\flat, cluster} \approx    \tilde{\underline{Z}}_{G,  m  , \mathcal{P}}^{\flat, prel} $$
 on $U^{\flat, cluster}_{\mathcal{P}} \times \R$.

A the standard transversality argument implies that we may define $ \tilde{\underline{Z}}_{G,  m , \mathcal{P} }^{\flat}$ close in the $C^0$-topology to $ \tilde{\underline{Z}}_{G,  m , \mathcal{P}   }^{\flat, prel}$ such that    transversality condition is satisfied,
$  \tilde{\underline{Z}}_{G,  m  , \mathcal{P}}^{\flat}  =   \tilde{\underline{Z}}_{G,  m  , \mathcal{P}}^{\flat, cluster} $ on $U^{\flat, cluster}_{\mathcal{P}} \times \R$,
the cycle condition and the translation condition is satisfied for $(G,  m  , \mathcal{P})$  and
$$\tilde{ {\underline{Z}}}_{G,m , \mathcal{P}  }^{\flat, < -\lambda^{\dagger}} = \underline{Z}_{G,m, \mathcal{P}  }^{\flat, 0} \times \R_{<-\lambda^{\dagger}}, \quad \tilde{\underline{Z}}_{G,m , \mathcal{P} }^{\flat,  > \lambda^{\dagger}} \cong  \underline{Z}_{G,m, \mathcal{P} }^{\flat,1} \times \R_{>\lambda^{\dagger}}$$
for $\lambda^{\dagger} \gg 0$.




\end{proof}

\begin{lemma}
Assume $E^{ex}(G) = \emptyset, H(G) \neq \emptyset$. For each $\tilde{Z}_{\prec G}^{\flat} \in  \tilde{\mathcal{Z}}_{\prec G}^{\flat} $ there exists $\tilde{Z}_{\preccurlyeq G}^{\flat}  \in  \tilde{\mathcal{Z}}_{\preccurlyeq G}^{\flat}  $ extending $\tilde{Z}_{\prec G}^{\flat}$ .
\end{lemma}

\begin{proof}

Assume  we have defined   $\tilde{\underline{Z}}_{G, m}^{\flat}$ and $\tilde{\underline{Z}}_{G,m \sqcup \{ E(G)\}}^{\flat}$ for $|m| <l$ and $E(G) \notin m$.

Fix $m$ with $|m|=l$ and $E(G) \notin m$.

Take $\rho^{\dagger}$ big enough such that $    \tilde{\underline{Z}}^{\flat}_{(G',m')}$ and $ \tilde{\underline{Z}}^{\flat}_{( G',  m' \sqcup \{E(G') \})}$ satisfy the support condition of radius $\frac{1}{2} \rho^{\dagger}$
when $(G',m') \prec (G,m)$,
  $ {\underline{Z}}^{\flat,0}_{G',m'} , {\underline{Z}}^{\flat,1}_{G',m'} $ satisfy the support condition of radius $\rho^{\dagger}$   when $(G',m') \preccurlyeq (G,m)$.

We want to define $(\underline{Z}^{\flat}_{G,m, \mathcal{P}} ,  \underline{Z}^{\flat}_{G,  m \sqcup \{E(G) \}, \mathcal{P} }, \rho_{\mathcal{P}},   \eta_{\mathcal{P}})_{\mathcal{P}}$ which satisfy the following conditions.



Take a chain $\tilde{\underline{Z}}_{G,m \sqcup \{ E(G) \} }^{\dagger}$ which agrees with $ \underline{Z}_{G,m \sqcup \{ E(G) \} }^0 \times \R_{<-\lambda^{\dagger}}$ and  $ \underline{Z}_{G,m \sqcup \{ E(G) \} }^1 \times \R_{>\lambda^{\dagger}}$  for $\lambda^{\dagger} \gg 0$.

Set 
$$   \tilde{\underline{Z}}_{G,  m , \mathcal{P} }^{\flat, prel} =  (-1)^l(   \sum_i (-1)^i  \tilde{\underline{Z}}_{G, \partial_i m \sqcup \{ E(G) \} , \mathcal{P}}^{\flat} + \sum_{G'/e'=G}  \delta_{e'}  \tilde{\underline{Z}}_{G', m' \sqcup \{ E(G)\} , \mathcal{P}  }^{\flat} - \partial \tilde{\underline{Z}}_{G,m \sqcup \{ E(G) \}, \mathcal{P} }^{\dagger} ). $$

We want to define $( \tilde{\underline{Z}}^{\flat}_{G,m, \mathcal{P}} ,   \tilde{\underline{Z}}^{\flat}_{G,  m \sqcup \{E(G) \}, \mathcal{P} }, \rho_{G,m,\mathcal{P}},   \eta_{G,m,\mathcal{P}})_{\mathcal{P}}$ such that following conditions are satisfied.



\begin{itemize} 
\item  $ \tilde{\underline{Z}}^{\flat}_{G,m, \mathcal{P}}$ is close in the $C^{0}$-topology to $  \tilde{\underline{Z}}_{G,  m , \mathcal{P} }^{\flat, prel}$;


Let   $ \tilde{\underline{Z}}^{\flat,cob}_{G,m \sqcup \{ E(G)\}  , \mathcal{P}}$ be a chain obtained from a small cobordism between  $ \tilde{\underline{Z}}^{\flat}_{G,m, \mathcal{P}}$ and $  \tilde{\underline{Z}}_{G,  m , \mathcal{P} }^{\flat, prel}$. 
We may assume that $ \tilde{\underline{Z}}^{\flat,cob}_{G,m \sqcup \{ E(G)\}  , \mathcal{P}}$ satisfies the translation condition.

Set
$$ \tilde{\underline{Z}}^{\flat,  prel}_{G,m \sqcup \{ E(G)\} ,\mathcal{P}} = (-1)^l  \tilde{\underline{Z}}^{\flat,cob}_{G,m \sqcup \{ E(G)\}   , \mathcal{P}} + \tilde{\underline{Z}}_{G,m \sqcup \{ E(G)\} , \mathcal{P}  }^{\flat, \dagger}.$$

\item $  \tilde{\underline{Z}}^{\flat}_{G,  m \sqcup \{E(G) \}, \mathcal{P} }$  is close in the $C^{0}$-topology to $  \tilde{\underline{Z}}^{\flat, prel}_{G,  m \sqcup \{E(G) \}, \mathcal{P} }$;


\item  transversality condition, translation  condition and cycle condition are satisfied for   $ \tilde{\underline{Z}}^{\flat}_{G,m, \mathcal{P}}$ and $  \tilde{\underline{Z}}^{\flat}_{G,  m \sqcup \{E(G) \}, \mathcal{P} }$;

\item $$\tilde{ {\underline{Z}}}_{G,m , \mathcal{P}  }^{\flat, < -\lambda^{\dagger}} = \underline{Z}_{G,m, \mathcal{P}  }^{\flat, 0} \times \R_{<-\lambda^{\dagger}}, \quad \tilde{\underline{Z}}_{G,m , \mathcal{P} }^{\flat,  > \lambda^{\dagger}} =  \underline{Z}_{G,m }^{\flat,1} \times \R_{>\lambda^{\dagger}}, $$
$$\tilde{\underline{Z}}_{G,m \sqcup \{ E(G)\} }^{\flat, < -T} = \underline{Z}_{G,m \sqcup \{ E(G)\} }^{\flat, 0} \times \R_{<-\lambda^{\dagger}}, \quad \tilde{\underline{Z}}_{G,m \sqcup \{ E(G)\} }^{\flat,  > \lambda^{\dagger}} =  \underline{Z}_{G,m \sqcup \{ E(G)\}}^{\flat,1} \times \R_{>\lambda^{\dagger}}$$
for $\lambda^{\dagger} \gg 0$

\item $\rho_{\mathcal{P}}$,   $\eta_{\mathcal{P}}$ satisfies the conditions of Lemma  \ref{cluster-flat-lemma};

\item   $ \tilde{\underline{Z}}^{\flat}$ satisfies the cluster condition
for
\begin{equation}  \label{matching-preview-nuclear-iso}
    (G',m', \mathcal{P},  \rho_{\mathcal{P}}, \eta_{ \mathcal{P}}),( G',  m' \sqcup \{E(G') \},\mathcal{P},  \rho_{ \mathcal{P}}, \eta_{ \mathcal{P}})
\end{equation}
when $(G',m') \preccurlyeq (G,m)$   ;


\item $Z^{\flat,0}$, $Z^{\flat,1}$ satisfies the cluster condition 
for
\begin{equation}  \label{matching-inductive-nuclear-iso-boundary}
    (G,m, \mathcal{P}' , \rho_{ \mathcal{P}}, \eta_{ \mathcal{P}}),( G,  m \sqcup \{E(G) \},\mathcal{P}, \rho_{\mathcal{P}}, \eta_{ \mathcal{P}})
\end{equation}
when $\mathcal{P}' \preccurlyeq \mathcal{P} $.

\end{itemize}

Assume we have defined $( \tilde{\underline{Z}}^{\flat}_{G,m, \mathcal{P'}} ,   \tilde{\underline{Z}}^{\flat}_{G,  m \sqcup \{E(G) \}, \mathcal{P'} }, \rho_{\mathcal{P'}},   \eta_{\mathcal{P'}})_{\mathcal{P'} \prec \mathcal{P}}$.

Define $ \rho_{\mathcal{P}},   \eta_{\mathcal{P}}$ as in the proof Lemma  \ref{cluster-flat-lemma} asking in addition  condition
(\ref{matching-preview-nuclear-iso}) when  $(G',m') \prec (G,m)$ and condition   ( \ref{matching-inductive-nuclear-iso-boundary}) .

$  \tilde{\underline{Z}}_{G,  m  , \mathcal{P}}^{prel}$ satisfies the translation condition and
$$\tilde{\underline{Z}}_{G,m , \mathcal{P}  }^{\flat, prel,< -\lambda^{\dagger}} = \underline{Z}_{G,m, \mathcal{P}  }^{\flat, 0} \times \R_{<-\lambda^{\dagger}}, \quad \tilde{\underline{Z}}_{G,m , \mathcal{P}  }^{\flat, prel, > \lambda^{\dagger}} =  \underline{Z}_{G,m }^{\flat,1} \times \R_{>\lambda^{\dagger}}$$  
for $\lambda^{\dagger} \gg 0$. 
Using the cycle condition for $(G, \partial_i m \sqcup \{ E(G) \} , \mathcal{P})$ and  $(G', m' \sqcup \{ E(G)\} , \mathcal{P}  )$ we have 
$$\partial  \tilde{\underline{Z}}_{G,  m  , \mathcal{P}}^{\flat , prel} = \sum_i (-1)^i  \tilde{\underline{Z}}_{G, \partial_i m , \mathcal{P}}^{\flat} + \sum_{e'}  \delta_{e'}  \tilde{\underline{Z}}_{G',m', \mathcal{P} }^{\flat} .$$




Condition (\ref{matching-preview-nuclear-iso}) implies that for each $\mathcal{P}'$ with $\mathcal{P}' \prec \mathcal{P}$ we have
\begin{equation} \label{matching-flat-prelimary-iso}
  \tilde{\underline{Z}}_{G,  m  , \mathcal{P}}^{\flat, prel}  =   \tilde{\underline{Z}}_{G,  m  , \mathcal{P}'}^{\flat, prel}
 \end{equation}
on $ \underline{\digamma}_{G, {\mathcal{P}} }(\mathcal{P}', \rho_{\mathcal{P}'},   \eta_{\mathcal{P}'} )  \times \R   $.
From ( \ref{matching-flat-prelimary-iso}) 
 it follows that  $  \tilde{\underline{Z}}_{G,  m  , \mathcal{P}}^{\flat, prel}$ is close in  $C^0$-topology to $   \tilde{\underline{Z}}_{G,  m  , \mathcal{P}}^{\flat, cluster} $ on $U^{\flat, cluster}_{\mathcal{P}} \times \R$.

A standard transversality argument shows that we may define a chain $ \tilde{\underline{Z}}_{G,  m , \mathcal{P} }^{\flat}$ close in the $C^0$-topology to $ \tilde{\underline{Z}}_{G,  m , \mathcal{P}   }^{\flat, prel}$ such that the transversality condition is satisfied,
$\partial  \tilde{\underline{Z}}_{G,  m   , \mathcal{P} }^{\flat} = \partial  \tilde{\underline{Z}}_{G,  m   , \mathcal{P} }^{\flat, prel}$,  $ \tilde{\underline{Z}}_{G,  m , \mathcal{P} }^{\flat}= \tilde{\underline{Z}}_{G,  m , \mathcal{P} }^{\flat, cluster}$ on $U^{\flat, cluster}_{\mathcal{P}} \times \R$, the translation  condition  is satisfied and
$$\tilde{\underline{Z}}_{G,m   }^{\flat, < -\lambda^{\dagger}} = \underline{Z}_{G,m  }^{\flat, 0} \times \R_{<-\lambda^{\dagger}}, \quad \tilde{\underline{Z}}_{G,m   }^{\flat,  > \lambda^{\dagger}} =  \underline{Z}_{G,m }^{\flat,1} \times \R_{>\lambda^{\dagger}}, $$
for $\lambda^{\dagger} \gg 0$.


  
Now we define $ \tilde{\underline{Z}}^{\flat}_{G,m \sqcup \{ E(G)\} ,\mathcal{P}}$.
$  \tilde{\underline{Z}}_{G,  m \sqcup \{ E(G)\} , \mathcal{P}}^{\flat, prel}$ satisfies the translation condition and the conditions
$$\tilde{\underline{Z}}_{G,m \sqcup \{ E(G)\} }^{\flat, prel,< -\lambda^{\dagger}} = \underline{Z}_{G,m \sqcup \{ E(G)\} }^{\flat, 0} \times \R_{<-\lambda^{\dagger}}, \quad \tilde{\underline{Z}}_{G,m \sqcup \{ E(G)\} }^{\flat, prel, > \lambda^{\dagger}} =  \underline{Z}_{G,m \sqcup \{ E(G)\}}^{\flat,1} \times \R_{>\lambda^{\dagger}}$$  
for $\lambda^{\dagger} \gg 0$,
$$\partial  \tilde{\underline{Z}}_{G,  m \sqcup \{ E(G)\}  , \mathcal{P}}^{\flat , prel} = (-1)^l (   \tilde{\underline{Z}}_{G,  m  , \mathcal{P}}^{\flat, prel }  -  \tilde{\underline{Z}}_{G,  m  , \mathcal{P}}^{\flat } )+ \partial  \tilde{\underline{Z}}_{G,m \sqcup \{ E(G)\}}^{\flat, \dagger}  . $$

 

Moreover, using property  ( \ref{matching-flat-prelimary-iso}), we may take  $ \tilde{\underline{Z}}^{\flat,cob}_{G,m \sqcup \{ E(G)\}  , \mathcal{P}}$ such that  
 $ \tilde{\underline{Z}}^{\flat,  prel}_{G,m \sqcup \{ E(G)\} ,\mathcal{P}} =  \tilde{\underline{Z}}^{\flat,  prel}_{G,m \sqcup \{ E(G)\} ,\mathcal{P}'} $ on $ \underline{\digamma}_{G, {\mathcal{P}} }(\mathcal{P}', \rho_{G,m, \mathcal{P}'},   \eta_{G,m,\mathcal{P}'} )     $ for each $\mathcal{P}' \prec \mathcal{P}$. 
 It follows that  $ \tilde{\underline{Z}}^{\flat,  prel}_{G,m \sqcup \{ E(G)\} ,\mathcal{P}} $ is close in the $C^0$-topology to $ \tilde{\underline{Z}}^{\flat,  cluster}_{G,m \sqcup \{ E(G)\} ,\mathcal{P}} $ on $U^{\flat, cluster}_{\mathcal{P}} \times \R$. 

A standard transversality argument shows that we may define a chain  $ \tilde{\underline{Z}}_{G,m \sqcup \{ E(G)\} ,\mathcal{P}}$ close in the $C^0$-topology to $ \tilde{\underline{Z}}^{\flat,  prel}_{G,m \sqcup \{ E(G)\} ,\mathcal{P}} $ such that  it satisfies the  transversality condition, translation condition and the conditions
$\partial  \tilde{\underline{Z}}_{G,m \sqcup \{ E(G)\} ,\mathcal{P}}  = \partial  \tilde{\underline{Z}}^{\flat,  prel}_{G,m \sqcup \{ E(G)\} ,\mathcal{P}}   $,  $ \tilde{\underline{Z}}_{G,m \sqcup \{ E(G)\} ,\mathcal{P}}=  \tilde{\underline{Z}}_{G,m \sqcup \{ E(G)\} ,\mathcal{P}}^{cluster}$ on $U^{\flat, cluster}_{\mathcal{P}} \times \R$,   
$$\tilde{\underline{Z}}_{G,m \sqcup \{ E(G)\} }^{\flat, < -\lambda^{\dagger}} = \underline{Z}_{G,m \sqcup \{ E(G)\} }^{\flat, 0} \times \R_{<-\lambda^{\dagger}}, \quad \tilde{\underline{Z}}_{G,m \sqcup \{ E(G)\} }^{\flat,  > \lambda^{\dagger}} =  \underline{Z}_{G,m \sqcup \{ E(G)\}}^{\flat,1} \times \R_{>\lambda^{\dagger}}$$  
for $\lambda^{\dagger} \gg 0$.

\end{proof}

\subsection{Flat-Approaching $PS$-cycles}

Let $\mathbf{fr} \in \mathfrak{Fr}(M  )$ and let  $Z^{\flat, \circledcirc}$ be a flat annulus data.

A \emph{flat-approccing-annulus-data} $Z^{\wr, \circledcirc}$ converging to $(\mathbf{fr}, Z^{\flat, \circledcirc})$   is a nice annulus data  with the following properties:
\begin{itemize}
\item
$  \underline{Z}_{G_1}^{\wr, \circledcirc}  =   \underline{Z}_{G_2}^{\wr, \circledcirc} =0$;
\item
$\underline{Z}_{G_0}^{\wr, \circledcirc} $ satisfies the following convergence condition:
for each $\rho>0$
  \begin{equation} \label{convergence-annulus} 
 Dil_n( \text{exp}^{-1}(\underline{D}_{\rho/n, {z} } \cap \underline{Z}_{G_0}^{\wr, \circledcirc, (n)} ) ) \rightarrow \underline{D}_{\rho, (0,0,0)}^{\flat } \cap  \underline{Z}_{G_0}^{\flat, \circledcirc}
   \end{equation}
 in $C^{\infty}$-topology, uniformly for ${z}  \in M$. Here we have used $\mathbf{fr} $ to identify the tangent space of $z$ with $\R^3$ and  $Dil_n $ is the dilatation map of radio $n$.
\end{itemize}

Let $Z^{\flat} \in \mathcal{Z}^{\flat}( Z^{\flat, \circledcirc} )$ and $Z^{\wr, \circledcirc}$ converging to $(\mathbf{fr}, Z^{\flat, \circledcirc})$.

A flat-approaching $PS$-cycles  $\mathcal{Z}^{\wr}$  converging to $(\mathbf{fr} , Z^{\flat} )$ with annulus data $Z^{\wr, \circledcirc}$ is defined by a  collection of chains
\begin{equation} \label{flat-appr-complex}
    ( Z^{\wr , (n)}_{G,m, \mathcal{P}} )_{G,m, \mathcal{P}}
\end{equation}
$$Z^{\wr , (n)}_{G,m, \mathcal{P}} \in C_{|H(G)| + |m| }(M^{H(G)} , \mathfrak{o}_{H(G)}  )^{\text{Aut}(G,m, \mathcal{P})}$$ 
such that :
\begin{itemize}
\item  
 transversality condition is satisfied;
\item   forgetful compatibility condition is satisfied;
\item support  condition is satisfied;
    \item cluster condition  is satisfied;
    \item convergence  condition is satisfied .
\end{itemize}

We denote by ${\mathcal{Z}}^{\wr}(M|  {Z}^{\wr, \circledcirc})$ the space  of asymptotic flat $PS$-cycles. 

\subsubsection{Transversality condition}

We say that the transversality condition is satisfied if $Z^{\wr , (n)}_{G,m, \mathcal{P}}$ is transversal  to   $ \cap_{e \in E'} \pi_e^{-1}(Diag)$ for each subset $E' \subset E^{in }(G , \mathcal{P})  \setminus E^l$.

\subsubsection{Forgetful Compatibility}

The multi-loop space  $\underline{M}_{G,\mathcal{P}}$ associated to $(G, \mathcal{P}) \in \mathfrak{G}^{part}$  is the set of arrays $(( {w}_v , t_v) )_{v \in V(G)}  $  where  
  ${w}_v : \mathbb{S}^1 \rightarrow M$,  $t_v : H_v  \hookrightarrow \mathbb{S}^1$.

Set $Supp({w}_v,t_v)= {w}_v(\mathbb{S}^1)$ for each $v \in V(G)$.

We say that the collection of chains (\ref{flat-appr-complex}) are forgetful compatible if there exists a collections of  chains 
\begin{equation} \label{flat-appr-complex-loop}
    (\underline{Z}^{\wr , (n)}_{G,m, \mathcal{P}} )_{G,m, \mathcal{P}}
\end{equation}
$$\underline{Z}^{\wr , (n)}_{G,m, \mathcal{P}} \in C_{|H(G)| + |m| }(\mathfrak{L}_{\hat{V}(G)}M , \mathfrak{o}_{H(G)}  )^{\text{Aut}(G,m, \mathcal{P})}$$
such that
$$ev_*( \underline{Z}^{\wr , (n)}_{G,m, \mathcal{P}}    )=Z^{\wr , (n)}_{G,m, \mathcal{P}}$$
and for each external edge $e \in E^{ex}(G)$ we have
\begin{itemize}
\item 
    \begin{equation} \label{flat-appr-stable-forget-condition}    
     \underline{Z}_{G,m, \mathcal{P}} = \underline{Z}_{G',m'} \times_{\textbf{forget}_e} \mathfrak{L}_{\hat{V}(G)}( M)  .
     \end{equation}
     when $G$ is stable after removing $e $;     
       \item in the case $G$ is unstable after removing $e $ and $e$ is removable  we require
 \begin{equation} \label{flat-appr-removible-forget-condition}
 \underline{Z}^{\wr , (n)}_{G,m, \mathcal{P}}  =  \underline{Z}^{\wr , (n)}_{G',m', \mathcal{P}} \times \underline{Z}_{G_e}^{\wr , \circledcirc} 
 \end{equation}
 when $G_e \in Ann0$, and
 \begin{equation} \label{flat-appr-removible-forget-condition2}
 \underline{Z}^{\wr , (n)}_{G,m, \mathcal{P}} =  \underline{Z}^{\wr , (n)}_{G',m', \mathcal{P}} \times \underline{Z}_{Disk0}  
 \end{equation}  
 when $G_e= Disk0$.
      \item In the case  $e$ is not removable we require that $ \underline{Z}^{\wr , (n)}_{G,m, \mathcal{P}}$ is close in the $C^1$-topology to $  \underline{Z}^{\wr , (n)}_{G',m', \mathcal{P}}\times \underline{Z}_{Disk0}^{}$:
 \begin{equation} \label{flat-appr-unstable-forget-condition}
 \underline{Z}^{\wr , (n)}_{G,m, \mathcal{P}} \approx \underline{Z}^{\wr , (n)}_{G',m', \mathcal{P}} \times \underline{Z}_{Disk0}^{} .
 \end{equation}
  \end{itemize}

\subsubsection{Support Condition}

Take $(G, \mathcal{P}) \in \mathfrak{G}^{part}$. For $\rho >0$ define the set  $\underline{M}_{G, \mathcal{P}}^{< \rho}$ as the set of elements
$( w_v , t_v)_v \in \underline{M}_{G, \mathcal{P}}$  such that
\begin{itemize}
 \item
$ Diameter ( \sqcup_{v \in V_c} w_v(\mathbb{S}^1) ) < \rho$  when $c \in  Comp_0$  ;
\end{itemize}

We say that the support condition is satisfied for  $(G,m, \mathcal{P}) \in \mathfrak{G}^{part}_*  $ if  there exists  a  real positive number $\rho^{\dagger}$ such that
$$ Supp(  \underline{Z}_{G,m, \mathcal{P}}^{\wr, (n)} )  \subset  \underline{M}_{G, \mathcal{P}}^{<\rho^{\dagger}/n} \quad \forall n .$$

\subsubsection{Cluster Condition}

Denote by $Diag_{\mathcal{P}}$  the set of $\mathbf{z} = (z_v)_{v \in V(G)} \in M^{V(G)}$ such that $z_{v_1} =z_{v_2}$ when $v_1 \sim_{\mathcal{P}} v_2$. 



\begin{definition}

Let $\mathcal{P} $ be a partition of $G \in \mathfrak{G}$. For  $  \mathcal{P}^{\triangleright} $ with  $\mathcal{P}  \prec  \mathcal{P}^{\triangleright} $ and $\eta>0$  denote
by $ Cut_{\mathcal{P}^{\triangleright} , \eta} ( Diag_{\mathcal{P} })$ the set of $\mathbf{z}  \in Diag_{\mathcal{P} }$ such that
$d ( \mathbf{z}, Diag_{\mathcal{P}^{\triangleright, \prime}}  ) > \eta  $ for each $\mathcal{P}^{\triangleright, \prime}$ with $\mathcal{P}  \prec \mathcal{P}^{\triangleright, \prime} \preccurlyeq \mathcal{P}^{\triangleright} $.

\end{definition}


For $z \in M$ and $\rho>0$, denote by $D_{\rho, z}$ the disk of radius $\rho$ with center $z$.

For $\mathbf{z} = (z_{v})_{v \in V(\mathcal{P}) } \in M^{V(\mathcal{P})}$, 
let $\underline{D}_{\rho, \mathbf{z}}$ be
the set of elements $({w}_v,t_v)_v  \in  \underline{M}_{G,\mathcal{P}}$
such that 
\begin{equation} \label{shring-condition-loop0}
{w}_v(\mathbb{S}^1) \subset D_{\rho, z_{ v}} \quad    \forall  v \in V(\mathcal{P}).
\end{equation}

For $\mathcal{P} \prec \mathcal{P}^{\triangleright} $,
for each $\rho>0$ there exists  $\eta$ big enough such that 
\begin{equation} \label{disjoint-disks}
  \mathbf{z} \in  Cut_{\mathcal{P}^{\triangleright} , \eta} ( Diag_{\mathcal{P} }),  v_1 \nsim_{\mathcal{P}^{\triangleright}} v_2 \Longrightarrow  {D}_{\rho, z_{ v_1}} \cap {D}_{\rho, z_{ v_2}} = \emptyset.
 \end{equation}


Property (\ref{disjoint-disks}) tells us that we have an inclusion of sets
\begin{equation} \label{inclusion-disks-partion-loop}
     \underline{D}_{\rho, \mathbf{z}} \hookrightarrow   \underline{M}_{G, \mathcal{P}^{\triangleright} } 
\end{equation}
which is compatible with the cycle condition in the sense that $\delta_e \neq 0$ on   the chains of $ \underline{D}_{\rho, \mathbf{z}}$ for  $e \notin  E(G, \mathcal{P})$.

Let $(G,m,\mathcal{P}) \in \mathfrak{G}_*^{part} $, $\rho, \eta> 0$.
We say that the cluster condition is satisfied for $(G,m, \mathcal{P}, \rho, \eta)$ if 
\begin{equation} \label{matching-asintontically}
\underline{Z}^{\wr, (n)}_{G,m, \mathcal{P}} \cap \underline{D}_{\rho/n ,  \mathbf{z}}  = \underline{Z}^{\wr, (n)}_{G,m, \mathcal{P}^{\triangleright}}  \cap \underline{D}_{\rho/n , \mathbf{z}},
  \end{equation}
for each $\mathcal{P}^{\triangleright} $ with $\mathcal{P} \prec \mathcal{P}^{\triangleright} $ and $ \mathbf{z} \in  Cut_{\mathcal{P}^{\triangleright} , \eta/n} ( Diag_{\mathcal{P} }) $.

We say that  the cluster condition is satisfied for
$(G,m, \mathcal{P} , \rho)$ if there exists $\eta$ big enough such that the cluster condition is satisfied for
$(G,m,  \mathcal{P}, \rho, \eta)$ .





\subsubsection{Convergence Condition}

Use $\mathbf{fr} $ to identify $T_zM$ with $\R^3$ for each $z \in M$.

Denote by $Dil_n$  the dilation of ratio $n$ of the tangent space on $\mathbf{z} $.

We say that the convergence condition is satisfied for $(G,m,\mathcal{P})$ if    for each $\rho>0$
  \begin{equation} \label{convergence-loop} 
 Dil_n( \text{exp}_{\mathbf{z}}^{-1}(\underline{D}_{\rho/n, \mathbf{z} } \cap \underline{Z}_{G,m,\mathcal{P}}^{\wr, (n)} ) ) \rightarrow \underline{D}_{\rho, \mathbf{0}}^{\flat} \cap  {\underline{Z}}^{\flat}_{G,m, \mathcal{P}} 
   \end{equation}
   in the $C^{\infty}$-topology, 
 uniformly for $\mathbf{z}  \in  Diag_{\mathcal{P}}$.

We say that $Z$ satisfies the convergence condition if the convergence condition is satisfied for each $(G,m,\mathcal{P})$.

\subsubsection{Isotopies}

Let $\tilde{Z}^{\flat, \circledcirc}$ be an isotopy between flat annulus data ${Z}^{\flat, \circledcirc,0}$ and  ${Z}^{\flat, \circledcirc,1}$. Let $\tilde{Z}^{\flat} \in \tilde{\mathcal{Z}}(M|  \tilde{Z}^{\flat, \circledcirc})$ be an isotopy of flat $PS$-cycles  $Z^{\flat,0} \in \mathcal{Z}(M|  {Z}^{\flat, \circledcirc,0})$ and $Z^{\flat,1} \in \mathcal{Z}(M|  {Z}^{\flat, \circledcirc,1})$.

Let $\mathbf{fr}^0$ and $\mathbf{fr}^1$ be frames of $M$. Let $\tilde{\mathbf{fr}}$ be an isotopy  between  $\mathbf{fr}^0$ and $\mathbf{fr}^1$.

Let  ${Z}^{\wr, \circledcirc,0}$ and  ${Z}^{\wr , \circledcirc,1}$ be flat approaching annulus data converging to $({Z}^{\flat, \circledcirc,0}, \mathbf{fr}^0), ({Z}^{\flat, \circledcirc,1},\mathbf{fr}^1)$ respectively.

Let $\tilde{Z}^{\wr, \circledcirc}$ be an isotopy between  ${Z}^{\wr, \circledcirc,0}$ and  ${Z}^{\wr , \circledcirc,1}$ converging to $( \tilde{Z}^{\flat, \circledcirc} , \tilde{\mathbf{fr}}  )$ . Let $Z^{\wr,0} \in \mathcal{Z}(M|  {Z}^{\circledcirc,0})$, $Z^{\wr,1} \in \mathcal{Z}(M|  {Z}^{\circledcirc,1})$ converging  to $(Z^{\flat,0 }, \mathbf{fr}^0), (Z^{\flat, 1 },\mathbf{fr}^1)$ respectively.

An isotopy  
of flat approching $PS$-cycles
  between $Z^{\wr,0}$ and $Z^{\wr,1}$ converging to $(\tilde{Z}^{\flat} , \tilde{\mathbf{fr}})$   
   consists in a collection of chains
\begin{equation} \label{collection-chains-isotopy-approching}
\tilde{Z}^{\wr}=  ( \tilde{Z}_{G,m, \mathcal{P}}^{\wr} )_{G,m, \mathcal{P}} ,
\end{equation}
where
$$\tilde{Z}_{G,m, \mathcal{P}}^{\wr} \in C_{|H(G)| + |m|+1} ( M^{H(G) \times \R } , \mathfrak{o}_{H(G)})^{\text{Aut}(G,m, \mathcal{P})}  ,$$
such that 
\begin{itemize}
\item there exists  $T >0$  big enough such that
$$ \tilde{Z}^{\wr, <-\lambda^{\dagger}} =  Z^{\wr,0}  \times  \R_{<-\lambda^{\dagger}}  , \quad \tilde{Z}^{\wr, >\lambda^{\dagger}} =  Z^{\wr,1}  \times \R_{>\lambda^{\dagger}}   ;$$

 \item 
the transversality condition is satisfied;


\item the forgetful compatibility property is satisfied;

\item the cycle condition is satisfied;

\item the support property is satisfied;

\item the cluster condition is satisfied;

\item the convergence condition is satisfied.


\end{itemize}

We denote by $\tilde{\mathcal{Z}}^{\wr}(M|  \tilde{Z}^{\wr, \circledcirc})$ the space of isotopies of asymptotic flat $PS$-cycles.

The transversality condition, forgetful compatibility condition and cycle condition are defined analogously to section  \ref{isotopies-section}.

We say that the support condition is satisfied for  $(G,m, \mathcal{P})$ if there exists a  real positive number $\rho_{G, m}^{\dagger}$ such that
$$ Supp(  \tilde{\underline{Z}}_{G,m}^{(n)} )  \subset   \underline{M}_{G, \mathcal{P}}^{<\rho_{G,m}^{\dagger}/n}  \times \R.$$

The cluster condition is the one parameter version of the cluster condition (\ref{matching-asintontically}).  
Let $(G,m,\mathcal{P}) \in \mathfrak{G}_*^{part} $, $\rho, \eta> 0$.
We say that the cluster condition is satisfied for $(G,m, \mathcal{P}, \rho, \eta)$ if 
\begin{equation} \label{matching-asintontically-isotopy}
\underline{Z}^{\wr, (n)}_{G,m, \mathcal{P}} \cap (\underline{D}_{\rho/n ,  \mathbf{z}} \times \R)  = \underline{Z}^{\wr, (n)}_{G,m, \mathcal{P}^{\triangleright}}  \cap (\underline{D}_{\rho/n , \mathbf{z}} \times \R),
  \end{equation}
for each $\mathcal{P}^{\triangleright} $ with $\mathcal{P} \prec \mathcal{P}^{\triangleright} $ and $ \mathbf{z} \in  Cut_{\mathcal{P}^{\triangleright} , \eta/n} ( Diag_{\mathcal{P} }) $.

The convergence condition is  the one parameter version of  the convergence condition  (\ref{convergence-loop}).
We say that 
 $(\underline{\tilde{Z}}^n_{G,m,\mathcal{P}})_n$ converges to $  ({\underline{\tilde{Z}}}^{\flat}_{G,m,\mathcal{P}}, \tilde{\mathbf{fr}})$, if
  \begin{equation} \label{convergence-loop-isotopy} 
 Dil_n( \text{exp}_{\mathbf{z}}^{-1}((\underline{D}_{\rho/n, \mathbf{z} } \times \R) \cap \tilde{\underline{Z}}_{G,m,\mathcal{P}}^n ) ) \rightarrow(   \underline{D}_{\rho, \mathbf{0}}^{\flat} \times \R ) \cap \tilde{\underline{Z}}^{\flat}_{G,m, \mathcal{P}} 
   \end{equation}
in the $C^{\infty}$-topology, uniformly for $\mathbf{z}  \in  Diag_{\mathcal{P}}$,    for each $\rho>0$ .

\subsubsection{Existence}

Let $(G,m) \in \mathfrak{G}_*$.

Set
$$ \underline{M}_{G, \mathcal{P} }(\rho) = \bigsqcup_{ \mathbf{z} \in  Diag_{\mathcal{P} }}   \underline{D}_{ \rho, \mathbf{z}}  . $$
and,
for $\mathcal{P}' \prec \mathcal{P}, \rho', \eta'>0$,
$$ \underline{M}_{G, \mathcal{P} }(\mathcal{P} ', \rho', \eta') = \bigsqcup_{ \mathbf{z}'  \in Cut_{\mathcal{P} , \eta'} ( Diag_{\mathcal{P} '}) }   \underline{D}_{ \rho', \mathbf{z}'}  \subset  \underline{M}_{G, {\mathcal{P}'} }    . $$

From (\ref{disjoint-disks-asymptotic}) we have that, for $\eta' $ big enough, there is a natural inclusion of sets
\begin{equation} \label{inclusion-partion-loop}
     \underline{M}_{G, \mathcal{P} }(\mathcal{P} ', \rho', \eta') \hookrightarrow   \underline{M}_{G, \mathcal{P} } . 
\end{equation}
 compatible with the cycle condition.



Let  $(\rho_{\mathcal{P}}$,   $\eta_{\mathcal{P}})_{\mathcal{P}}$, $\rho^{\dagger}$ be a collection of real positive numbers as in Lemma \ref{cluster-flat-lemma}.

Set
\begin{equation} \label{set-U-asymptotic}
    U^{\wr, cluster, (n)}_{\mathcal{P}} =\bigsqcup_{\mathcal{P}' \prec \mathcal{P}} \underline{M}_{G,\mathcal{P}}(\mathcal{P}', \rho_{\mathcal{P}'}/n,   \eta_{\mathcal{P}'}/n ).
\end{equation}

Conditions (\ref{covering-boundary-condition-nuclear}), (\ref{not-overlapping-condition-nuclear}) and  (\ref{disjoint-disks-nuclear2}) imply the following conditions:  
\begin{equation} \label{covering-boundary-condition-asymptotic}
  \underline{M}_{G,\mathcal{P}}^{ <  \rho^{\dagger}/n }  \subset  \underline{M}_{G, \mathcal{P} }(\rho/n) \sqcup   U^{\wr, cluster, (n)}_{\mathcal{P}},
\end{equation} 

   \begin{equation} \label{not-overlapping-condition-asymptotic}
     \underline{M}_{G, {\mathcal{P}} }(\mathcal{P}_1, \rho_{\mathcal{P}_1},   \eta_{\mathcal{P}_1} ) \cap  \underline{M}_{G, {\mathcal{P}} }(\mathcal{P}_2, \rho_{ \mathcal{P}_2},   \eta_{\mathcal{P}_2} ) 
   = \emptyset
    \end{equation}   
when $\mathcal{P}_1 \prec \mathcal{P}, \mathcal{P}_2 \prec \mathcal{P}, \mathcal{P}_1 \npreceq \mathcal{P}_2 , \mathcal{P}_2 \npreceq \mathcal{P}_1$.

 \begin{equation} \label{disjoint-disks-asymptotic}
  \mathbf{z} \in Cut_{\mathcal{P}^{\triangleright}, \eta_{\mathcal{P}}} ( Diag_{\mathcal{P}}), v_1 \sim_{\mathcal{P}^{\triangleright}} v_2 ,   v_1 \nsim_{\mathcal{P}} v_2 \Longrightarrow  {D}_{\rho_{ \mathcal{P}}, z_{ v_1}} \cap {D}_{\rho_{\mathcal{P}}, z_{ v_2}} = \emptyset .
\end{equation}
for each $\mathcal{P}^{\triangleright}$ with $\mathcal{P} \prec \mathcal{P}^{\triangleright}$.

\begin{proposition}   \label{existence-asymptotic-proposition}
For each flat $PS$-cycle $Z^{\flat}$ and frame $\mathbf{fr}$ of $M$,
there exists a  flat-approaching $PS$-cycle $Z^{\wr}$ converging  to $(Z^{\flat }, \mathbf{fr})$.
\end{proposition}

\begin{proof}

Assume we have defined $Z_{\prec G}^{\wr}$. We want to show that there exists $Z_{\preccurlyeq G}^{\wr}$ extending $Z_{\prec G}^{\wr}$ .




Assume  we have defined   $ \underline{Z}^{\wr, (n)}_{G,  m }  $ for $|m| <l$. Take $m$ with $|m|=l$.

Take $\rho^{\dagger}$ big enough such that the support property of radius $\rho^{\dagger}$ is satisfied for  $\underline{{Z}}^{\wr} $ for $(G',m'))$ when $(G',m') \prec (G,m)$ and for  $\underline{{Z}}^{\flat} $  for $(G',m')$ when $(G',m') \preccurlyeq (G,m)$.

We want to define $(\underline{Z}^{\wr , (n)}_{G,m, \mathcal{P}}, \rho_{\mathcal{P}},   \eta_{\mathcal{P}})_{\mathcal{P}}$ such that

\begin{itemize}

\item  transversality condition, convergence condition and cycle condition hold for   $\underline{Z}^{\wr}_{G,m, \mathcal{P}}$;




\item Conditions (\ref{covering-boundary-condition-nuclear}), (\ref{not-overlapping-condition-nuclear}) and  (\ref{disjoint-disks-nuclear2}) are satisfied;

\item the support condition of radius $\rho^{\dagger}$ is satisfied for  $\underline{Z}^{\wr}$ for  $(G,m)$ ;

\item $\underline{Z}^{\flat}$  satisfies the cluster condition for
\begin{equation}  \label{matching-preview-asymptotic-flat}
    (G',m', \mathcal{P},  \rho_{ \mathcal{P}}, \eta_{ \mathcal{P}})
    \end{equation}
when $(G',m') \preccurlyeq (G,m)$;

\item $\underline{Z}^{\wr}$  satisfies the cluster condition for
\begin{equation}  \label{matching-preview-asymptotic}
    (G',m', \mathcal{P},  \rho_{ \mathcal{P}}, \eta_{\mathcal{P}})
    \end{equation}
when $(G',m') \preccurlyeq (G,m)$.




\end{itemize}

Let $\mathcal{P}$ be a partition of $G$ and assume we have defined $(\underline{Z}^{\wr, (n)}_{G,m, \mathcal{P}'}, \rho_{\mathcal{P}'},   \eta_{\mathcal{P}'})_{\mathcal{P}' \prec \mathcal{P}}$.


Take $\rho_{\mathcal{P}}$ big enough such that (\ref{covering-boundary-condition-nuclear}) is satisfied,
then take $\eta_{\mathcal{P}}$ big enough such that condition (\ref{disjoint-disks-nuclear2}) , (\ref{not-overlapping-condition-nuclear}) and (\ref{matching-inductive-lemma-flat}) are satisfied. 
We  also assume that (\ref{matching-preview-asymptotic}) is satisfied for $(G',m') \prec (G,m)$. 


Let $\underline{Z}^{  \wr, cluster, (n)}_{G,m ,\mathcal{P}} $ be the chain on $U^{cluster, (n)}_{\mathcal{P}}$ such that
\begin{equation} \label{definition-cluster-asimptotic}
 \underline{Z}_{G,  m  , \mathcal{P}}^{\wr, cluster, (n)}  =  \underline{Z}_{G,  m  , \mathcal{P}'}^{\wr , (n)} 
 \end{equation}
on $ \underline{M}_{G, {\mathcal{P}} }(\mathcal{P}', \rho_{G,m, \mathcal{P}'},   \eta_{G,m,\mathcal{P}'} )     $, for each $\mathcal{P}' \prec \mathcal{P}$.
Conditions (\ref{matching-preview-asymptotic})
 and  (\ref{not-overlapping-condition-asymptotic}) imply that conditions (\ref{definition-cluster-asimptotic}) are compatible for different $\mathcal{P}'$.



$ \underline{Z}_{G,  m  , \mathcal{P}}^{ cluster, (n)}$ satisfies the transversality condition.

For each $\rho >0$,
for $n \rightarrow \infty$, the set $Dil_n(\underline{D}_{\rho/n,\mathbf{z}} \cap U^{\wr, cluster, (n)}_{\mathcal{P}})$ converges in an obvious sense to the set $\underline{D}_{\rho,\mathbf{0}}^{\flat} \cap U^{\flat, cluster}_{\mathcal{P}}$, where $U^{\flat, cluster}_{\mathcal{P}}$ is the open set defined in (\ref{set-U-nuclear}).

It is easy to show that the convergence condition for $\mathcal{P}' \prec \mathcal{P}$ yields 
\begin{equation} \label{convergence-cluster}
 Dil_n (\text{exp}^{-1}_{\mathbf{z}} (\underline{D}_{\rho/n,\mathbf{z}} \cap \underline{Z}_{G,m , \mathcal{P}}^{\wr, cluster, (n)}  )) \rightarrow \underline{D}_{\rho,\mathbf{0}}^{\flat} \cap  \underline{Z}^{\flat , cluster}_{G,m, \mathcal{P}}   
 \end{equation}
uniformly for $\mathbf{z}  \in  Diag_{\mathcal{P}}$.


The inductive convergence condition yields
\begin{equation} \label{convergence-boundary}
Dil_n (\text{exp}^{-1}_{\mathbf{z}} (  \underline{D}_{\rho/n,\mathbf{z}} \cap (  \sum_i (-1)^i \underline{Z}_{G, \partial_i m , \mathcal{P}}^{\wr (n)} + \sum_{e'}  \delta_{e'} \underline{Z}_{G',m', \mathcal{P} }^{\wr,  (n)})) ) \rightarrow \underline{D}_{\rho,\mathbf{0}}^{\flat} \cap \partial  \underline{Z}^{\flat }_{G,m, \mathcal{P}}  .
 \end{equation}
uniformly for $\mathbf{z}  \in  Diag_{\mathcal{P}}$.

From (\ref{matching-preview-asymptotic}) and the inductive cycle condition we have  
\begin{equation} \label{boundary-cluster}
\partial \underline{Z}_{G,  m  , \mathcal{P}}^{\wr, cluster, (n)} = \sum_i (-1)^i \underline{Z}_{G, \partial_i m , \mathcal{P}}^{\wr,  (n)} + \sum_{e'}  \delta_{e'} \underline{Z}_{G',m', \mathcal{P} }^{\wr, (n)} 
\end{equation}
on $U^{\wr, cluster, (n)}_{\mathcal{P}}$.

For each $\mathbf{z} \in Diag_{\mathcal{P}}$, let   $\underline{Z}_{G,  m , \mathcal{P} }^{prel,\mathbf{z},(n)}$ be the chain  on $ \underline{D}_{\rho_{\mathcal{P}}/n,\mathbf{z}}$ defined by 
$$  \underline{Z}_{G,  m , \mathcal{P} }^{prel,\mathbf{z},(n)}= \text{exp}_{\mathbf{z}}(Dil_{1/n} (   \underline{D}_{\rho_{\mathcal{P}},\mathbf{0}}^{\flat} \cap  \underline{Z}^{\flat }_{G,m, \mathcal{P}}  ) ).$$

 

We claim that there exists a chain $\underline{Z}_{G,  m , \mathcal{P} }^{\wr, (n)}$ such that
\begin{itemize}
 \item  \begin{equation} \label{approx-existence}
 \underline{D}_{\rho_{\mathcal{P}}/n,\mathbf{z}} \cap \underline{Z}_{G,  m , \mathcal{P} }^{\wr, (n)} \approx \underline{Z}_{G,  m , \mathcal{P} }^{prel,\mathbf{z},(n)} 
 \end{equation}
where the error goes to zero faster than $1/n$ in the $C^k$-norm for each $k$;
\item   the cycle condition, the cluster condition, the transversality condition are satisfied.
\end{itemize}

 
  We observe that, in the case that $M$ is flat, the translation invariance condition on $Z^{\flat}$ implies  that the chains $\underline{Z}_{G,  m , \mathcal{P} }^{prel,\mathbf{z},(n)}$ matches on the overlapping regions  for different $\mathbf{z}$ . For not flat manifolds this is not anymore true. However the error  goes  to zero faster than $1/n$ in the $C^k$-norm for each $k$;

 Condition  (\ref{convergence-cluster}) yields
$$  \underline{Z}_{G,  m , \mathcal{P} }^{prel, \mathbf{z}, (n)} \approx \underline{Z}_{G,  m  , \mathcal{P}}^{\wr, cluster, (n)}   $$
 on $U^{\wr, cluster, (n)}_{\mathcal{P}}$
 and  condition (\ref{convergence-boundary}) yields
$$\partial \underline{Z}_{G,  m , \mathcal{P} }^{prel, \mathbf{z}, (n)} \approx \sum_i (-1)^i \underline{Z}_{G, \partial_i m , \mathcal{P}}^{\wr, (n)} + \sum_{e'}  \delta_{e'} \underline{Z}_{G',m', \mathcal{P} }^{\wr,  (n)}    .$$
In both relations the error goest to zero faster than $1/n$ in the $C^k$ norm for each $k$;

It follows that exists $\underline{Z}_{G,  m , \mathcal{P} }^{\wr, (n)}$ satisfying (\ref{approx-existence}) such that
\begin{itemize}
\item \begin{equation} \label{cluster-asi-construction}
\underline{Z}_{G,  m , \mathcal{P} }^{\wr, (n)} = \underline{Z}_{G,  m  , \mathcal{P}}^{\wr, cluster, (n)}  
\end{equation}
on $U^{\wr, cluster, (n)}_{\mathcal{P}}$,
\item $$ \partial \underline{Z}_{G,  m , \mathcal{P} }^{\wr, (n)} = \sum_i (-1)^i \underline{Z}_{G, \partial_i m , \mathcal{P}}^{\wr (n)} + \sum_{e'}  \delta_{e'} \underline{Z}_{G',m', \mathcal{P} }^{\wr,  (n)}   ,$$
\item the convergence condition is satisfied.
\end{itemize}

The cluster condition follows from (\ref{cluster-asi-construction}) as in Lemma \ref{cluster-flat-lemma}.


\end{proof}

\subsubsection{Isotopies}

\begin{proposition} \label{unicity-asymptotic-proposition}







Let  $\tilde{Z}^{\flat}$ be an isotopy of flat $PS$-cycles between  $Z^{\flat, 0}$ and $Z^{\flat, 1}$.   Let $\tilde{\mathbf{fr}}$  be an isotopy of frames between  $\mathbf{fr}^0$ and $\mathbf{fr}^1$.   Let   $Z^{\wr,0}, Z^{\wr,1}$  be flat-approaching $PS$-cycles converging  to $(Z^{\flat,0 }, \mathbf{fr}^0), (Z^{\flat, 1 },\mathbf{fr}^1)$ respectively.
 
There exists an isotopy of flat-approaching $PS$-cycles $\tilde{Z}^{\wr}$ between $Z^{\wr,0}$ and  $Z^{\wr,1}$ converging to  $(\tilde{Z}^{\flat},  \tilde{\mathbf{fr}}  )$.

\end{proposition}

\begin{proof}

Assume we have defined $\tilde{Z}_{\prec G}^{\wr}$. We want to show that there exists $\tilde{Z}_{\preccurlyeq G}^{\wr}$ extending $\tilde{Z}_{\prec G}^{\wr}$ .




Assume  we have defined   $ \underline{\tilde{Z}}^{(n)}_{G,  m }  $ for $|m| <l$. Take $m$ with $|m|=l$.

Take $\rho^{\dagger}$ big enough such that the support property of radius $\rho^{\dagger}$ is satisfied for  $\underline{\tilde{Z}}^{\wr} $ for $(G',m'))$ when $(G',m') \prec (G,m)$ and for  $\underline{\tilde{Z}}^{\flat} $  for $(G',m')$ when $(G',m') \preccurlyeq (G,m)$.


We want to define $(\underline{\tilde{Z}}^{(n)}_{G,m, \mathcal{P}}, \rho_{\mathcal{P}},   \eta_{\mathcal{P}})_{\mathcal{P}}$ such that

\begin{itemize}

\item  transversality condition, convergence condition and cycle condition are satisfied for   $\underline{Z}^{\wr}_{G,m, \mathcal{P}}$;




\item Conditions (\ref{covering-boundary-condition-nuclear}), (\ref{not-overlapping-condition-nuclear}) and  (\ref{disjoint-disks-nuclear2}) are satisfied;

\item the support condition of radius $\rho^{\dagger}$ is satisfied for  $\underline{\tilde{Z}}^{\wr}$ for $(G,m)$;

\item $\underline{\tilde{Z}}^{\wr}_{G,m, \mathcal{P}}$ and $\underline{\tilde{Z}}^{\flat}_{G,m, \mathcal{P}}$ satisfies the cluster condition for
\begin{equation}  \label{matching-preview-asymptotic-iso}
    (G',m', \mathcal{P},  \rho_{ \mathcal{P}}, \eta_{ \mathcal{P}})
    \end{equation}
when $(G',m') \preccurlyeq (G,m)$ ;


\item $Z^{\flat, 0}, Z^{\wr, 0}, Z^{\flat, 1 }, Z^{\wr, 1} $ satisfies the cluster condition for
\begin{equation}  \label{matching-inductive-lemma-asi-flat-start-end}
    (G,m, \mathcal{P}, \rho_{\mathcal{P}}, \eta_{\mathcal{P}}).
\end{equation}

\end{itemize}

Let $\mathcal{P}$ be a partition of $G$ and assume we have defined $(\underline{\tilde{Z}}^{\wr, (n)}_{G,m, \mathcal{P}'}, \rho_{\mathcal{P}'},   \eta_{\mathcal{P}'})_{\mathcal{P}' \prec \mathcal{P}}$.

Define $\rho_{\mathcal{P}}$ big enough such that (\ref{covering-boundary-condition-nuclear}) is satisfied,
then take $\eta_{\mathcal{P}}$ big enough such that condition (\ref{disjoint-disks-nuclear2}) , (\ref{not-overlapping-condition-nuclear}) ,  (\ref{matching-inductive-lemma-flat}), (\ref{matching-inductive-lemma-asi-flat-start-end}) are satisfied. 
We  also assume that condition (\ref{matching-preview-asymptotic-iso}) is satisfied for $(G',m') \prec (G,m)$.

Let $\underline{\tilde{Z}}^{  \wr, cluster, (n)}_{G,m ,\mathcal{P}} $ be the chain on $U^{cluster, (n)}_{\mathcal{P}} \times \R$ such that
\begin{equation} \label{definition-cluster-asi-iso}
 \underline{\tilde{Z}}_{G,  m  , \mathcal{P}}^{\wr, cluster, (n)}  =  \underline{\tilde{Z}}_{G,  m  , \mathcal{P}'}^{(n)} 
 \end{equation}
on $ \underline{M}_{G, {\mathcal{P}} }(\mathcal{P}', \rho_{G,m, \mathcal{P}'},   \eta_{G,m,\mathcal{P}'} )  \times \R    $, for each $\mathcal{P}' \prec \mathcal{P}$. Conditions (\ref{matching-preview-asymptotic-iso}) and  (\ref{not-overlapping-condition-asymptotic}) imply that conditions (\ref{definition-cluster-asi-iso}) are compatible for different $\mathcal{P}'$.

$ \underline{\tilde{Z}}_{G,  m  , \mathcal{P}}^{ cluster, (n)}$ satisfies the transversality condition.



As in ( \ref{convergence-cluster}), the convergence condition for $\mathcal{P}' \prec \mathcal{P}$ yields
\begin{equation} \label{convergence-cluster-iso}
 Dil_n  ( \text{exp}^{-1}_{\mathbf{z}} ((D_{\rho/n,\mathbf{z}} \times \R) \cap \tilde{Z}_{G,m , \mathcal{P}}^{\wr, cluster, (n)}  )) \rightarrow (D_{\rho,\mathbf{0}}^{\flat} \times \R)  \cap  \underline{\tilde{Z}}^{\flat , cluster}_{G,m, \mathcal{P}}   
 \end{equation}
uniformly for $\mathbf{z}  \in  Diag_{\mathcal{P}}$.


The inductive convergence condition yields
\begin{equation} \label{convergence-boundary-iso}
Dil_n   \text{exp}^{-1}_{\mathbf{z}} ((  \underline{D}_{\rho/n,\mathbf{z}} \times \R) \cap (\sum_i (-1)^i \underline{\tilde{Z}}_{G, \partial_i m , \mathcal{P}}^{\wr (n)} + \sum_{e'}  \delta_{e'} \underline{\tilde{Z}}_{G',m', \mathcal{P} }^{\wr,  (n)}))  \rightarrow (\underline{D}_{\rho,\mathbf{0}}^{\flat} \times \R) \cap \partial  \underline{\tilde{Z}}^{\flat }_{G,m, \mathcal{P}}  .
 \end{equation}
uniformly for $\mathbf{z}  \in  Diag_{\mathcal{P}}$.

From (\ref{matching-preview-asymptotic-iso}) and the inductive cycle condition we have  
\begin{equation} \label{boundary-cluster-iso}
\partial \underline{\tilde{Z}}_{G,  m  , \mathcal{P}}^{\wr, cluster, (n)} = \sum_i (-1)^i \underline{\tilde{Z}}_{G, \partial_i m , \mathcal{P}}^{\wr,  (n)} + \sum_{e'}  \delta_{e'} \underline{\tilde{Z}}_{G',m', \mathcal{P} }^{\wr, (n)} 
\end{equation}
on $U^{\wr, cluster, (n)}_{\mathcal{P}} \times \R$.

From  (\ref{matching-inductive-lemma-asi-flat-start-end}) we have 
\begin{equation}  \label{start-end-cluster}
{ {\underline{\tilde{Z}}}}_{G,m , \mathcal{P}  }^{\wr, cluster,(n), < -\lambda^{\dagger}} = \underline{Z}_{G,m, \mathcal{P}  }^{\wr, 0, (n)} \times \R_{<-\lambda^{\dagger}}, \quad {\underline{\tilde{Z}}}_{G,m , \mathcal{P} }^{\wr, cluster, (n), > \lambda^{\dagger}} =  \underline{Z}_{G,m }^{\wr,1 , (n)} \times \R_{>\lambda^{\dagger}}
\end{equation}
on $U^{\wr, cluster, (n)}_{\mathcal{P}} \times \R_{<-\lambda^{\dagger}}, U^{\wr, cluster, (n)}_{\mathcal{P}} \times \R_{>\lambda^{\dagger}}$,
for $\lambda^{\dagger} \gg 0$ and $n \gg 0$.



As in the proof of Proposition  \ref{existence-asymptotic-proposition}, using
( \ref{convergence-cluster-iso}) ,  ( \ref{convergence-boundary-iso}), (\ref{boundary-cluster-iso}) and ( \ref{start-end-cluster}), starting from the approximative solution
$$  (\underline{D}_{\rho_{\mathcal{P}},\mathbf{z}} \times \R) \cap \underline{\tilde{Z}}_{G,  m , \mathcal{P} }^{prel,\mathbf{z},(n)} \approx \text{exp}_{\mathbf{z}}(Dil_{1/n} (  ( \underline{D}_{\rho_{\mathcal{P}},\mathbf{0}}^{\flat} \times \R)  \cap  \underline{\tilde{Z}}^{\flat }_{G,m, \mathcal{P}}  )) ,$$ 
 we may construct a chain $\underline{\tilde{Z}}_{G,  m , \mathcal{P} }^{\wr, (n)}$ such that
 \begin{itemize}
\item  $ \underline{\tilde{Z}}_{G,  m  , \mathcal{P}}^{\wr, (n)}  =  \underline{\tilde{Z}}_{G,  m  , \mathcal{P}}^{\wr, cluster, (n)} $ on $U^{\wr, cluster, (n)}_{\mathcal{P}}$,
\item $$ {\underline{\tilde{Z}}}_{G,m , \mathcal{P}  }^{\wr,(n), < -T} = \underline{Z}_{G,m, \mathcal{P}  }^{\wr, 0, (n)} \times \R_{<-\lambda^{\dagger}}, \quad \underline{\tilde{Z}}_{G,m , \mathcal{P} }^{\wr, (n), > T} =  \underline{Z}_{G,m }^{\wr,1 , (n)} \times \R_{>\lambda^{\dagger}}$$
for $T \gg 0$ and $n \gg 0$,
\item the transversality condition, the cycle condition and the convergence condition are satisfied for $(G,  m  , \mathcal{P})$  .
\end{itemize}


\end{proof}

\section{Reduced $PSH$}
 

We now consider a suitable quotient of $PSH$ which has  a natural structure of module over the ring of formal power series $\Q[[\hbar, a]]$, where $\hbar$ and $a$ are formal variables.

Consider decorated graphs with two extra parameters  $\chi_{\vartriangle} \in \Z_{\leq 0} , d_{\vartriangle} \in \Z_{\geq 0}$. Hence we consider arrays
$$  ( \chi_{\vartriangle} , d_{\vartriangle} ,  Comp_0 , (V_c, d_c , g_c)_c , (H_v )_v  , E ) .$$
 Using these graphs  we define $PS$-cycles analogously as before. 

On the set of the graphs we consider the equivalence relation generated by
 \begin{multline} \label{equivalence-zero-component}
       ( \chi_{\vartriangle} , d_{\vartriangle} ,  Comp_0 , (V_c, d_c , g_c)_c , (H_v )_v  , E ) \sim  \\   
( \chi_{\vartriangle} + \chi_{c'} , d_{\vartriangle} + d_{c'} ,  Comp_0 \setminus c' , (V_c,d_c, g_c)_c , (H_v )_v  , E ).
 \end{multline}
if $c' \in Comp_0$ is a component with $V_{c'} = \emptyset$,

Let $\mathcal{Z}^{null}$ be the set of $PS$-cycles $Z$ satisfying the following property: for any equivalence class of graphs (\ref{equivalence-zero-component}) $T$ and $m$ a marking of $T$ we have
$$ \sum_{[G] = T} Z_{G,m}=0.$$

Set
$$PSH = PSH/ \mathcal{Z}^{null}. $$



The ring of formal power series $\Q[[\hbar, a]]$ acts on $PSH$:
  the action of  $\hbar^{k_0} a^{d_0}$ on $PSH$ is yielded by the action on $\mathfrak{G}$ given by 
$$ \chi_{\vartriangle} (\hbar^{k_0} a^{d_0} G) = \chi_{\vartriangle} (G) - k_0,  d_{\vartriangle} (\hbar^{k_0} a^{d_0} G) = d_{\vartriangle}(G) + d_0 .$$

\subsection{Computation}

In the context of reduced $PSH$ we need to consider the obstruction defined in subsection \ref{obstruction-section}  up to elements of $\mathcal{Z}^{null}$. This means that  we are interested to 
$$ \sum_{[G] = T} \mathbf{obs}_{G}$$
for any equivalence class of graphs (\ref{equivalence-zero-component}) $T$.

\begin{proposition} \label{one-dimension}
 $PSH(M|  \underline{Z}^{\circledcirc}  )$ is a rank one free module over $\Q[[\hbar, a]]$.
\end{proposition}

\begin{proof}

Denote by $G^{\heartsuit}$ the trivial decorated graph with $Comp_0= \emptyset$.

Fix a nice $PS$-cycle $Z^{\heartsuit} \in \mathcal{Z}$ with  
\begin{equation} \label{generator-prop2}
\underline{Z}^{\heartsuit}_{G^{\heartsuit}}= 1.
\end{equation}
The existence of $\underline{Z}^{\heartsuit}$ may be seen applying recursively Lemma \ref{extension-nice}.

We claim that the map
$$    \Q[[\hbar, a]] \rightarrow  PSH(M|  \underline{Z}^{\circledcirc}  )$$
\begin{equation} \label{iso-map}
  q(\hbar, a) \mapsto  q(\hbar, a)  Z^{\heartsuit} 
\end{equation}
is an isomorphism.

We first  prove that  (\ref{iso-map}) is injective. Assume that there exists an isotopy $\tilde Z$ between $q(\hbar,a) Z^{\heartsuit}$ and a null $PS$-cycle $Z^{null}$.

Let  $(\kappa_0, d_0)$ be the leading term appearing in the formal power series  $q$. So   
$ Z_{\prec \hbar^{\kappa_0} a^{d_0} G^{\heartsuit}} =0$. 


For each $ G \in  \mathfrak{G}^{\star}  $ set
$$\tilde{\underline{Z}}_G^{\star} = \sum_{G'} \pi_*( \tilde{\underline{Z}}_{G'} )$$
where the sum run over the graphs $G' \in \mathfrak{G}$ such that $\delta_{V^{\star}(G')} G' =G$ and $\pi$ denotes the obvious  projection $\mathfrak{L}_{\hat{V}(G')}(M) \rightarrow \mathfrak{L}_{\hat{V}(G)}(M)$.


Using the cycle condition   we have 
\begin{equation} \label{cycle-star}
  \partial  \tilde{\underline{Z}}_{G}^{\star}  =    \sum_{(G', e')}  \delta_{e'} \tilde{\underline{Z}}_{G'}^{\star} . 
  \end{equation}
The terms in $\delta_{v'}$ are cancelled in the sum.


In particular we have $\partial \tilde{\underline{Z}}^{\star}_{ \hbar^{\kappa_0} a^{d_0} G^{\heartsuit}}= 0$ and hence 
$$  \underline{Z}_{ \hbar^{\kappa_0} a^{d_0} G^{\heartsuit}} = \partial  pr_* (\tilde{\underline{Z}}^{\star}_{ \hbar^{\kappa_0} a^{d_0} G^{\heartsuit}})  .$$
Since $ \underline{Z}_{ \hbar^{\kappa_0} a^{d_0} G^{\heartsuit}} \in \Q$ is different from zero by definition, we obtain a contraction.

Now we prove that the map (\ref{iso-map}) is surjective. We need to show that for each $Z \in \mathcal{Z}$  there exists a formal power series  $q(\hbar,a) \in \Q[[\hbar ,a]]$ such that $Z$ and $q(\hbar, a)  Z^{\heartsuit}$ are isotopic up a null $PS$-cycle.  We shall construct the formal power series $q$ using an inductive argument on the graphs.


Assume we have defined $q$  and $\tilde{\underline{Z}}_{\prec  G}^{(n)}  \in \tilde{\mathcal{Z}}_{\prec G} $ isotopy between   $ \underline{Z}_{\prec  G}^{(n)}$ and $ (q(\hbar,a) \underline{Z}^{\heartsuit,(n)})_{\prec G} $.

The obstruction to the existence of $\tilde{\underline{Z}}_{\preccurlyeq G} $ are concentrated on the graphs $G$ with  $H(G) = \emptyset$, $V(G) = \emptyset$. Let  $r \in \Q$ such that $\mathfrak{Obs}_G = r $. The replacement  
   $$ q  \leadsto q + r a^{d(G)}  \hbar^{\kappa(G)} $$ 
makes the obstruction to vanish.


   

Finally, an argument similar to the one used in the proof of  the injectivity above yields the independence of  $q$ by $n$, at least for $n$ big enough.

\end{proof}


\begin{thebibliography}{10}


\bibitem{AS1}  
S. Axelrod and I. M. Singer, \emph{Chern-Simons perturbation theory}, Proceedings of the XXth International Conference on Differential Geometric Methods in Theoretical Physics, Vol. 1, 2 (New York, 1991), World Sci. Publ., River Edge, NJ, 1992, pp. 3-45. MR 1225107

\bibitem{AS2}   S. Axelrod and I. M. Singer,  \emph{Chern-Simons perturbation theory. II}, J. Differential Geom. 39 (1994), no. 1, 173-213. MR 1258919















\bibitem{CS} V. Iacovino, \emph{Master Equation and Perturbative Chern-Simons theory}, arXiv:0811.2181.














\bibitem{MCH} V. Iacovino,  \emph{Open Gromov-Witten theory without Obstruction}, arXiv:1711.05302.













\bibitem{OGW-QME1} V. Iacovino,  \emph{Quantum Master Equation and Open Gromov-Witten theory}, arXiv:2412.04230.

\bibitem{OGW-QME2}  V. Iacovino,  \emph{Quantum Master Equation and Open Gromov-Witten theory 2},  arXiv:2412.06450.




\bibitem{PSPCS2} V. Iacovino,  \emph{Point-Splitting Perturbative Chern-Simons II}.


\bibitem{PSPCS3} V. Iacovino,  \emph{Point-Splitting Perturbative Chern-Simons III}, in preparation. 












\end{thebibliography}
\end{document}